\documentclass[11pt, reqno]{amsart}

\usepackage[top=3.75cm, bottom=3cm, left=3cm, right=3cm]{geometry}
\usepackage{amsmath}
\usepackage{amsthm}
\usepackage{amsfonts}
\usepackage{amssymb}
\usepackage{bm}
\usepackage{hyperref}
\usepackage{color}
\usepackage[all,cmtip]{xy}
\usepackage{enumitem}
\usepackage{tikz}  
\usepackage{graphicx}
\usepackage{scalerel}
\usepackage{todonotes}
\usepackage{comment}

\numberwithin{equation}{section}

\theoremstyle{plain}
\newtheorem{theorem}{Theorem}[section]
\newtheorem{lemma}[theorem]{Lemma}
\newtheorem{proposition}[theorem]{Proposition}
\newtheorem{corollary}[theorem]{Corollary}

\theoremstyle{definition}
\newtheorem{definition}[theorem]{Definition}

\newtheorem{remark}[theorem]{Remark}

\setlist[itemize]{leftmargin=*, itemsep={2pt}}
\setlist[enumerate]{leftmargin=*, itemsep={2pt}} 

\newcommand{\st}{\mid} 
\newcommand{\set}[1]{\left\{ \, #1 \, \right\}}

\newcommand{\Db}{\mathrm{D^b}}
\newcommand{\llangle}{\left \langle}
\newcommand{\rrangle}{\right \rangle}
\DeclareMathOperator{\Forg}{Forg}

\newcommand{\Cl}{{\mathop{\mathcal{C}\!\ell}}}

\newcommand{\Gr}{\mathrm{Gr}}
\DeclareMathOperator{\Sym}{Sym}
\DeclareMathOperator{\Proj}{Proj}
\DeclareMathOperator{\Pic}{Pic}

\newcommand{\tX}{\widetilde{X}}
\newcommand{\wtilde}{\widetilde}

\DeclareMathOperator{\Cone}{Cone}

\newcommand{\tM}{\widetilde{M}}

\newcommand{\Spec}{\mathrm{Spec}}
\newcommand{\pug}{\mathrm{pug}}

\newcommand{\cHom}{\mathcal{H}\!{\it om}}
\DeclareMathOperator{\Hom}{Hom}
\DeclareMathOperator{\Ext}{Ext}

\DeclareMathOperator{\Aut}{Aut}

\newcommand{\cUv}{\mathcal{U}^{\vee}}
\newcommand{\hcS}{\widehat{\mathcal{S}}}
\newcommand{\cSv}{\cS^{\vee}}
\newcommand{\tcK}{\wtilde{\mathcal{K}}}
\newcommand{\vV}{V^{\vee}}

\newcommand{\Coh}{\mathrm{Coh}}

\newcommand{\Ku}{\mathcal{K}u}

\DeclareMathOperator{\characteristic}{char}

\DeclareMathOperator{\Stab}{Stab}

\newcommand{\tH}{\widetilde{\rH}}

\DeclareMathOperator{\Br}{Br}

\DeclareMathOperator{\CH}{CH}

\newcommand{\usigma}{\underline{\sigma}}

\newcommand{\dP}{\mathrm{dP}}

\newcommand{\ord}{\mathrm{ord}}

\newcommand{\ch}{\mathrm{ch}}
\newcommand{\rk}{\mathrm{rk}}
\newcommand{\td}{\mathrm{td}}

\newcommand{\cO}{\mathcal{O}}
\newcommand{\cA}{\mathcal{A}}

\newcommand{\cD}{\mathcal{D}}
\newcommand{\cE}{\mathcal{E}}
\newcommand{\cF}{\mathcal{F}}
\newcommand{\cG}{\mathcal{G}}

\newcommand{\cK}{\mathcal{K}}
\newcommand{\cL}{\mathcal{L}}
\newcommand{\cM}{\mathcal{M}}
\newcommand{\cN}{\mathcal{N}}
\newcommand{\cP}{\mathcal{P}}
\newcommand{\cQ}{\mathcal{Q}}
\newcommand{\cR}{\mathcal{R}}
\newcommand{\cS}{\mathcal{S}}
\newcommand{\cT}{\mathcal{T}}
\newcommand{\cU}{\mathcal{U}}
\newcommand{\cV}{\mathcal{V}}
\newcommand{\cW}{\mathcal{W}}
\newcommand{\cX}{\mathcal{X}}
\newcommand{\cY}{\mathcal{Y}}

\newcommand{\ccP}{\mathcal{P}}

\newcommand{\rH}{\mathrm{H}}
\newcommand{\rK}{\mathrm{K}}

\newcommand{\rL}{\mathrm{L}}

\newcommand{\rR}{\mathrm{R}}

\newcommand{\bC}{\mathbf{C}}

\newcommand{\bZ}{\mathbf{Z}}
\newcommand{\bP}{\mathbf{P}}
\newcommand{\bQ}{\mathbf{Q}}
\newcommand{\bR}{\mathbf{R}}

\newcommand{\sA}{\mathsf{A}}

\begin{document}

\title[Stability conditions and moduli for Kuznetsov components of GM varieties]{Stability conditions and moduli spaces for Kuznetsov components of Gushel--Mukai varieties}

\author{Alexander Perry}
\address{Department of Mathematics, University of Michigan, Ann Arbor, MI 48109, USA \smallskip}
\email{arper@umich.edu}
\urladdr{http://www-personal.umich.edu/~arper/}

\author{Laura Pertusi} 
\address{Dipartimento di Matematica ``F.\ Enriques'' \\
Via Cesare Saldini 50 \\
Universit\`a degli Studi di Milano \\
20133 Milano, Italy \smallskip 
}
\email{laura.pertusi@unimi.it}
\urladdr{http://www.mat.unimi.it/users/pertusi/index.html}

\author{Xiaolei Zhao}
\address{Department of Mathematics \\
South Hall, Room 6607 \\
University of California \\
Santa Barbara, CA 93106, USA \smallskip
}
\email{xlzhao@math.ucsb.edu}
\urladdr{https://sites.google.com/site/xiaoleizhaoswebsite/}

\thanks{A.P. was partially supported by NSF grant DMS-1902060 and the Institute for Advanced Study. L.P. was supported by the ERC Consolidator Grant ERC-2017-CoG-771507-StabCondEn. X.Z. was partially supported by the Simons Collaborative Grant 636187.}


\begin{abstract}
We prove the existence of Bridgeland stability conditions on the Kuznetsov components of Gushel--Mukai varieties, 
and describe the structure of moduli spaces of Bridgeland semistable objects in these categories in the even-dimensional case. 
As applications, we construct a new infinite series of unirational locally complete families of polarized hyperk\"{a}hler varieties of K3 type, and characterize Hodge-theoretically when the Kuznetsov component of an even-dimensional Gushel--Mukai variety is equivalent to the derived category of a K3 surface. 
\end{abstract}

\maketitle


\section{Introduction}
\label{section-intro} 

The purpose of this paper is to prove the existence of stability conditions on a certain family of noncommutative K3 surfaces, and to use this for applications to hyperk\"{a}hler geometry, Hodge theory, and the structure of the derived categories of an associated family of Fano varieties.   

\subsection{Background on cubic fourfolds} 
The prototype for this line of thought is the example of cubic fourfolds, i.e. smooth cubic hypersurfaces in the projective space $\bP^5$. 
For such a fourfold $Y \subset \bP^5$, Kuznetsov \cite{kuznetsov-cubic} defined a subcategory $\Ku(Y) \subset \Db(Y)$ of the bounded derived category of coherent sheaves --- now known as the \emph{Kuznetsov component} of $Y$ --- by the semiorthogonal decomposition 
\begin{equation*}
\Db(Y) = \llangle \Ku(Y), \cO_Y, \cO_Y(1), \cO_Y(2) \rrangle. 
\end{equation*} 
Kuznetsov showed the category $\Ku(Y)$ can be thought of as a noncommutative K3 surface, 
in the sense that it has the same homological invariants (Serre functor and Hochschild homology) as the derived category of a K3 surface. 
Moreover, he proved that for several families of rational cubic fourfolds $\Ku(Y)$ is equivalent to the derived category of a K3 surface, and conjectured that this condition is equivalent to the rationality of $Y$. 

Stability conditions on triangulated categories were introduced by Bridgeland \cite{bridgeland}, 
and have been extremely influential due to their applications in algebraic geometry via moduli spaces and wall-crossing. 
In general, it is a very difficult problem to construct stability conditions on a given triangulated category, 
but recently Bayer, Lahoz, Macr\`{i}, and Stellari \cite{BLMS} made a breakthrough by solving this problem for the 
Kuznetsov component $\Ku(Y)$ of a cubic fourfold (as well as Kuznetsov components of many Fano threefolds). 
Motivated by the desire to study moduli spaces of stable objects in $\Ku(Y)$ by deforming to 
the case of the derived category of a K3 surface, in \cite{stability-families} a general theory of stability conditions in families was developed and 
used to show that moduli spaces of semistable objects in $\Ku(Y)$ 
are smooth projective hyperk\"{a}hler varieties of the expected dimension. 

These moduli spaces have many applications. 
For instance, they give rise to unirational locally complete families of polarized hyperk\"{a}hler varieties of arbitrarily large dimension and degree. 
Before this only several constructions of polarized hyperk\"{a}hler varieties were known, the main ones being 
moduli spaces of stable objects in the derived category of a K3 surface (which do not give locally complete families), and a few celebrated examples constructed via the classical geometry of cubic fourfolds \cite{beauville-donagi, LLSVS, LSV}. 
In \cite{LPZ, LPZ2} the latter examples are shown to also arise via moduli spaces of stable objects in $\Ku(Y)$; the category $\Ku(Y)$
thus unifies the above two constructions and can be thought of as a primordial source of hyperk\"{a}hler varieties.  
Moduli spaces of objects in $\Ku(Y)$ were also used in \cite{stability-families} to give a Hodge-theoretic characterization of when $\Ku(Y)$ is equivalent to the derived category of a K3 surface, extending a result by Addington and Thomas \cite{addington-thomas}. 

\subsection{Kuznetsov components of Gushel--Mukai varieties} 
Given the success of the example of cubic fourfolds, it is natural to look for other situations where 
the above program can be carried out. 
The main properties of the categories $\Ku(Y)$ needed for this can be abstracted as follows. 
We would like a family of smooth projective varieties, such that for each member $X$ in the family 
there is a semiorthogonal component $\Ku(X) \subset \Db(X)$ with the following properties: 
\begin{enumerate}
\item \label{category-condition} 
$\Ku(X)$ is a noncommutative K3 surface, which is equivalent to the derived category of a K3 surface for special $X$ but not equivalent to such a category for general $X$. 
\item \label{stability-condition} 
$\Ku(X)$ admits a Bridgeland stability condition. 
\end{enumerate}
Several examples of varieties $X$ with a  
noncommutative K3 surface $\Ku(X) \subset \Db(X)$ are described in \cite{kuznetsov-CY}, 
but besides the case of cubic fourfolds, so far only one class of examples is known to satisfy condition~\eqref{category-condition}: 
those coming from Gushel--Mukai varieties. 

\begin{definition}
\label{definition-GM} 
A \emph{Gushel--Mukai (GM) variety} is a smooth $n$-dimensional intersection 
\begin{equation*}
X = \Cone(\Gr(2,5)) \cap Q, \quad 2 \leq n \leq 6 , 
\end{equation*} 
where $\Cone(\Gr(2,5)) \subset \bP^{10}$ is the projective cone over the Pl\"{u}cker embedded Grassmannian 
$\Gr(2,5) \subset \bP^{9}$ and $Q \subset \bP^{10}$ is a quadric hypersurface in a linear subspace $\bP^{n+4} \subset \bP^{10}$. 
\end{definition} 

The classification results of Gushel \cite{gushel} and Mukai \cite{mukai} show that (in characteristic $0$) these varieties coincide with the class of all smooth Fano varieties of Picard number $1$, coindex $3$, and degree $10$ (corresponding to $n \geq 3$), together with the Brill--Noether general polarized K3 surfaces of degree $10$ (corresponding to $n = 2$). Recently, GM varieties have attracted attention because of the rich structure of their birational geometry, Hodge theory, and derived categories \cite{DIM3fold, IM-EPW, DIM4fold, DebKuz:birGM, DebKuz:moduli, DebKuz:periodGM, KuzPerry:dercatGM}. 

In particular, in \cite{KuzPerry:dercatGM} for any GM variety a Kuznetsov component $\Ku(X) \subset \Db(X)$ is defined by the semiorthogonal decomposition 
\begin{equation}
\label{equation-Ku}
\Db(X) = \llangle \Ku(X), \cO_X, \cU_X^{\vee}, \dots, \cO_X(\dim(X)-3), \cU_X^{\vee}(\dim(X)-3) \rrangle  , 
\end{equation} 
where $\cU_X$ and $\cO_X(1)$ denote the pullbacks to $X$ of the rank $2$ tautological subbundle and Pl\"{u}cker line bundle on $\Gr(2,5)$. 
The category $\Ku(X)$ is a noncommutative K3 or Enriques surface according to whether $\dim(X)$ is even or odd. 
Moreover, the results of \cite{KuzPerry:dercatGM, KuzPerry:cones} show that the Kuznetsov components of GM fourfolds and sixfolds satisfy property~\eqref{category-condition} above. 

\subsection{Results} 
We work over an algebraically closed field $k$ of characteristic $0$. 
Our first main result verifies property~\eqref{stability-condition} above. 

\begin{theorem}
\label{main-theorem}
If $X$ is a GM variety, 
then the category $\Ku(X)$ has a Bridgeland stability condition. 
\end{theorem} 

The crucial case of this theorem, and our main contribution, is when $\dim(X) = 4$. 
Indeed, Bridgeland's work~\cite{bridgeland-K3} gives the $\dim(X) = 2$ case, since 
then $\Ku(X) = \Db(X)$ and $X$ is a K3 surface; 
the $\dim(X) = 3$ case is proved in \cite{BLMS}; and by the duality conjecture for GM varieties proved in \cite{KuzPerry:cones}, the $\dim(X) = 5$ and $\dim(X) = 6$ cases reduce respectively to the $\dim(X) = 3$ and $\dim(X) = 4$ cases. 
Our proof when $\dim(X) = 4$ is inspired by the case of cubic fourfolds treated in \cite{BLMS}, and involves as the starting point an embedding of $\Ku(X)$ as a semiorthogonal component in the derived category of sheaves of modules over a Clifford algebra on a quadric threefold, associated to a conic fibration of $X$. 
We note that, due to the more complicated geometric setup, 
there are a number of additional difficulties that arise in the case of GM fourfolds versus the case of cubic fourfolds. 

As a first consequence of Theorem~\ref{main-theorem}, we deduce the following result. 

\begin{corollary}
\label{corollary-DbX-stability}
If $X$ is a GM variety, 
then the category $\Db(X)$ has a Bridgeland stability condition. 
\end{corollary} 

When $\Ku(X)$ is a noncommutative K3 surface, or equivalently when $X$ is even-dimensional, 
we show that Theorem~\ref{main-theorem} has many applications parallel to those in \cite{stability-families} for cubic fourfolds. 
One particularly interesting output is that moduli spaces of stable objects in $\Ku(X)$ give rise to a new infinite series of unirational locally complete families of hyperk\"{a}hler varieties of K3 type (Theorem~\ref{theorem-unirational-families}); 
this gives after \cite{stability-families} the second known infinite series of such families, 
which have long been sought-after. 
For this and the other results enumerated below, we assume $k = \bC$ is the complex numbers, as we use Hodge theory. 

\subsubsection{The stability manifold} 
\label{intro-stab-manifold}
Following Addington and Thomas \cite{addington-thomas} (see \cite{Pert} and the review in Section~\ref{section-Mukai-HS}), 
for any even-dimensional GM variety $X$ there is an associated weight $2$ Hodge structure $\tH(\Ku(X), \bZ)$,
which is equipped with a natural pairing $(-,-)$ and 
agrees with the usual Mukai Hodge structure $\tH(S, \bZ)$ when $\Ku(X) \simeq \Db(S)$ for a K3 surface $S$. 
Let $\Stab(\Ku(X))$ denote the space of full numerical stability conditions on $\Ku(X)$, i.e. the space of stability conditions which satisfy the support property with respect to the lattice of integral Hodge classes $\tH^{1,1}(\Ku(X), \bZ)$. 
The central charge of any $\sigma \in \Stab(\Ku(X))$ is given by pairing with an element 
in the complexification $\tH^{1,1}(\Ku(X), \bC)$ of the lattice of Hodge classes; this association 
gives a map $\eta \colon \Stab(\Ku(X)) \to \tH^{1,1}(\Ku(X), \bC)$. 
Following \cite{bridgeland-K3}, we define $\ccP(\Ku(X)) \subset \tH^{1,1}(\Ku(X), \bC)$ as the open subset consisting of vectors whose real and imaginary parts span positive-definite two-planes, and set 
\begin{equation}
\label{cP0}
\ccP_0(\Ku(X)) = \ccP(\Ku(X)) \setminus \bigcup_{\delta \in \Delta} \delta^{\perp} \quad \text{where} \quad 
\Delta = \set{ \delta \in \tH^{1,1}(\Ku(X), \bZ) \st (\delta, \delta) = -2}. 
\end{equation} 
We prove the following analog for the noncommutative K3 surfaces $\Ku(X)$ 
of a result of Bridgeland \cite{bridgeland-K3} for K3 surfaces. 

\begin{theorem}
\label{theorem-Stab-dagger}
Let $X$ be a GM fourfold or sixfold. 
The stability conditions on $\Ku(X)$ constructed in the proof of Theorem~\ref{main-theorem} are 
full numerical stability conditions. 
Moreover, the connected component $\Stab^{\dagger}(\Ku(X)) \subset \Stab(\Ku(X))$ containing these stability conditions is mapped by $\eta$ onto a connected component $\ccP_0^+(\Ku(X))$ of $\ccP_0(\Ku(X))$, and the induced map $\Stab^{\dagger}(\Ku(X)) \to \ccP_0^+(\Ku(X))$ is a covering map. 
\end{theorem} 

\subsubsection{Moduli spaces} 
\label{intro-moduli}
Our next result establishes the basic properties of moduli spaces of $\sigma$-semistable objects in the setting of Theorem~\ref{theorem-Stab-dagger}. 

\begin{theorem}
\label{theorem-moduli-space}
Let $X$ be a GM fourfold or sixfold. 
Let $v \in \tH^{1,1}(\Ku(X), \bZ)$ be a nonzero primitive vector and let 
$\sigma \in \Stab^{\dagger}(\Ku(X))$ be a stability condition generic with respect to $v$. 
\begin{enumerate}
\item \label{theorem-moduli-space-nonempty}
There is a coarse moduli space $M_{\sigma}(\Ku(X), v)$ parameterizing $\sigma$-semistable objects in $\Ku(X)$ of class $v$, which is nonempty if and only if $(v,v) \geq -2$. 
\item \label{theorem-moduli-space-dim}
When nonempty, $M_{\sigma}(\Ku(X), v)$ is a smooth projective hyperk\"{a}hler variety of dimension $(v,v) + 2$, deformation equivalent to the Hilbert scheme of points on a K3 surface. 
\item \label{theorem-moduli-space-H2}
If $(v,v) \geq 0$, then there is a natural Hodge isometry 
\begin{equation*}
\rH^2(M_{\sigma}(\Ku(X), v), \bZ) \xrightarrow{\, \sim \,} 
\begin{cases}
v^{\perp}  & \text{if } (v,v) > 0 \\ 
v^{\perp}/\bZ v & \text{if } (v,v) = 0, 
\end{cases}
\end{equation*} 
where the orthogonal is taken inside $\tH(\Ku(X), \bZ)$. 
\end{enumerate}
\end{theorem}

\begin{remark}
In \cite{stability-families}, the cubic fourfold version of the nonemptiness result 
Theorem~\ref{theorem-moduli-space}\eqref{theorem-moduli-space-nonempty}
was used to reprove the integral Hodge conjecture for cubic fourfolds. 
Theorem~\ref{theorem-moduli-space}\eqref{theorem-moduli-space-nonempty} 
similarly implies the integral Hodge conjecture for GM fourfolds; 
however, as demonstrated in \cite{IHC-CY2}, this result already follows from a more basic ``Mukai theorem'' on smoothness of moduli spaces of objects in CY2 categories, which does not require the use of stability conditions and is in fact an important ingredient in the proof of Theorem~\ref{theorem-moduli-space}. 
\end{remark}

The proof of Theorem~\ref{theorem-moduli-space} is based on the 
construction of relative moduli spaces of stable objects in the categories $\Ku(X)$ for a family of GM fourfolds, together with the fact that moduli spaces of objects are well-understood in the case where $\Ku(X) \simeq \Db(S)$ for a K3 surface $S$ 
(due to results of many authors which culminated in Bayer and Macr\`{i}'s work \cite{bayer-macri-projectivity}). 
We refer to Theorem~\ref{theorem-relative-moduli} for the general statement 
on relative moduli spaces, and record here one of its consequences. 

\begin{theorem}
\label{theorem-unirational-families}
For any pair $(a,b)$ of coprime integers, there is a unirational locally complete family, over an open subset of the moduli space of GM fourfolds, of smooth polarized hyperk\"{a}hler varieties of dimension $2(a^2+b^2+1)$, degree $2(a^2+b^2)$, and divisibility $a^2+b^2$.  
\end{theorem} 

\begin{remark}
\label{remark-moduli-cubic-comparison}
Theorem~\ref{theorem-unirational-families} gives the second known 
construction of an infinite series of unirational, locally complete families of hyperk\"{a}hler varieties. 
Note that the dimensions and degrees are indeed different than the ones arising from cubic fourfolds  in \cite[Corollary 29.5]{stability-families}. 
\end{remark}

The moduli spaces from Theorem~\ref{theorem-moduli-space} should be useful for further applications, some of which we discuss in Section~\ref{subsection_examples}. 
We expect that in low dimensions these moduli spaces  
are isomorphic to classically constructed hyperk\"{a}hler varieties. 
Namely, we conjecture the minimal dimension $4$ hyperk\"{a}hler varieties from Theorem~\ref{theorem-unirational-families} ($a^2+b^2=1$) recover O'Grady's famous double EPW sextics~\cite{OG-EPW}, and the $6$-dimensional ones ($a^2+b^2=2$) recover the recently constructed EPW cubes~\cite{IKKR}. 
In fact, in Proposition~\ref{proposition-EPW} we prove this conjecture for double EPW sextics in the very general case; we plan to remove the very general assumption in future work. 
We also observe in Proposition~\ref{prop_existinvolution} that the hyperk\"{a}hler varieties in Theorem~\ref{theorem-unirational-families} are always equipped with an antisymplectic involution, which in the $4$-dimensional case corresponds to the canonical involution of a double EPW sextic. 

Beyond dimension $6$, the hyperk\"{a}hler varieties in Theorem~\ref{theorem-unirational-families} appear to be completely new. 
The $12$-dimensional case is already quite interesting: we expect that any GM fourfold admits a (possibly only rationally defined) embedding into such a hyperk\"{a}hler. 
This would give an analog of the embedding of a cubic fourfold into the associated Lehn--Lehn--Sorger--van Straten hyperk\"{a}hler $8$-fold \cite{LLSVS}, 
and may lead to a solution to the open problem of determining the image of the period map for GM fourfolds, along the lines of the recent new proof by Bayer and Mongardi of the corresponding result for cubic fourfolds (see \cite[Proposition B.12]{Deb:survey}). 

Finally, we note that our results for moduli spaces of objects in $\Ku(X)$ are confined to the case of even-dimensional GM varieties. 
The odd-dimensional case is more subtle, because then the category $\Ku(X)$ is \emph{never} equivalent to the derived category of a variety, so we cannot directly reduce the problem to a more geometric one. 
However, using the relation between Kuznetsov components of even- and odd-dimensional GM varieties from \cite{cyclic-covers}, in future work we plan to use Theorem~\ref{theorem-moduli-space} to analyze moduli spaces in the odd-dimensional case.

\subsubsection{Associated K3 surfaces} 
\label{intro-associated-K3}
If $X$ is an even-dimensional GM variety, we say that $X$ has a \emph{homological associated K3 surface} if there exist a projective K3 surface $S$ and an equivalence
$\Ku(X) \simeq \Db(S)$. 
The importance of this property stems from the GM analog of Kuznetsov's rationality conjecture, which predicts a GM fourfold is rational if and only if it has a homological associated K3 surface. 
The following result, based on Theorem~\ref{theorem-moduli-space}, gives a Hodge-theoretic characterization of this property. 

\begin{theorem}
\label{theorem-Ku-geometric} 
Let $X$ be a GM fourfold or sixfold. 
Then $X$ has a homological associated K3 surface if and only if the lattice $\tH^{1,1}(\Ku(X), \bZ)$ contains a hyperbolic plane. 
More generally, there exists an equivalence $\Ku(X) \simeq \Db(S, \alpha)$ 
for a projective K3 surface $S$ with a Brauer class $\alpha \in \Br(S)$ 
if and only if there exists a nonzero primitive vector $v \in \tH^{1,1}(\Ku(X), \bZ)$ such that $(v,v) = 0$. 
\end{theorem} 

\begin{remark}
There is also a notion of a Hodge-theoretically associated K3 surface. 
For simplicity, we explain the situation for a GM fourfold $X$. 
There is an embedding of $\rH^4(\Gr(2,5), \bZ)$ into $\rH^4(X, \bZ)$ 
whose image is a rank $2$ sublattice denoted by $L_{\Gr}$.  
We say a polarized K3 surface $(S,f)$ is \emph{Hodge-theoretically associated} to $X$ if 
there is a rank $3$ saturated sublattice $L_{\Gr} \subset L \subset \rH^4(X, \bZ)$ such that there is a Hodge isometry 
\begin{equation}
\label{HT-associated-K3} 
\rH^2(S, \bZ) \supset f^{\perp} \xrightarrow{\, \sim \,} L^{\perp}(1) \subset \rH^4(X, \bZ)(1)
\end{equation} 
where $(1)$ denotes a Tate twist. 
Debarre, Iliev, and Manivel \cite{DIM4fold} showed the existence of a Hodge-theoretic associated K3 cuts out a countable union of Noether--Lefschetz divisors in the moduli space of GM fourfolds. 
The results of \cite{Pert} combined with Theorem~\ref{theorem-Ku-geometric} show that if $X$ has a Hodge-theoretic associated K3 surface then it also has a homological associated K3 surface, but the converse does \emph{not} hold. 
There are thus two competing conjectural conditions 
(existence of a homological versus Hodge-theoretic K3) encoding the rationality of a GM fourfold. 
This should be contrasted with the case of cubic fourfolds, where these conditions are known to be equivalent \cite{addington-thomas, stability-families} (see Remark \ref{rmk-HTK3versusHomologicalK3}). 
\end{remark} 

Following~\cite{addington-thomas}, we deduce from Theorem~\ref{theorem-Ku-geometric} the algebraicity of the Hodge isometries~\eqref{HT-associated-K3} defining Hodge-theoretic associated K3s.  
Slightly more generally, we show the following. 
For simplicity we state the result for GM fourfolds, but there is an obvious analog for GM sixfolds. 

\begin{corollary}
\label{corollary-hodge-isom}
Let 
$S$ be a projective K3 surface and $X$ a GM fourfold. 
Let $K \subset \rH^{1,1}(S, \bZ)$ and $L_{\Gr} \subset L \subset \rH^{2,2}(X, \bZ)$ be sublattices such that there is a Hodge isometry $\varphi \colon K^{\perp} \xrightarrow{\sim} L^{\perp}(1)$. 
Then there is an algebraic cycle in $\CH^3(S \times X)$ which induces $\varphi$. 
\end{corollary} 

\subsection{Organization of the paper} 
In Section~\ref{section-conic-fibration} we show that an ordinary Gushel-Mukai fourfold $X$ is birational to a conic fibration over a quadric threefold $Y$. Then, under an appropriate genericity assumption on $X$ (smoothness of a canonically associated quadric surface), we construct in Theorem~\ref{theorem-KX-Cl} a fully faithful embedding of $\Ku(X)$ into the bounded derived category $\Db(Y, \Cl_0)$ of coherent sheaves of modules over the even part $\Cl_0$ of the sheaf of Clifford algebras arising from the conic fibration. 

Section~\ref{section-BI} is devoted to proving a generalized Bogomolov inequality for slope semistable $\Cl_0$-modules, which is necessary for the construction of weak stability conditions on $\Db(Y, \Cl_0)$. 
We first prove the inequality for a smooth linear section of $Y$, and then deduce the statement for $Y$ arguing by induction on the rank of the sheaf.

In Section~\ref{section-stability} we prove Theorem \ref{main-theorem} and Corollary \ref{corollary-DbX-stability}, as well as the first part of Theorem~\ref{theorem-Stab-dagger} asserting that in the even-dimensional case the constructed stability conditions on $\Ku(X)$ are full numerical stability conditions. 
Making use of the previous sections, we first prove the results for ordinary GM fourfolds with smooth canonical quadric, and 
then use the duality conjecture for GM varieties to reduce the general case to this one. 

Section~\ref{section_applications} is devoted to the applications. In particular, we prove the results in Sections~\ref{intro-stab-manifold}, \ref{intro-moduli}, and \ref{intro-associated-K3} concerning the stability manifold of $\Ku(X)$, the hyperk\"ahler varieties arising as moduli spaces of stable objects in $\Ku(X)$, 
and the existence of associated K3 surfaces.
We end with a discussion of the low-dimensional hyperk\"{a}hler varieties from Theorem~\ref{theorem-unirational-families}. 

\subsection{Notation and conventions} 
\label{section-notation}
We work over an algebraically closed field $k$ of characteristic $0$. 
In fact, our proof of Theorem~\ref{main-theorem} should also go through in large enough positive characteristic. Namely, we need $\characteristic{k} > 2$ to apply Kuznetsov's results \cite{kuznetsov08quadrics} on derived categories of quadric fibrations (see \cite{quadrics-char2}), and we need $\characteristic{k}$ sufficiently large so that coherent cohomology computations involving $\Gr(2,5)$ work as in characteristic $0$; we leave the details to the interested reader.  
We assume $k = \bC$ in Sections~\ref{section-support-fourfold} and~\ref{section_applications} because there we use Hodge theory. 

A vector bundle $\cE$ on a variety $X$ is a finite locally free $\cO_X$-module. The projective bundle of such an $\cE$ is 
\begin{equation*}
\bP(\cE) = \Proj(\Sym^*(\cE^{\vee})) \xrightarrow{\, \pi \,} X
\end{equation*} 
with $\cO_{\bP(\cE)}(1)$ normalized so that $\pi_* \cO_{\bP(\cE)}(1) = \cE^{\vee}$. 

If $X \to Y$ is a morphism of varieties and $\cF \in \Db(Y)$, then 
we often write $\cF_X$ for the pullback of $\cF$ to $X$. 
By abuse of notation, if $D$ is a divisor on a variety $Y$, we often 
denote still by $D$ its pullback to any variety mapping to $Y$. 

As mentioned already, we use $\Db(X)$ to denote the bounded derived category of coherent sheaves on a variety $X$. 
By convention all functors are derived. 
In particular, for a morphism $f \colon X \to Y$ of varieties we write $f_*$ and $f^*$ for the derived pushforward and pullback functors, and for $E, F \in \Db(X)$ we write $E \otimes F$ for their derived tensor product. 
All functors between triangulated categories in this paper will be given by Fourier--Mukai functors. 
In particular, when we write $\Ku(X) \simeq \Db(S)$ for a K3 surface $S$, we mean that there is an equivalence given by a Fourier--Mukai kernel on the product $X \times S$. 
Finally, the term K3 surface means a smooth projective K3 surface. 

\subsection{Acknowledgements} 
We are very grateful to Arend Bayer, Daniel Huybrechts, Sasha Kuznetsov, Chunyi Li, Zhiyu Liu, Emanuele Macr\`{i}, and Paolo Stellari for many useful conversations about this work. 
We would especially like to thank Sasha Kuznetsov for his help in proving Theorem~\ref{theorem-KX-Cl}, and for many useful comments on an earlier version of this paper. 

Part of this paper was written when the second author was visiting the Max-Planck-Institut f\"ur Mathematik in Bonn, whose hospitality is gratefully acknowledged.


\section{Conic fibrations and Kuznetsov components of GM fourfolds}
\label{section-conic-fibration}

We fix a $5$-dimensional vector space $V_5$. 
Recall that a GM variety as in Definition~\ref{definition-GM} is called \emph{ordinary} if $\bP^{n+4}$ does not intersect the vertex of $\Cone(\Gr(2,V_5))$, and \emph{special} otherwise. 
We consider an ordinary GM fourfold $X$, which by projection from the vertex of $\Cone(\Gr(2,V_5))$ 
can be expressed as an intersection 
\begin{equation}
\label{GM4fold}
X = \Gr(2, V_5) \cap \bP(W) \cap Q 
\end{equation}
where $\bP(W) \subset \bP(\wedge^2V_5)$ is a hyperplane in 
the Pl\"{u}cker space $\bP(\wedge^2V_5)$ and $Q \subset \bP(W)$ is a quadric hypersurface.  
We denote by $H$ the hyperplane class on the Pl\"{u}cker space $\bP(\wedge^2V_5)$, 
and by $\cU$ the tautological rank $2$ subbundle on $\Gr(2,V_5)$.  
Recall that the Kuznetsov component $\Ku(X)$ 
of $X$ is defined by the semiorthogonal decomposition 
\begin{equation}
\label{dbX}
\Db(X) = \llangle \Ku(X), \cO_X, \cUv_X, \cO_X(H), \cUv_X(H) \rrangle. 
\end{equation} 

In this section, we start by showing that $X$ is birationally a conic fibration  
over a quadric threefold $Y$, by elaborating on a construction from \cite[\S3]{DIM4fold}. 
Associated to this conic fibration is a sheaf $\Cl_0$ of even parts of the corresponding 
Clifford algebra on $Y$. 
Under a suitable genericity hypothesis (smoothness of a canonical quadric surface associated to $X$ in Section~\ref{subsection-conicfibr}), 
our main result expresses $\Ku(X)$ as a semiorthogonal component in the derived 
category $\Db(Y, \Cl_0)$ of coherent sheaves of $\Cl_0$-modules. 
We prove a preliminary version of this result in Section~\ref{subsection-Ku-in-Cl}, 
and then in Section~\ref{section-explicit-Ku-in-Cl} prove a refined version as Theorem~\ref{theorem-KX-Cl}. 

\subsection{The conic fibration} 
\label{subsection-conicfibr}
The hyperplane $\bP(W) \subset \bP(\wedge^2 V_5)$ is defined by a 
skew-symmetric form $\omega \in \wedge^2V_5^{\vee}$, which by the 
smoothness of $X$ has a $1$-dimensional kernel $V_1 \subset V_5$. 
Consider the $3$-dimensional projective space 
\begin{equation*}
\bP^3_W = \set{ V_2 \in \Gr(2,V_5) \st V_1 \subset V_2 } 
\cong \bP(V_1 \wedge V_5) \subset \bP(\wedge^2V_5). 
\end{equation*} 
Note that we have an inclusion 
\begin{equation*}
\bP^3_W \subset \bP(W). 
\end{equation*} 
Linear projection from $\bP^3_W$ provides $X$ with a conic bundle 
structure. 
To state this result precisely, we first introduce some notation. 

Let 
\begin{equation}
\label{eq_T}
T = \bP_W^3 \cap Q \subset X 
\end{equation} 
be the quadric surface cut out by $\bP^3_W$ in $X$, 
called the \emph{canonical quadric} of $X$. 
Note that 
$T$ is indeed a surface as $X$ does not contain $3$-planes. 
As noted in \cite[\S3]{DIM4fold}, $T$ is the unique quadric surface contained in $X$ such that, 
when regarded as a subvariety of $\Gr(2,V_5)$, its points correspond to $2$-dimensional subspaces 
of $V_5$ that all contain a fixed $1$-dimensional subspace (namely $V_1$). 

Next, set 
\begin{equation*}
V_4 = V_5/V_1 
\end{equation*} 
and let 
\begin{equation*}
W' \subset \wedge^2 V_4 = \wedge^2V_5/(V_1 \wedge V_5) 
\end{equation*} 
be the hyperplane given as the image of $W \subset \wedge^2 V_5$.  
Let $Y \subset \bP(W') \cong \bP^4$ be the quadric threefold 
given by the smooth intersection 
\begin{equation*}
Y = \Gr(2,V_4) \cap \bP(W'). 
\end{equation*} 
Let $h$ denote the hyperplane class on $\bP(\wedge^2V_4)$. 
Let $\cS \subset V_4 \otimes \cO$ denote the tautological rank $2$ subbundle on $\Gr(2,V_4)$, 
and let $\hcS \subset V_5 \otimes \cO$ denote its preimage under the surjection $V_5 \otimes \cO \to V_4 \otimes \cO$. 
Note that a choice of splitting $V_5 \cong V_1 \oplus V_4$ induces an isomorphism 
\begin{equation*}
\hcS \cong (V_1 \otimes \cO) \oplus \cS. 
\end{equation*} 

\begin{lemma}
\label{lemma-conic-bundle}
Let $b \colon \tX \to X$ be the blowup with center $T \subset X$.  
Then linear projection from $\bP^3_W \subset \bP(W)$ induces a regular 
map $\pi \colon \tX \to Y$, which is a conic fibration. 
More precisely, consider the rank $3$ vector bundle  
\begin{equation*}
\cE = \wedge^2( \hcS_Y) \cong (V_1 \otimes \cS_Y) \oplus \cO(-h) 
\end{equation*} 
on $Y$. 
Then: 
\begin{enumerate}
\item \label{E-to-W} There is an inclusion $\cE \hookrightarrow W \otimes \cO$ which induces a 
morphism $\bP_Y(\cE) \to \bP(W)$. 
\item \label{O1E} There is an isomorphism $\cO_{\bP_Y(\cE)}(1) \cong \cO(H)$. 
\item \label{X-in-PE} 
$\tX$ is a divisor in $\bP_Y(\cE)$ cut out by a section of $\cO(2H)$, namely $\tX$ 
is the preimage of $Q \subset \bP(W)$ under the morphism $\bP_Y(\cE) \to \bP(W)$. 
\item \label{disc-of-cf} 
The discriminant locus of the conic fibration $\pi \colon \tX \to Y$ is a divisor in $Y$ 
defined by a section of $\cO(4h)$. 
\end{enumerate} 
This situation is summarized by the commutative diagram 
\begin{equation}
\label{main-diagram}
\vcenter{
\xymatrix{
& E \ar[r]^{i_E} \ar[dl]_{b_E} & \tX \ar[dl]_{b} \ar[dr]^{\pi} \ar[r]^-{j} & \bP_Y(\cE) \ar[d]^{p} \\ 
T \ar[r]^{i_T} & X & & Y 
}
} 
\end{equation} 
where $E$ is the exceptional divisor of $b \colon \tX \to X$. 
We have the following linear equivalences of divisors on $\tX$: 
\begin{equation}
\label{divisors-tX} 
h = H - E \quad \text{and} \quad K_{\tX} = -2H + E = -H - h. 
\end{equation}  
In particular, the restriction $\pi|_E: E \to Y$ is flat and finite of degree $2$. Finally, if $T$ is smooth, then so are all of the other varieties in the above diagram.
\end{lemma}

\begin{proof}
Claims~\eqref{E-to-W} and~\eqref{O1E} follow directly from the definitions. 

To see~\eqref{X-in-PE}, we first consider the linear projection from $\bP^3_W$ 
as a map from the entire Grassmannian $\Gr(2,V_5)$. 
This gives a diagram 
\begin{equation}
\label{Gr-blowup}
\vcenter{
\xymatrix{
& \wtilde{\Gr(2,V_5)} \ar[dl] \ar[dr] & \\ 
\Gr(2,V_5) && \Gr(2, V_4) 
}
}
\end{equation} 
where the first morphism is the blowup in $\bP^3_W$ and the second is 
induced by linear projection from $\bP^3_W \subset \bP(\wedge^2V_5)$. 
It is easy to see there is an isomorphism 
\begin{equation*}
\wtilde{\Gr(2,V_5)} \cong \Gr_{\Gr(2,V_4)}(2, \hcS)
\end{equation*} 
under which the first morphism is induced by the inclusion $\hcS \subset V_5 \otimes \cO$ 
and the second is the tautological projection. 
Passing to the intersection with $\bP(W)$, diagram~\eqref{Gr-blowup} induces 
an analogous diagram 
\begin{equation}
\label{GrPW-blowup}
\vcenter{
\xymatrix{
& \wtilde{\Gr(2,V_5) \cap \bP(W)}  \ar[dl] \ar[dr] & \\ 
\Gr(2,V_5) \cap \bP(W) && Y = \Gr(2, V_4) \cap \bP(W')
}
}
\end{equation} 
where the first morphism is the blowup in $\bP^3_W$ and the second is 
induced by linear projection from $\bP^3_W \subset \bP(W)$. 
Moreover, we have isomorphisms  
\begin{equation*}
\wtilde{\Gr(2,V_5) \cap \bP(W)} \cong \Gr_Y(2, \hcS_Y) \cong \bP_Y(\wedge^2 \hcS_Y) = \bP_Y(\cE), 
\end{equation*} 
under which the first morphism is induced by $\bP_Y(\cE) \to \bP(W)$ and the 
second is the tautological projection. 
The blowup $\tX \to X$ is given by the proper transform of $X \subset \Gr(2, V_5) \cap \bP(W)$ 
under the first morphism in~\eqref{GrPW-blowup}. 
Since $X \subset \Gr(2, V_5) \cap \bP(W)$ is the divisor defined by the quadric $Q \subset \bP(W)$ and 
$T = X \cap \bP^3_W$ is a surface, this coincides with the preimage of $Q \subset \bP(W)$ under 
$\bP_Y(\cE) \to \bP(W)$. 
This proves claim~\eqref{X-in-PE}. 

The section of $\cO_{\bP_Y(\cE)}(2H)$ defining $\tX$ corresponds to a morphism 
of vector bundles $\cE \to \cE^{\vee}$ on $Y$. 
The discriminant locus of the conic fibration $\pi \colon \tX \to Y$ is the vanishing 
locus of the determinant of this morphism, i.e. is defined by a section of 
\begin{equation*}
\det(\cE^{\vee})^2 \cong \cO(4h). 
\end{equation*}
This proves claim~\eqref{disc-of-cf}. 

The final claims about the equalities~\eqref{divisors-tX} and smoothness are straightforward. 
\end{proof} 

\subsection{The embedding of $\Ku(X)$ into $\Db(Y, \Cl_0)$} 
\label{subsection-Ku-in-Cl}
From now on, we assume the canonical quadric of $X$ is smooth. 

\begin{remark}
The smoothness of the canonical quadric of $X$ holds generically, namely 
on the complement of a divisor in moduli. This is a consequence of \cite[Lemma 2.1]{DebKuz:periodGM} and \cite[Proposition 4.5]{DebKuz:birGM}, as explained in the proof of Theorem \ref{theorem-duality}. 
\end{remark}

Recall from \cite{kuznetsov08quadrics} that associated to the conic fibration 
$\pi \colon \tX \to Y$, there are sheaves $\Cl_0$ and $\Cl_1$ of even and 
odd parts of the corresponding Clifford algebra on $Y$, which as sheaves of 
$\cO_{Y}$-modules are given by
\begin{align}
\label{Cl_0} 
\Cl_0  & = \cO_Y \oplus \wedge^2 \cE  \cong \cO_Y \oplus \cO_Y(-h) \oplus \cS_Y(-h)   \\
\label{Cl_1} 
\Cl_1  & = \cE \oplus \wedge^3 \cE \cong \cS_Y \oplus \cO_Y(-h) \oplus \cO_Y(-2h). 
\end{align}
Note that $\Cl_0$ is a sheaf of $\cO_Y$-algebras via Clifford multiplication, 
and $\Cl_1$ is a $\Cl_0$-module. Moreover, by Lemma \ref{lemma-conic-bundle}\eqref{X-in-PE}, $\wtilde{X}$ is the zero locus of a section in $\rH^0(\bP_Y(\cE),\cO(2H) \otimes \pi^* \cL^\vee)$, where $\cL \cong \cO_Y$. As a consequence, the natural $\Cl_0$-bimodules introduced in \cite[ (15)]{kuznetsov08quadrics} satisfy
\begin{equation}
\label{eq_Cl_i}
\Cl_{2i}:=\Cl_0 \otimes \cL^{-i}\cong \Cl_0 \quad \mbox{and} \quad \Cl_{2i+1}:=\Cl_1 \otimes \cL^{-i}\cong \Cl_1.  
\end{equation} 

Our goal is to realize $\Ku(X)$ as a semiorthogonal component inside the derived category 
$\Db(Y, \Cl_0)$ of coherent sheaves of right $\Cl_0$-modules. 
To state the result precisely, we need some preliminary notation. 

Again by \cite{kuznetsov08quadrics}, there is a fully faithful functor 
\begin{equation*}
\Phi \colon \Db(Y, \Cl_0) \to \Db(\tX), 
\end{equation*}
whose image fits into a semiorthogonal decomposition 
\begin{equation*}
\Db(\tX) = \llangle \Phi(\Db(Y, \Cl_0)), \pi^* \Db(Y) \rrangle
\end{equation*}
Since the quadric threefold $Y \subset \bP(W')$ is smooth by Lemma~\ref{lemma-conic-bundle}, the category 
$\Db(Y)$ admits a standard decomposition, which can be written as
\begin{equation*} 
\label{DbY}
\Db(Y) = \llangle \cO_Y(-h), \cO_Y, \cS^{\vee}_Y, \cO_Y(h) \rrangle.
\end{equation*} 
Here, we have used that $\cS_Y$ is the spinor bundle on $Y$, 
being the restriction of the spinor bundle $\cS$ on the quadric 
$\Gr(2, V_4) \subset \bP(\wedge^2V_4)$. 
All together, we obtain a semiorthogonal decomposition 
\begin{equation}
\label{sod-conic}
\Db(\tX) = \llangle \Phi(\Db(Y, \Cl_0)), \cO_{\tX}(-h), \cO_{\tX}, \cSv_{\tX}, \cO_{\tX}(h)  \rrangle . 
\end{equation}

On the other hand, by Orlov's blowup formula we have a semiorthogonal 
decomposition 
\begin{equation*}
\Db(\tX) = \llangle b^*\Db(X), i_{E*}b_{E}^* \Db(T) \rrangle , 
\end{equation*} 
where the morphisms $b, b_E,$ and $i_E$ are as in~\eqref{main-diagram}. 
Since $T \subset \bP^3_W$ is a smooth quadric surface, the category 
$\Db(T)$ admits a standard decomposition, which can be written as 
\begin{equation*}
\Db(T) = \llangle \cO_T(-c), \cO_T(-d), \cO_T, \cO_T(H) \rrangle, 
\end{equation*} 
where $c$ and $d$ are the semiample generators of $\text{Pic}(T)$ such that
\begin{equation*}
H|_T=c+d.
\end{equation*} 
Plugging this and~\eqref{dbX} into the blowup formula, we obtain 
\begin{multline}
\label{sod-blowup} 
\Db(\tX) = 
 \langle b^*\Ku(X), 
\cO_{\tX}, \cUv_{\tX}, \cO_{\tX}(H), \cUv_{\tX}(H),  \\ 
i_{E *}\cO_E(-c), i_{E *}\cO_E(-d), i_{E *}\cO_E, i_{E *}\cO_E(H)  \rangle . 
\end{multline} 

To realize $\Ku(X)$ as a semiorthogonal component inside of $\Db(Y, \Cl_0)$, 
we will find a sequence of mutations taking~\eqref{sod-blowup} into the form of~\eqref{sod-conic}. 
We freely use basic facts about mutation functors, as reviewed for instance in \cite[Section 2]{kuznetsov-cubic}. The precise result we will prove is the following. 

\begin{proposition} 
\label{proposition-KX-Cl} 
Let $X$ be an ordinary GM fourfold with smooth canonical quadric. 
Then there is a semiorthogonal decomposition 
\begin{equation*}
\Phi(\Db(Y, \Cl_0)) = 
\llangle \Ku(X)', \cU_{\tX}, \cF_c, \cF_d, \cG 
\rrangle 
\end{equation*} 
where $\cF_c$, $\cF_d$, and $\cG$ are rank 2 vector bundles on $\tX$ defined in Lemmas~\ref{lemma_defFaFb} and~\ref{lem_defG}, 
and $\Ku(X)'$ is the fully faithful image of $\Ku(X)$ under the functor 
\begin{equation*} 
\rL_{\cO_{\tX}(-h)} \circ b^* \circ \rL_{\cU_X} \colon \Db(X) \to \Db(\tX) . 
\end{equation*} 
\end{proposition}

The proof of Proposition~\ref{proposition-KX-Cl} is divided into steps. 

\medskip \noindent
\textbf{Step 1.} Mutate $\cUv_{\tX}(H)$ to the far left of the decomposition~\eqref{sod-blowup}. 
Since this is a mutation in $b^*\Db(X)$ and we have $K_X = -2H$ and $\cUv_{X}(-H) \cong \cU_{X}$, the result is 
\begin{equation*}
\Db(\tX) = 
\llangle \cU_{\tX}, b^*\Ku(X), 
\cO_{\tX}, \cUv_{\tX}, \cO_{\tX}(H), 
i_{E *}\cO_E(-c), i_{E *}\cO_E(-d), i_{E *}\cO_E, i_{E *}\cO_E(H)  \rrangle . 
\end{equation*} 

\medskip \noindent
\textbf{Step 2.} Left mutate the objects $i_{E *}\cO_E(-c), i_{E *}\cO_E(-d), i_{E *}\cO_E, i_{E *}\cO_E(H)$ through $\cO_{\tX}(H)$. 

\begin{lemma}
\label{lemma-mutate-thru-OH}
We have  
\begin{align*}
\rL_{\cO_{\tX}(H)}(i_{E *}\cO_E(-c)) & = i_{E *}\cO_E(-c), \\
\rL_{\cO_{\tX}(H)}(i_{E *}\cO_E(-d)) & = i_{E *}\cO_E(-d), \\
\rL_{\cO_{\tX}(H)}(i_{E *}\cO_E)  & = i_{E *}\cO_E, \\ 
\rL_{\cO_{\tX}(H)}(i_{E *}\cO_E(H)) & = \cO_{\tX}(h)[1].
\end{align*} 
\end{lemma}

\begin{proof}
For the first three equalities, we just need to check that there are no morphisms (in any degree) from $\cO_{\tX}(H)$ to the objects $i_{E *}\cO_E(-c), i_{E *}\cO_E(-d), i_{E *}\cO_E$. 
But for a divisor $\ell$ on $T$, we have 
\begin{equation*}
\Hom^{\bullet}(\cO_{\tX}(H), i_{E *}\cO_E(\ell)) = \rH^{\bullet}(\cO_{T}(\ell-H)), 
\end{equation*} 
which vanishes for $\ell = -c, -d, 0$. 
The final mutation follows from the exact triangle
\begin{equation*}
\cO_{\tX}(H) \to i_{E*} \cO_{E}(H) \to  \cO_{\tX}(h)[1], 
\end{equation*} 
obtained from the exact sequence 
\begin{equation*}
0 \to \cO_{\tX}(-E) \to \cO_{\tX} \to i_{E*} \cO_{E} \to 0 
\end{equation*} 
by twisting by $H$, using the equality $h = H-E$, and rotating. 
\end{proof} 

By the lemma, the result of the above mutation is 
\begin{equation*}
\Db(\tX) = 
\llangle \cU_{\tX}, b^*\Ku(X),
\cO_{\tX}, \cUv_{\tX}, 
i_{E *}\cO_E(-c), i_{E *}\cO_E(-d), i_{E *}\cO_E, \cO_{\tX}(h), \cO_{\tX}(H)  \rrangle . 
\end{equation*} 

\medskip \noindent
\textbf{Step 3.} Mutate $\cO_{\tX}(H)$ to the far left of the decomposition. 
Since $K_{\tX} = -H-h$ by Lemma~\ref{lemma-conic-bundle}, 
the result is 
\begin{equation*}
\Db(\tX) = 
\llangle 
\cO_{\tX}(-h), \cU_{\tX}, b^*\Ku(X),
\cO_{\tX}, \cUv_{\tX}, 
i_{E *}\cO_E(-c), i_{E *}\cO_E(-d), i_{E *}\cO_E, \cO_{\tX}(h) \rrangle . 
\end{equation*} 

\medskip \noindent
\textbf{Step 4.} Left mutate $b^*\Ku(X)$ through $\llangle \cO_{\tX}(-h), \cU_{\tX} \rrangle$. 
The result is 
\begin{equation*}
\Db(\tX) = 
\llangle 
\Ku(X)', \cO_{\tX}(-h),  \cU_{\tX}, 
\cO_{\tX}, \cUv_{\tX}, 
i_{E *}\cO_E(-c), i_{E *}\cO_E(-d), i_{E *}\cO_E, \cO_{\tX}(h) \rrangle . 
\end{equation*} 
where $\Ku(X)' = \rL_{\cO_{\tX}(-h)}b^*\rL_{\cU_X}(\Ku(X))$. 

\medskip \noindent 
\textbf{Step 5.} Right mutate $\cUv_{\tX}$ through the objects 
$i_{E *}\cO_E(-c), i_{E *}\cO_E(-d), i_{E *}\cO_E$. 

\begin{lemma}
We have $\rR_{\langle i_{E *}\cO_E(-c), i_{E *}\cO_E(-d) \rangle} \cUv_{\tX} = \cUv_{\tX}$. 
\end{lemma} 

\begin{proof}
The claim is equivalent to showing there are no morphisms (in any degree) 
from $\cUv_{\tX}$ to $i_{E *}\cO_E(-c)$ and $i_{E *}\cO_E(-d)$. 
This orthogonality follows from the split short exact sequence  
\begin{equation}
\label{cUT}
0 \to V_1 \otimes \cO_{T} \to \cU_{T}  \to \cO_{T}(-H) \to 0 
\end{equation} 
which can easily be deduced from the definition of $T$ (\cite[Lemma 3.7]{DebKuz:periodGM}). 
\end{proof} 

\begin{lemma}
\label{lemma-RUv}
There is an exact sequence 
\begin{equation*}
0 \to \cSv_{\tX} \to \cUv_{\tX} \to \vV_1 \otimes \cO_E \to 0 . 
\end{equation*} 
and an isomorphism $\rR_{i_{E *}\cO_E}(\cUv_{\tX}) \cong \cSv_{\tX}$.
\end{lemma}

\begin{proof}
By the exact sequence~\eqref{cUT} we have
\begin{equation*}
\Hom^\bullet(\cU_{\tX}^\vee, i_{E*}\cO_E)= \rH^\bullet(T,\cU_T)= V_1[0]. 
\end{equation*} 
Thus we get the triangle
$$\rR_{i_{E*}\cO_E}\cU_{\tX}^\vee \to \cU_{\tX}^\vee \to V_1^\vee \otimes i_{E*}\cO_E.$$
Consider the surjective morphism 
\begin{equation}
\label{cUvX-to-OT}
\cUv_X \to \vV_1 \otimes i_{T*}\cO_T 
\end{equation}
of sheaves on $X$ given as the composition of the canonical surjection 
$\cUv_X \to \cUv_X \otimes i_{T*}\cO_T$ with the pushforward along $i_T \colon T \to X$ 
of the canonical surjection $\cUv_T \to V_1^{\vee} \otimes \cO_T$. 
The morphism $\cUv_{\tX} \to \vV_1 \otimes \cO_E$ we are interested in is 
the pullback of~\eqref{cUvX-to-OT} along $b \colon \tX \to X$, and in particular is also surjective. 
We define sheaves $\cK$ and $\tcK$ on $X$ and $\tX$ by the exact sequences 
\begin{align}
\label{cK} 0 \to \cK \to \cUv_X \to \vV_1 \otimes \cO_T \to 0 , \\ 
\label{tcK} 0 \to \tcK \to \cUv_{\tX} \to \vV_1 \otimes \cO_E \to 0 . 
\end{align}
By the discussion above, we have 
\begin{equation*}
\rR_{i_{E*}\cO_E}\cU_{\tX}^\vee \cong \tcK,
\end{equation*} 
so we must show there is an isomorphism $\tcK \cong \cSv_{\tX}$. 

First we claim that $\tcK$ is a quotient of $\vV_4 \otimes \cO_{\tX}$. 
For this, by the above exact sequences it is enough to show the same for $\cK$.  
From~\eqref{cK} we find that 
$\rH^0(X, \cK) \cong \vV_4$, which corresponds to a morphism 
$\vV_4 \otimes \cO_X \to \cK$. 
Working fiberwise, it is straightforward to see that the sequence 
\begin{equation*}
\vV_4 \otimes \cO_X  \to \cUv_X \to \vV_1 \otimes \cO_T \to 0 
\end{equation*}  
is exact, i.e. $\vV_4 \otimes \cO_X \to \cK$ is surjective. 

By the defining exact sequence~\eqref{tcK}, the sheaf $\tcK$ has rank $2$ 
and is locally free since $E \subset \tX$ is a divisor. 
Hence by the above $\tcK$ induces a map $g \colon \tX \to \Gr(2,V_4)$ such that 
$\tcK \cong g^* \cSv$. 
To prove $\tcK \cong \cSv_{\tX}$, it therefore suffices to show that $g$ agrees with 
the map  
\begin{equation*}
\tX \to Y \hookrightarrow \Gr(2,V_4) .
\end{equation*}
But it is easy to see from the definitions that these maps agree on the complement of $E \subset \tX$, and hence 
agree everywhere. 
\end{proof}

By the lemmas, the result of the above mutations is 
\begin{equation*}
\Db(\tX) = 
\llangle 
\Ku(X)', \cO_{\tX}(-h),  \cU_{\tX}, 
\cO_{\tX}, 
i_{E *}\cO_E(-c), i_{E *}\cO_E(-d), i_{E *}\cO_E, \cSv_{\tX}, \cO_{\tX}(h) \rrangle . 
\end{equation*} 

\medskip \noindent
\textbf{Step 6.} Left mutate the objects $i_{E *}\cO_E(-c), i_{E *}\cO_E(-d), i_{E *}\cO_E$ through $\cO_{\tX}$. 
Arguing as in Lemma~\ref{lemma-mutate-thru-OH}, we find 
\begin{align*}
\rL_{\cO_{\tX}}(i_{E *}\cO_E(-c)) & = i_{E *}\cO_E(-c), \\ 
\rL_{\cO_{\tX}}(i_{E *}\cO_E(-d)) & = i_{E *}\cO_E(-d),  \\
\rL_{\cO_{\tX}}( i_{E *}\cO_E ) & = \cO_{\tX}(h-H)[1], 
\end{align*} 
so the result is 
\begin{equation}
\label{DbtX-Es}
\hspace{-2mm} \Db(\tX) \hspace{-1mm} = \hspace{-1mm} \llangle \hspace{-1mm} \Ku(X)' \!, \cO_{\tX}(-h), \cU_{\tX}, i_{E *}\cO_E(-c),  i_{E *}\cO_E(-d),  \cO_{\tX}(h-H),  \cO_{\tX}, \cS_{\tX}^\vee, \cO_{\tX}(h) \hspace{-1mm} \rrangle .
\end{equation}

\medskip \noindent
\textbf{Step 7.} Left mutate the objects $\cU_{\tX}, i_{E *}\cO_E(-c), i_{E *}\cO_E(-d), \cO_{\tX}(h-H)$ through $\cO_{\tX}(-h)$.

\begin{lemma}
\label{lemma_mutU}
We have $\rL_{\cO_{\tX}(-h)}(\cU_{\tX}) = \cU_{\tX}$.
\end{lemma}

\begin{proof}
We must show that $\Hom^{\bullet}(\cO_{\tX}(-h), \cU_{\tX}) = \rH^{\bullet}(\cU_{\tX}(h))$ vanishes. 
Note that there is an isomorphism $\cU_{\tX}(h) \cong \cUv_{\tX}(-E)$ since $h = H - E$ and $\cU_{\tX}(H) \cong \cUv_{\tX}$. 
From the resolution of $i_{E*} \cO_{E}$ on $\tX$ we obtain an exact sequence 
\begin{equation*}
0 \to \cUv_{\tX}(-E) \to \cUv_{\tX} \to i_{E*}\cUv_{E} \to 0 . 
\end{equation*} 
It follows from the exact sequence~\eqref{cUT} that the 
morphism on cohomology induced by the right arrow 
\begin{equation*}
V_5 \cong \rH^{\bullet}(\cUv_{\tX}) \to \rH^{\bullet}(i_{E*}\cUv_{E}) \cong \rH^{\bullet}(\cUv_T)\end{equation*} 
is an isomorphism, which implies the desired vanishing. 
\end{proof}

\begin{lemma}
\label{lemma_defFaFb}
We have $\rL_{\cO_{\tX}(-h)} i_{E *}\cO_E(-c) \cong \cF_c[1]$ and 
$\rL_{\cO_{\tX}(-h)} i_{E *}\cO_E(-d) \cong \cF_d[1]$, 
where $\cF_c$ and $\cF_d$ are rank $2$ vector bundles on $\tX$ 
defined by exact sequences 
\begin{equation}
\begin{aligned}
\label{eq_defFa}
&0 \to \cF_c \to \cO_{\tX}(-h)^{\oplus 2} \to i_{E *}\cO_E(-c) \to 0,  \\
&0 \to \cF_d \to \cO_{\tX}(-h)^{\oplus 2} \to i_{E *}\cO_E(-d) \to 0. 
\end{aligned}
\end{equation}
\end{lemma}

\begin{proof}
First we compute the pushforwards of $i_{E *}\cO_E(-c)$ and 
$i_{E *}\cO_E(-d)$ along $\pi \colon \tX \to Y$. 
Recall that the derived category of the quadric threefold $Y$ decomposes as 
\begin{equation*} 
\Db(Y) = \llangle \cO_Y(-h), \cO_Y, \cS^{\vee}_Y, \cO_Y(h) \rrangle 
\end{equation*} 
The decomposition \eqref{DbtX-Es} shows that $\pi_*i_{E*}\cO_E(-c)$ and $\pi_*i_{E*}\cO_E(-d)$ are right orthogonal to the objects $\cO_Y, \cS^{\vee}_Y, \cO_Y(h)$, and hence 
isomorphic to a sum of shifts of $\cO_Y(-h)$. 
On the other hand, as $\pi \circ i_E \colon E \to Y$ is flat and finite of degree $2$ by Lemma \ref{lemma-conic-bundle}, $\pi_*i_{E*}\cO_E(-c)$ and $\pi_*i_{E*}\cO_E(-d)$ are vector bundles of rank $2$. Together these observations prove that 
\begin{equation}
\label{piE-a}
\pi_*i_{E*}\cO_E(-c) \cong \cO_Y(-h)^{\oplus 2} \quad  \text{and} \quad \pi_*i_{E*}\cO_E(-d) \cong \cO_Y(-h)^{\oplus 2}. 
\end{equation} 

Thus we obtain 
\begin{equation*}
\Hom^\bullet(\cO_{\tX}(-h),i_{E*}\cO_E(-c)) \cong k^{2}[0] 
\cong 
\Hom^\bullet(\cO_{\tX}(-h),i_{E*}\cO_E(-d)), 
\end{equation*} 
and hence we have exact triangles 
\begin{equation*}
\begin{aligned}
&\cO_{\tX}(-h)^{\oplus 2} \to i_{E*}\cO_E(-c) \to \rL_{\cO_{\tX}(-h)}i_{E*}\cO_E(-c), \\
&\cO_{\tX}(-h)^{\oplus 2} \to i_{E*}\cO_E(-d) \to \rL_{\cO_{\tX}(-h)}i_{E*}\cO_E(-d).
\end{aligned}    
\end{equation*}
It follows from the finiteness of $\pi|_E \colon E \to Y$ (see Lemma~\ref{lemma-conic-bundle})
that the above morphisms $\cO_{\tX}(-h)^{\oplus 2} \to i_{E*}\cO_E(-c)$ and $\cO_{\tX}(-h)^{\oplus 2} \to i_{E*}\cO_E(-d)$ are surjective. This gives the sequences in \eqref{eq_defFa}. 
\end{proof}

\begin{lemma}
\label{lem_defG}
We have $\rL_{\cO_{\tX}(-h)}\cO_{\tX}(h-H) \cong \cG$, where $\cG$ is a rank $2$ vector bundle on $\tX$ defined by an exact sequence 
\begin{equation}
\label{G}
0 \to \cO_{\tX}(h-H) \to \cG \to \cO_{\tX}(-h) \to 0.   
\end{equation}
\end{lemma}

\begin{proof}
By the same argument as in Lemma \ref{lemma_defFaFb}, the object $\pi_*\cO_{\tX}(h-H)$ is a sum of shifts of $\cO_Y(-h)$. On the other hand, the restriction of $\cO_{\tX}(h-H)$ to the general fiber of $\pi$ is isomorphic to $\cO_{\bP^1}(-2)$. We deduce that
\begin{equation}
\label{eq_pushforwh-H}  
\pi_*\cO_{\tX}(h-H)\cong \cO_Y(-h)[-1].
\end{equation}

Thus we obtain 
\begin{equation*}
\Hom^\bullet(\cO_{\tX}(-h),\cO_{\tX}(h-H)) \cong k[-1] , 
\end{equation*} 
and hence there is an exact triangle 
\begin{equation*} 
\cO_{\tX}(-h)[-1] \to \cO_{\tX}(h-H) \to \rL_{\cO_{\tX}(-h)}\cO_{\tX}(h-H). 
\end{equation*}
It follows that $\rL_{\cO_{\tX}(-h)}\cO_{\tX}(h-H)$ is isomorphic to the unique nontrivial extension of $\cO_{\tX}(-h)$ by $\cO_{\tX}(h-H)$. 
\end{proof}

Combining the above lemmas, we see that the result of our mutations is 
\begin{equation}
\label{eq_finalstepdec}
\Db(\tX) = \llangle \Ku(X)', \cU_{\tX}, \cF_c, \cF_d, \cG, \cO_{\tX}(-h), \cO_{\tX}, \cS_{\tX}^\vee, \cO_{\tX}(h) \rrangle .
\end{equation}
Comparing this decomposition with~\eqref{sod-conic} completes the proof 
of Proposition~\ref{proposition-KX-Cl}. \qed

\begin{remark}
As pointed out by Kuznetsov, one can prove a similar statement for $X$ with singular canonical quadric $T$, up to replacing $\llangle \cF_c, \cF_d \rrangle$ with $\Db(\Cl_0^T)$, where $\Cl_0^T$ is the even part of the 
Clifford algebra over $k$ corresponding to $T$. 
\end{remark}

\subsection{Making explicit the embedding of $\Ku(X)$ into $\Db(Y, \Cl_0)$}
\label{section-explicit-Ku-in-Cl}
Proposition~\ref{proposition-KX-Cl} shows that $\Ku(X)$ embeds into $\Db(Y, \Cl_0)$ as the semiorthogonal complement of four exceptional objects, corresponding to $\cU_{\tX}$, $\cF_c$, $\cF_d$, and $\cG$. 
The goal of this section is to explicitly describe (a mutation of) these exceptional objects as $\Cl_0$-modules on $Y$. More precisely, we prove the following. 

\begin{theorem}
\label{theorem-KX-Cl} 
Let $X$ be an ordinary GM fourfold with smooth canonical quadric. 
Then there is a semiorthogonal decomposition 
\begin{equation*}
\Db(Y, \Cl_0) = 
\llangle \Psi(\Ku(X)), \Cl_1, \Cl_0, \cR_c, \cR_d \rrangle
\end{equation*} 
where $\Psi \colon \Ku(X) \to \Db(Y, \Cl_0)$ is a fully faithful functor 
defined in \eqref{Psi} and 
$\cR_c$ and $\cR_d$ are exceptional objects of $\Db(Y, \Cl_0)$ 
defined in \eqref{Ra-Rb} which as $\cO_Y$-modules are 
locally free of rank~$4$. 
\end{theorem} 

\begin{remark}
As observed in \eqref{eq_Cl_i}, we have $\Cl_{-1} \cong \Cl_1$, which explains why $\Cl_1$ sits to the left of $\Cl_0$ in the above decomposition. Moreover, $\Cl_0$ and $\Cl_1$ are completely orthogonal, so we can also write the decomposition as
\begin{equation*}
\Db(Y, \Cl_0) = 
\llangle \Psi(\Ku(X)), \Cl_0, \Cl_1, \cR_c, \cR_d \rrangle. 
\end{equation*} 
\end{remark}

To prove the theorem we first construct an explicit functor $\Xi \colon \Db(\tX) \to \Db(Y, \Cl_0)$ in~\eqref{Xi} which is the left inverse to $\Phi$ up to twisting by a line bundle. Then we show that the objects $\Cl_1, \Cl_0, \cR_c, \cR_d$ 
are given by the values of $\Xi$ on 
$\cU_{\tX}, \cG, \rR_{\cG}\cF_c[1], \rR_{\cG}\cF_d[1]$. 

\begin{remark}
The advantage of the mutated objects $\cR_c$ and $\cR_d$ over $\Xi(\cF_c)$ and $\Xi(\cF_d)$ is that their $\Cl_0$-modified discriminant vanishes, see Lemma~\ref{lemma-chleq2}. 
This will be an important point in our construction of stability conditions in Section~\ref{section-stability}.  
\end{remark}

\begin{lemma}
\label{lemma-piGG}
There is an isomorphism $\pi_*\cHom(\cG,\cG) \cong \Cl_0$ of $\cO_Y$-algebras, where the left side is equipped with the natural algebra structure given by composition. 
\end{lemma} 

\begin{proof}
Consider the sequence
\begin{equation}
\label{eq_defGdual}
0 \to \cO_{\tX}(h) \to \cG^{\vee} \to \cO_{\tX}(H-h) \to 0
\end{equation}
obtained by dualizing the sequence in Lemma \ref{lem_defG}.
The decomposition~\eqref{eq_finalstepdec} shows that $\cG$ (and hence also $\cO_{\tX}(h) \otimes \cG$) is right orthogonal to 
$\pi^*\Db(Y) \subset \Db(\tX)$, 
which by adjunction implies $\pi_*(\cO_{\tX}(h) \otimes \cG) = 0$. 
Thus 
\begin{equation*}
\pi_*\cHom(\cG,\cG)  = \pi_*(\cG^{\vee} \otimes \cG) \cong \pi_*(\cO_{\tX}(H-h) \otimes \cG). 
\end{equation*} 
Now consider the sequence
\begin{equation*}
0 \to  \cO_{\tX} \to \cO_{\tX}(H-h) \otimes \cG \to \cO_{\tX}(H-2h) \to 0
\end{equation*} 
obtained by tensoring the sequence in Lemma \ref{lem_defG} by $\cO_{\tX}(H-h)$. 
By Lemma \ref{lemma-conic-bundle}, $\tX$ is a conic fibration over $Y$ in the projective bundle $\bP_Y(\cE)$, 
so we have an identification 
\begin{equation*} 
\pi_*\cO_{\tX}(H) \cong \cE^\vee \cong \cO_Y(h) \oplus \cS_Y^\vee .
\end{equation*} 
Thus, since $\cS_Y^{\vee}(-h) \cong \cS_Y$, we have 
\begin{equation*}
\pi_*\cO_{\tX}(H-2h) \cong \cO_Y(-h) \oplus \cS_Y(-h). 
\end{equation*} 
Therefore, we have an exact sequence 
\begin{equation*}
0 \to \cO_{Y} \to \pi_*\cHom(\cG,\cG) \to \cO_Y(-h) \oplus \cS_Y(-h) \to 0. 
\end{equation*} 
This extension splits because 
$\rH^1(\cO_Y(h)) = \rH^1(\cS_{Y}^{\vee}(h)) = 0$. 
This shows 
\begin{equation*} 
\pi_*\cHom(\cG,\cG) \cong \cO_Y \oplus \cO_Y(-h) \oplus \cS_Y(-h), 
\end{equation*} 
which together with~\eqref{Cl_0} shows there is an isomorphism $\pi_*\cHom(\cG,\cG) \cong \Cl_0$ of $\cO_Y$-modules. 

It remains to show that the algebra structures on $\pi_*\cHom(\cG,\cG)$ and $\Cl_0$ are compatible. Note that since $\cG$ has a $\Cl_0$-module structure, there is a morphism of algebras 
$$\pi^*\Cl_0  \to \cHom(\cG,\cG),$$
which by adjunction induces a morphism of algebras
$$\Cl_0  \to \pi_*\cHom(\cG,\cG).$$
This morphism is injective, since $\Cl_0$ is an Azumaya algebra away from the discriminant locus by \cite[Proposition 3.13]{kuznetsov08quadrics}. Moreover, $c_1(\Cl_0)=c_1(\pi_*\cHom(\cG,\cG))$, so we conclude that they are isomorphic.
\end{proof}

For any $F \in \Db(\tX)$, the object $\pi_*(\cG^{\vee} \otimes F)$ naturally has the structure of a right module over $\pi_* \cHom(\cG, \cG)$. By Lemma~\ref{lemma-piGG} we may regard this as a $\Cl_0$-module structure, and define a functor 
\begin{equation}
\label{Xi} 
\Xi = \pi_*(\cG^{\vee} \otimes -) \colon \Db(\tX) \to \Db(Y, \Cl_0). 
\end{equation} 

\begin{lemma}
\label{lemma-Xi-Phi}
There is an isomorphism of functors 
\begin{equation*}
\Xi \circ \Phi \simeq (-) \otimes \cO(-h) \colon \Db(Y, \Cl_0) \to \Db(Y, \Cl_0). 
\end{equation*} 
In particular, $\Xi \colon \Db(\tX) \to \Db(Y, \Cl_0)$ restricts to an equivalence 
$\Phi(\Db(Y, \Cl_0)) \simeq \Db(Y, \Cl_0)$. 
\end{lemma}

\begin{proof}
Recall that the functor $\Phi \colon \Db(Y, \Cl_0) \to \Db(\tX)$ is given by $\Phi(F) = \pi^*(F) \otimes_{\pi^*\Cl_0} \cE'$, where $\cE'$ is a certain sheaf of left $\Cl_0$-modules denoted by $\cE'_{-1,0}$ in \cite{kuznetsov08quadrics}. 
Thus 
\begin{equation*}
\Xi(\Phi(F)) = \pi_*(\pi^*(F) \otimes_{\pi^*\Cl_0} \cE' \otimes \cG^{\vee}) \simeq 
F \otimes_{\Cl_0} \pi_*(\cE' \otimes \cG^{\vee}), 
\end{equation*} 
so we must show $\pi_*(\cE' \otimes \cG^{\vee}) \simeq \Cl_0(-h)$. 
Recall that we have an exact sequence 
\begin{equation*}
0 \to \cO(h) \to \cG^\vee \to \cO(H-h) \to 0. 
\end{equation*} 
Note that $\cE' = \Phi(\Cl_0)$ is right orthogonal to $\pi^*\Db(Y)$, so  
$\pi_*(\cE') = 0$. Thus we also have $\pi_*(\cE' \otimes \cO(h)) = \pi_*(\cE') \otimes \cO(h) = 0$, and so 
\begin{equation*}
\pi_*(\cE' \otimes \cG^{\vee}) \simeq \pi_*(\cE' \otimes \cO(H-h)). 
\end{equation*} 
In terms of the notation from~\eqref{main-diagram} for the embedding of $X$ into the 
projective bundle $\bP_Y(\cE)$, we can write this as
\begin{equation*}
\pi_*(\cE' \otimes \cG^{\vee}) \simeq p_* (j_*\cE' \otimes \cO(H-h)). 
\end{equation*} 
Using the exact sequence 
\begin{equation*}
0 \to \cO(-2H) \otimes \Cl_{-1} \to \cO(-H) \otimes \Cl_0 \to j_* \cE' \to 0 
\end{equation*} 
on $\bP_Y(\cE)$ from \cite[Equation~(23)]{kuznetsov08quadrics} 
(our $\cE'$ is $\cE'_{-1,0}$ in the notation there), we conclude  that 
\begin{equation*}
\pi_*(\cE' \otimes \cG^{\vee}) \simeq \Cl_0(-h).
\end{equation*} 
As in Lemma \ref{lemma-piGG}, one can verify that it respects multiplication by $\Cl_0$ on each side. This completes the proof of the statement. 
\end{proof}

Next we compute the value of $\Xi$ on the objects $\cU_{\tX}, \rR_{\cG}\cF_c[1], \rR_{\cG}\cF_d[1]$. 

\begin{lemma}
\label{projU}
We have $\Xi(\cU_{\tX}) \cong \Cl_1$.
\end{lemma}

\begin{proof}
As $\cU_{\tX}$ is right orthogonal to $\pi^*\Db(Y)$ by~\eqref{eq_finalstepdec}, 
the same argument as in Lemma~\ref{lemma-piGG} shows that 
\begin{equation*}
\Xi(\cU_{\tX}) 
= \pi_*(\cG^{\vee} \otimes \cU_{\tX})
 \cong \pi_*(\cO_{\tX}(H-h) \otimes \cU_{\tX}). 
\end{equation*} 
Since $\cU_{\tX} \cong \cU_{\tX}^\vee(-H)$, this can be written as 
\begin{equation*}
\Xi(\cU_{\tX}) \cong \cO_Y(-h) \otimes \pi_*(\cU_{\tX}^\vee).
\end{equation*} 
Pushing forward the exact sequence of Lemma~\ref{lemma-RUv} by $\pi$, we obtain an exact sequence 
\begin{equation*}
0 \to \cS_Y^\vee \to \pi_*(\cU_{\tX}^\vee) \to \pi_* i_{E*}\cO_E \to 0. 
\end{equation*}
Pushing forward the exact sequence 
\begin{equation*}
0 \to \cO_{\tX}(-E) \to \cO_{\tX} \to i_{E*}\cO_E \to 0 
\end{equation*} 
by $\pi$ and using $-E = h - H$ and\eqref{eq_pushforwh-H}, 
we obtain an exact sequence 
\begin{equation*}
0 \to \cO_Y \to \pi_* i_{E*}\cO_E \to \cO_Y(-h) \to 0. 
\end{equation*} 
This extension splits because $\rH^1(\cO_Y(h)) = 0$, 
and so the extension describing $\pi_*(\cU_{\tX}^\vee)$ also splits 
because of the vanishings  $\rH^1(\cS_Y^{\vee}) = \rH^1(\cS_Y^{\vee}(h)) = 0$, which follow from the exact sequence $0 \to \cO(-h) \to \cS^{\vee} \to \cS_Y^{\vee} \to 0$ on $\Gr(2,V_4)$. 
Thus, using $\cS_Y^{\vee}(-h) \cong \cS_Y$, we conclude there is an isomorphism 
\begin{equation*}
\Xi(\cU_{\tX}) \cong \cS_Y \oplus \cO_Y(-h) \oplus \cO_Y(-2h). 
\end{equation*} 
Together with~\eqref{Cl_1} this shows there is an isomorphism $\Xi(\cU_{\tX}) \cong \Cl_1$ of $\cO_Y$-modules. 
Tracing through our construction of this isomorphism shows that it respects the $\Cl_0$-module structures on each side.
\end{proof}

We define objects 
\begin{equation}
\label{Ra-Rb}
\cR_c = \Xi(\rR_{\cG}\cF_c[1]) \quad \text{and} \quad \cR_d = \Xi(\rR_{\cG}\cF_d[1]) 
\end{equation} 
in $\Db(Y, \Cl_0)$. 

\begin{lemma}
\label{lemma-Ra-Rb}
As sheaves of $\cO_Y$-modules, we have isomorphisms 
\begin{equation*}
\cR_c \cong \cO_Y^{\oplus 2} \oplus \cO_Y(-h)^{\oplus 2} \cong \cR_d. 
\end{equation*} 
\end{lemma} 

\begin{remark}
As sheaves of $\Cl_0$-modules, $\cR_c$ and $\cR_d$ are of course different, since they are orthogonal objects in $\Db(Y, \Cl_0)$. 
\end{remark} 

Before proving Lemma~\ref{lemma-Ra-Rb}, we make a preliminary computation. 

\begin{lemma}
\label{lemma-cF'a}
Set $\cF'_c = \rR_{\cG}\cF_c[1]$ and $\cF'_d = \rR_{\cG}\cF_d[1]$. 
Then $\cF'_c$ and $\cF'_d$ are sheaves, given by extensions 
\begin{align*}
0 \to \cO_{\tX}(h-H)^{\oplus 2} \to & \cF'_c \to i_{E*}\cO_E(-c) \to 0 , \\ 
0 \to \cO_{\tX}(h-H)^{\oplus 2} \to & \cF'_d \to i_{E*}\cO_E(-d) \to 0. 
\end{align*} 
\end{lemma} 

\begin{proof}
We have 
\begin{alignat*}{2}
\Hom^{\bullet}(\cF_c, \cG) & \cong \Hom^{\bullet}(i_{E*}\cO_E(-c), \cG)[1]  && \qquad \text{by \eqref{eq_defFa} and \eqref{eq_finalstepdec}}  \\ 
& \cong \Hom^{\bullet}(i_{E*}\cO_E(-c), \cO_{\tX}(h-H))[1] & & \qquad \text{by \eqref{G} and \eqref{DbtX-Es}} \\ 
& \cong \Hom^{\bullet}(\cO_E(-c), \cO_{E})  && \qquad \text{since $i_E^!(F) = F(H-h)[-1]$}\\ 
& \cong \rH^\bullet(\cO_T(c)) \cong k^{2}[0],  && 
\end{alignat*} 
where $i_E^!(-)=i_E^*(-) \otimes \omega_{E/\tX}[-1]$ is the right adjoint to the pushforward functor $i_{E*}$. It follows that there is a commutative diagram 
\begin{equation*}
\xymatrix{
\cF_c \ar[r] \ar[d] & \cG^{\oplus 2}\ar[d] \ar[r] & \cF'_c \ar[d] \\
\cO_{\tX}(-h)^{\oplus 2}\ar[r] \ar[d] & \cO_{\tX}(-h)^{\oplus 2} \ar[d] \ar[r] & 0 \ar[d] \\
i_{E*}\cO_E(-c) \ar[r] & \cO_{\tX}(h-H)^{\oplus 2}[1] \ar[r] & \cF'_c[1] ,
}    
\end{equation*}
whose rows and columns are exact triangles, where the first column is given by~\eqref{eq_defFa} and the second by a direct sum of two copies of~\eqref{G}; in fact, the top row is obtained from the bottom row by applying $\rL_{\cO_{\tX(-h)}}$.
Indeed, the above computation of $\Hom^{\bullet}(\cF_c, \cG)$ implies the top left square is commutative, and then we can extend by the octahedral axiom to the full diagram. 
In particular, we see that $\cF'_c$ is an extension as claimed. 
The same argument applies to $\cF'_d$. 
\end{proof}

\begin{proof}[Proof of Lemma~\ref{lemma-Ra-Rb}]
The same argument as in Lemma~\ref{lemma-piGG} shows that 
\begin{equation*}
\cR_c = \pi_*(\cG^{\vee} \otimes \cF'_c) \cong \pi_*(\cF'_c(H-h)) . 
\end{equation*} 
By Lemma~\ref{lemma-cF'a} we have an exact sequence 
\begin{equation*}
0 \to \cO_{\tX}^{\oplus 2} \to  \cF'_c(H-h) \to i_{E*}\cO_E(d-h) \to 0 . 
\end{equation*} 
We claim that 
\begin{equation*} 
\pi_*i_{E*}\cO_E(d-h) \cong \cO_Y(-h)^{\oplus 2} . 
\end{equation*} 
This will complete the proof for $\cR_c$, since then the pushforward of the above exact sequence by $\pi$ gives an exact sequence 
\begin{equation*}
0 \to \cO_{Y}^{\oplus 2} \to  \cR_c \to \cO_Y(-h)^{\oplus 2} \to 0, 
\end{equation*} 
which splits since $\rH^1(\cO_Y(h)) = 0$. 

To prove the claim, note that by~\eqref{piE-a} we have 
\begin{equation*}
\cHom(\pi_*i_{E*}\cO_E(-d) , \cO_Y) \cong \cO_Y(h)^{\oplus 2}. 
\end{equation*} 
On the other hand, it is easy to compute that the relative canonical class of the morphism $\pi \circ i_E \colon E \to Y$ is $h$, so by Grothendieck duality the left side can also be written as
\begin{equation*}
\cHom(\pi_*i_{E*}\cO_E(-d) , \cO_Y) \cong 
\pi_*i_{E*}\cHom(\cO_E(-d), \cO_Y(h)) \cong 
\pi_*i_{E*} \cO_E(d+h). 
\end{equation*} 
Thus $\pi_*i_{E*} \cO_E(d+h) \cong \cO_Y(h)^{\oplus 2}$, which is equivalent to our claim. 
This finishes the proof for $\cR_c$, and the same argument applies for $\cR_d$. 
\end{proof}

\begin{proof}[Proof of Theorem~\ref{theorem-KX-Cl}] 
Define  
\begin{equation} 
\label{Psi} 
\Psi = \Xi \circ \rL_{\cO_{\tX}(-h)} \circ b^* \circ \rL_{\cU_X} \colon \Ku(X) \to \Db(Y, \Cl_0) . 
\end{equation} 
Then by Proposition~\ref{proposition-KX-Cl} and Lemma~\ref{lemma-Xi-Phi}, $\Psi$ is fully faithful and there is a 
semiorthogonal decomposition 
\begin{equation*}
\Db(Y, \Cl_0) = \llangle \Psi(\Ku(X)), 
\Xi(\cU_{\tX}), \Xi(\cG), \Xi(\rR_{\cG}\cF_c[1]), \Xi(\rR_{\cG}\cF_d[1])  
\rrangle . 
\end{equation*} 
By Lemmas~\ref{lemma-piGG} and \ref{projU} the first two exceptional objects 
in this decomposition are $\Cl_1$ and~$\Cl_0$. 
By definition the last two are $\cR_c$ and $\cR_d$, which by Lemma~\ref{lemma-Ra-Rb} are locally free of rank~$4$ as $\cO_Y$-modules. 
\end{proof}


\section{A Bogomolov inequality for $\Cl_0$-modules on $Y$}
\label{section-BI}

Maintaining the notation of the previous section, 
the purpose of this section is to prove a Bogomolov inequality for $\Cl_0$-modules on $Y$. 
This will allow us to define weak stability conditions on $\Db(Y,\Cl_0)$ by tilting slope stability, which we need for our construction of stability conditions on $\Ku(X)$ in the next section. 

\subsection{The main theorem} 
For an object $E \in \Db(Y, \Cl_0)$ we write $\Forg(E)$ for the underlying complex of $\cO_Y$-modules, and set  
\begin{equation*}
\ch(E) = \ch(\Forg(E)) \in \CH^*(Y) \otimes \bQ. 
\end{equation*} 
In particular, we can define the rank $\rk(E) = \ch_0(E)$ of any object $E \in \Db(Y, \Cl_0)$ and the associated slope function 
\begin{equation*}
\mu_h(E) = \frac{\ch_1(E) \cdot h^2}{\rk(E)h^3}. 
\end{equation*}
This gives rise as usual to a well-behaved notion of slope stability for objects of $\Coh(Y, \Cl_0)$. 
By definition, an object $E \in \Coh(Y, \Cl_0)$ is torsion free if $\Forg(E)$ is. 
 
\begin{theorem}
\label{thm_BI}
Let $E \in \Coh(Y,\Cl_0)$ be a slope semistable torsion free sheaf. Then 
\begin{equation*}
\label{eq_BI}
\Delta_{\Cl_0}(E):= h \cdot \emph{ch}_1(E)^2-2 \emph{rk}(E) \left( h \cdot \emph{ch}_2(E) - \frac{1}{4}\emph{rk}(E) \right)   \geq 0.
\end{equation*}
\end{theorem}

In order to prove the theorem, we first prove the corresponding inequality on a smooth hyperplane section of $Y$. Then we deduce the result on $Y$ by induction on the rank of $E$, arguing as in \cite[\S8]{BLMS} (see also \cite[\S3]{Lang:ssposchar}).

\subsection{The case of a hyperplane section}

Let $\Sigma \subset Y$ be a generic hyperplane section of $Y$, which is a quadric surface. 
Let $h_1$ and $h_2$ be the two generators of $\Pic(\Sigma)$ such that
\begin{equation*}
h= h_1+h_2.
\end{equation*} 
Consider the restriction of the conic fibration $\pi$ to $\Sigma$, sitting in the diagram
\begin{equation}
\vcenter{
\label{diagram-Z}
\xymatrix{
Z \ar@{^{(}->}[r]^{i_Z} \ar[d]_{\pi_{Z}}& \tX \ar[d]^\pi \\
\Sigma  \ar@{^{(}->}[r]^{i_\Sigma} & Y.
}    
}
\end{equation}
Note that by Bertini's theorem, $Z$ is smooth. 
The sheaves of even and odd parts of the Clifford algebra on $\Sigma$ associated to the conic fibration $\pi_Z$ 
are the restrictions to $\Sigma$ of the corresponding sheaves on $Y$; 
in order to simplify notation, we denote these restrictions by the same symbols $\Cl_0$ and $\Cl_1$. 
Note that the restriction to $\Sigma$ of the spinor bundle $\cS$ on $Y$ splits as 
$\cS_{\Sigma} \cong \cO_{\Sigma}(-h_1) \oplus \cO_{\Sigma}(-h_2)$.  
Thus by~\eqref{Cl_0} and~\eqref{Cl_1}, as sheaves of $\cO_{\Sigma}$-modules we have isomorphisms 
\begin{align}
\label{Cl_0-Sigma}
\Cl_0 & \cong \cO_\Sigma \oplus \cO_\Sigma(-h_1-h_2) \oplus \cO_\Sigma(-2h_1-h_2) \oplus \cO_\Sigma(-h_1-2h_2) ,   \\ 
\label{Cl_1-Sigma} 
\Cl_1 & \cong \cO_{\Sigma}(-h_1) \oplus \cO_{\Sigma}(-h_2) \oplus \cO_{\Sigma}(-h_1-h_2) \oplus \cO_{\Sigma}(-2h_1 - 2h_2). 
\end{align}

As in the case of $(Y, \Cl_0)$ discussed above, for an object $E \in \Db(\Sigma, \Cl_0)$ we define its Chern character 
as that of the underlying complex of $\cO_{\Sigma}$-modules $\Forg(E) \in \Db(\Sigma)$.  
The aim of this section is to prove the following Bogomolov inequality for 
$\Cl_0$-modules on $\Sigma$.

\begin{proposition}
\label{prop_BIsurfacecase}
Let $E \in \Coh(\Sigma, \Cl_0)$ be a slope semistable torsion free sheaf. 
Then 
\begin{equation*}
\Delta_{\Cl_0}(E):= \emph{ch}_1(E)^2-2 \emph{rk}(E) \left( \emph{ch}_2(E) - \frac{1}{4}\emph{rk}(E) \right) \geq 0.
\end{equation*} 
\end{proposition}

We give the proof at the end of this section, after some preliminary results.  
First we discuss the structure of the derived category of the conic fibration $\pi_Z \colon Z \to \Sigma$, which is parallel to that of $\pi \colon \tX \to Y$. 

\begin{lemma}
\label{lemma-DbZ}
\begin{enumerate}
\item \label{PhiSigma}
There is a fully faithful functor $\Phi_{\Sigma} \colon \Db(\Sigma, \Cl_0) \to \Db(Z)$ such that there is a semiorthogonal decomposition
\begin{equation*}
\Db(Z) = \llangle \Phi_{\Sigma}(\Db(\Sigma, \Cl_0)) , \pi_Z^*\Db(\Sigma) \rrangle. 
\end{equation*} 

\item \label{XiSigma}
Let $\cG_Z$ be the restriction to $Z$ of the vector bundle $\cG$ on $\tX$ defined in~\eqref{lem_defG}. Then there is an isomorphism of $\cO_\Sigma$-algebras $\pi_{Z*}\cHom(\cG_Z, \cG_Z) \simeq \Cl_0$ and the resulting functor 
\begin{equation*}
\Xi_{\Sigma} = \pi_{Z*}(\cG_Z^{\vee} \otimes - ) \colon \Db(Z) \to \Db(\Sigma, \Cl_0) 
\end{equation*} 
satisfies $\Xi_\Sigma \circ \Phi_\Sigma \simeq (-) \otimes \cO_{\Sigma}(-h)$. 
In particular, $\Xi_{\Sigma} \colon \Db(Z) \to \Db(\Sigma, \Cl_0)$ restricts to an equivalence 
$\Phi_{\Sigma}(\Db(\Sigma, \Cl_0)) \simeq \Db(\Sigma, \Cl_0)$. 
\end{enumerate}
\end{lemma}

\begin{proof}
\eqref{PhiSigma} holds by \cite{kuznetsov08quadrics}.
Base change gives an isomorphism $\pi_{Z*}\cHom(\cG_Z, \cG_Z) \cong i_\Sigma^* \pi_*\cHom(\cG, \cG)$ of $\cO_\Sigma$-algebras, so Lemma~\ref{lemma-piGG} then gives the first claim in~\eqref{XiSigma}. 
Then the isomorphism of functors $\Xi_\Sigma \circ \Phi_\Sigma \simeq (-) \otimes \cO_{\Sigma}(-h)$ follows as in Lemma~\ref{lemma-Xi-Phi}; 
alternatively, this is a formal consequence of 
Lemma~\ref{lemma-Xi-Phi} and the fact that 
$\Xi_\Sigma$ and $\Phi_\Sigma$ are the base changes of the $\Db(Y, \Cl_0)$-linear functors $\Xi$ and $\Phi$ along $(\Sigma, \Cl_0) \to (Y, \Cl_0)$.  
\end{proof} 

\begin{remark}
Note that $Z$ is a resolution of singularities (in fact the blowup in $T$) of a nodal GM threefold $X' \subset X$ containing $T$. 
Define $\wtilde{\Ku}(X') \subset \Db(Z)$ by the semiorthogonal decomposition 
\begin{equation*}
\Db(Z) = \llangle \wtilde{\Ku}(X'), \cO_Z, \cUv_Z \rrangle . 
\end{equation*} 
It would be interesting to study $\wtilde{\Ku}(X')$ as a categorical resolution of singularities of $\Ku(X')$ (defined by the same semiorthogonal decomposition as in the smooth case), in the spirit of~\cite{Kuz-res}. 
\end{remark}

We consider the Euler form on the Grothendieck group of $\Db(\Sigma, \Cl_0)$, defined by 
\begin{equation*}
\chi_{\Cl_0}(E, F) := \sum_{i} (-1)^i \dim \Hom^i_{\Cl_0}(E, F)  
\end{equation*} 
for $E, F \in \Db(\Sigma, \Cl_0)$, where $\Hom^i_{\Cl_0}(-,-)$ denotes the degree $i$ morphism space in 
the category $\Db(\Sigma, \Cl_0)$. 
Let $\cN(\Sigma,\Cl_0)$ be the numerical Grothendieck group of $\Db(\Sigma,\Cl_0)$, i.e. the quotient of the Grothendieck group of $\Db(\Sigma, \Cl_0)$ by the kernel of the Euler form $\chi_{\Cl_0}$. 

If $E \in \Db(\Sigma, \Cl_0)$ and $F \in \Db(\Sigma)$, then 
we write $E \otimes F$ for the usual tensor product $E \otimes_{\cO_Y} F$, 
which carries a natural $\Cl_0$-action  from the first factor, 
and hence can be regarded as an object of $\Db(\Sigma, \Cl_0)$. 
In case $F = \cO(D)$ for a divisor $D$, we write as usual $E(D) = E \otimes \cO(D)$. 

\begin{lemma}
\label{lemma_basisnGgroup}
The classes of the objects 
\begin{equation}
\label{basis_nGgroup}    
\Cl_0, \, \Cl_0(-h_1), \, \Cl_0(-h_2), \, \Cl_0(-h), 
\end{equation}
form a $\bQ$-basis for $\cN(\Sigma,\Cl_0) \otimes \bQ$.  
\end{lemma}

\begin{proof}
The numerical Grothendieck group is additive under semiorthogonal decompositions, so by Lemma~\ref{lemma-DbZ}\eqref{PhiSigma} we have
\begin{equation*}
\cN(Z) = \cN(\Sigma, \Cl_0) \oplus \cN(\Sigma), 
\end{equation*}
where $\cN(Z)$ and $\cN(\Sigma)$ denote the numerical Grothendieck groups of $\Db(Z)$ and $\Db(\Sigma)$. 
By Hirzebruch--Riemann--Roch, the numerical Grothendieck group of a smooth projective variety is rationally isomorphic to its numerical Chow ring, so $\cN(\Sigma)$ has rank $4$. 
Similarly, since by Lemma~\ref{lemma-PicZ} below the Picard rank of $Z$ is $3$, it follows that 
$\cN(Z)$ has rank $8$. 
Thus from the above direct sum decomposition we conclude 
that $\cN(\Sigma,\Cl_0)$ has rank $4$.  

Therefore, to prove the lemma it suffices to show that objects \eqref{basis_nGgroup} give linearly independent elements of $\cN(\Sigma, \Cl_0)$. 
For this, we compute the intersection matrix for these elements 
with respect to $\chi_{\Cl_0}$. 
Since these are 
exceptional objects of $\Db(\Sigma, \Cl_0)$, their self-pairing equals $1$. 
The other pairings can easily be computed using that the functor 
$\Forg: \Db(\Sigma,\Cl_0) \to \Db(\Sigma)$ 
is right adjoint to the functor  
$- \otimes \Cl_0: \Db(\Sigma) \to \Db(\Sigma, \Cl_0)$;  
for instance, using~\eqref{Cl_1-Sigma} we find 
\begin{align*}
\chi_{\Cl_0}(\Cl_0(-h_1),\Cl_0(-h_2)) & = \chi_{\Cl_0}(\Cl_0,\Cl_0(h_1-h_2)) \\
&=  \chi(\cO_\Sigma, \cO_\Sigma(h_1-h_2) \oplus \cO_\Sigma(-2h_2) \oplus \cO_\Sigma(-h_1-2h_2) \oplus \cO_\Sigma(-3h_2)) \\
& =  -3 . 
\end{align*}
In this way, we find that the matrix representing $\chi_{\Cl_0}$ on the collection \eqref{basis_nGgroup} is 
\begin{equation*}
\begin{pmatrix}
1 & 1 & 1 & 5 \\
1 & 1 & -3 & 1 \\
1 & -3 & 1 & 1 \\
5 & 1 & 1 & 1
\end{pmatrix}.
\end{equation*}
This matrix has determinant $256$, which implies the 
linear independence of the classes of the objects \eqref{basis_nGgroup} 
in $\cN(\Sigma, \Cl_0)$. 
\end{proof}

The rank of the Picard group of $Z$ was used in the above proof; for later use in this section, we prove the following more precise statement. 

\begin{lemma}
\label{lemma-PicZ}
There is an isomorphism $\Pic(Z) \cong \bZ h_1 \oplus \bZ h_2  \oplus \bZ H$. 
\end{lemma}

\begin{proof}
First we claim that the Picard rank of $Z$ is indeed $3$. 
Equivalently, the relative Picard rank of $\pi_Z \colon Z \to \Sigma$ is $1$, or equivalently still, 
every irreducible divisor $D$ in $\Sigma$ has irreducible preimage in $Z$. 
If $D$ is not contained in the discriminant divisor of $\pi_Z$, then $\pi_{Z}^{-1}(D) \to D$ is a flat morphism with irreducible generic fiber, hence $\pi_Z^{-1}(D)$ is irreducible. 
It remains to show that the same holds if $D$ is a component of the discriminant divisor $D(\pi_Z) \subset \Sigma$ of $\pi_Z$. 
Note that $D(\pi_Z) = D(\pi) \cap \Sigma$ where $D(\pi) \subset Y$ is the discriminant divisor of $\pi \colon \tX \to Y$. 
On the other hand, it follows from the description of $\tX$ as the blowup of $X$ that 
$\pi$ has relative Picard rank $1$, so in particular every component of $D(\pi)$ has irreducible preimage in $\tX$. 
By Bertini's theorem and the genericity of $\Sigma$, the same statement for $D(\pi_Z)$ follows. 

It follows from the previous paragraph that $\Pic(Z) \otimes \bQ$ is spanned by $h_1, h_2, H$. 
To show that this holds integrally, let $D \in \Pic(Z)$ and write 
\begin{equation*}
D = a h_1 + b h_2 + c H 
\end{equation*} 
for rational numbers $a, b, c \in \bQ$. 
Note that 
\begin{equation*}
\frac{1}{2} h_1h_2 , \quad \frac{1}{2} h_1 H, \quad  \frac{1}{2} h_2 H 
\end{equation*} 
are the classes of curves on $Z$: the first is represented by an irreducible component of a degenerate fiber of $\pi_Z$ that is a union of two lines; 
the second is represented by a section of the restriction of $\pi_Z$ to a line of class $h_1$ in $\Sigma$; and the third is similarly represented by a section of the restriction to a line of class $h_2$. 
Therefore, the following intersection numbers are integers: 
\begin{align*}
\frac{1}{2} h_1h_2 \cdot D & = c \\ 
\frac{1}{2} h_1 H \cdot D & = b + \frac{1}{2} h_1 H^2 c \\ 
\frac{1}{2} h_2 H \cdot D & = a + \frac{1}{2} h_2 H^2c. 
\end{align*} 
The coefficients of $c$ in the second two equations are similarly integers (in fact both equal $2$), so we conclude that $a,b,c$ are integers. 
\end{proof}

Next we prove a Hirzebruch--Riemann--Roch formula for $\Cl_0$-modules on $\Sigma$, which gives a topological formula for the Euler form. 

\begin{lemma}
\label{lemma:HRR}
For $E, F \in \Db(\Sigma,\Cl_0)$ we have
\begin{equation*}
\chi_{\Cl_0}(E,F)=\int_{\Sigma} \emph{ch}(E)^\vee \emph{ch}(F) \cdot \left(\frac{1}{4}  -\frac{1}{16} h_1h_2 \right) . 
\end{equation*}
In particular, we have 
\begin{equation}
\label{eq_chiSigmaCl0}
\chi_{\Cl_0}(E,E)= -\frac{1}{16} \rk(E)^2 + \frac{1}{4}\left(2\rk(E)\ch_2(E)-\ch_1(E)^2 \right).
\end{equation}
\end{lemma}

\begin{proof}
By Lemma~\ref{lemma_basisnGgroup} it suffices to check the formula for $E$ and $F$ among the objects~\eqref{basis_nGgroup}, which can be verified by a direct computation. 
Alternatively, we can argue more conceptually as follows. 
By Lemma~\ref{lemma_basisnGgroup} it suffices to prove the formula for $E$ of the form $E = \Cl_0 \otimes E'$ for $E' \in \Db(\Sigma)$ (even $E'$ a line bundle would suffice). 
The functor $- \otimes E' \colon \Db(\Sigma, \Cl_0) \to \Db(\Sigma, \Cl_0)$ has a right adjoint given by $- \otimes (E')^{\vee}$, and as mentioned above $- \otimes \Cl_0 \colon \Db(\Sigma) \to \Db(\Sigma, \Cl_0)$ has as a right adjoint the forgetful functor $\Forg$; thus, we have 
\begin{equation*}
\chi_{\Cl_0}(E, F) = \chi_{\Cl_0}(\Cl_0, (E')^{\vee} \otimes F) = 
\chi(\cO_{\Sigma}, (E')^{\vee} \otimes \Forg(F)). 
\end{equation*}
Note that 
\begin{equation*}
\ch((E')^{\vee}) = \frac{\ch(E)^{\vee}}{\ch(\Cl_0)^{\vee}} , 
\end{equation*}  
so by the usual Hirzebruch--Riemann--Roch theorem, the right side is given by  
\begin{equation*}
\int_{\Sigma} \ch((E')^{\vee}) \ch(F) \td(\Sigma) = 
\int_{\Sigma} \ch(E)^{\vee} \ch(F) \frac{\td(\Sigma)}{\ch(\Cl_0)^{\vee}} . 
\end{equation*} 
Finally, a direct computation shows that 
\begin{equation*}
\frac{\td(\Sigma)}{\ch(\Cl_0)^{\vee}} = \frac{1}{4}  -\frac{1}{16} h_1h_2 , 
\end{equation*} 
which finishes the proof. 
\end{proof}

\begin{remark}
The proof of Lemma~\ref{lemma:HRR} applies more generally to give a Hirzebruch--Riemann--Roch theorem 
in the case of a pair $(M, \cA)$ where $M$ is a smooth projective variety and $\cA$ is a coherent sheaf of algebras on $M$, 
such that $\cN(M, \cA) \otimes \bQ$ is spanned by classes of objects of the form $E = \cA \otimes E'$ for $E' \in \Db(M)$. Namely, if $\chi_{\cA}$ denotes the Euler form for $\Db(M, \cA)$, then 
\begin{equation*}
\chi_{\cA}(E, F) = \int_{M}  \ch(E)^{\vee} \ch(F) \frac{\td(\Sigma)}{\ch(\cA)^{\vee}} . 
\end{equation*} 
Of course the above assumption that $\cN(M, \cA) \otimes \bQ$ is spanned by objects of the form $\cA \otimes E'$ is not always satisfied, as for example in the case of $(Y,\Cl_0)$.
\end{remark}

The following result concerns Serre duality on $(\Sigma, \Cl_0)$. 
\begin{lemma}
\label{lemma-serre-Cl0}
\begin{enumerate}
\item \label{S_Cl_0}
The category $\Db(\Sigma,\Cl_0)$ has Serre functor given by
$S_{\Cl_0}(-) = - \otimes_{\Cl_0} \Cl_1[2]$. 

\item \label{Coh-SigmaCl0-hd2}
The abelian category $\Coh(\Sigma, \Cl_0)$ has homological dimension $2$. 

\item \label{Cl1-preserve-ch} 
If $E \in \Db(\Sigma, \Cl_0)$, then $\ch(E \otimes_{\Cl_0} \Cl_1) = \ch(E)$. 

\item \label{Cl1-preserve-stability}
Let $E \in \Coh(\Sigma, \Cl_0)$. 
Then $E$ is a slope (semi)stable torsion free sheaf 
if and only if the same is true of $E\otimes_{\Cl_0} \Cl_1$. 
\end{enumerate}
\end{lemma}

\begin{proof}
By the discussion in the beginning of \cite[\S7]{BLMS}, 
the category $\Db(\Sigma,\Cl_0)$ has Serre functor given by 
$S_{\Cl_0}(-) = \omega_\Sigma \otimes_{\cO_\Sigma} (-) \otimes_{\Cl_0} \Cl_0^\vee  [2]$, 
where $\Cl_0^{\vee}$ denotes the dual of $\Cl_0$ as a coherent sheaf, with its natural $\Cl_0$-bimodule structure. 
Using~\eqref{Cl_0-Sigma} and~\eqref{Cl_1-Sigma}, it is easy to compute 
\begin{equation*}
\omega_{\Sigma}  \otimes \Cl_0^{\vee} \cong \Cl_1, 
\end{equation*} 
which proves~\eqref{S_Cl_0}. 
Then \eqref{Coh-SigmaCl0-hd2} follows from~\eqref{S_Cl_0}, because  
$\Cl_1$ is a flat $\Cl_0$-module by \cite[Corollary 3.9]{kuznetsov08quadrics}. 
By~\eqref{Cl_0-Sigma} and~\eqref{Cl_1-Sigma} we have $\ch(\Cl_0) = \ch(\Cl_1)$, 
from which~\eqref{Cl1-preserve-ch} follows from Lemma~\ref{lemma_basisnGgroup}. 
Finally, note that the functor $- \otimes_{\Cl_0} \Cl_1$ gives an involutive exact autoequivalence 
of $\Coh(\Sigma, \Cl_0)$; involutivity is a consequence of the isomorphism $\Cl_1 \otimes_{\Cl_0} \Cl_1 \cong \Cl_0$, which holds by \cite[Corollary 3.9]{kuznetsov08quadrics}. 
From this observation, \eqref{Cl1-preserve-stability} follows from~\eqref{Cl1-preserve-ch}. 
\end{proof}

Finally, we prove some numerical constraints on objects in $\Db(\Sigma, \Cl_0)$. 
\begin{lemma}
\label{lemma_rankCl0modules}
The rank of any object of $\Db(\Sigma, \Cl_0)$ is divisible by $4$.
\end{lemma}

\begin{proof}
By Lemma \ref{lemma-DbZ}\eqref{XiSigma} 
it suffices to show that for $E \in \Db(Z)$ the rank of 
$\pi_{Z*}(\cG_Z^{\vee} \otimes E)$ is divisible by $4$. 
If $i_C \colon C \to Z$ denotes a generic fiber of the conic fibration $\pi_Z \colon Z \to \Sigma$, then 
by base change this rank can be computed as 
\begin{equation*}
\chi(C, i_C^*(\cG_Z^\vee \otimes E)) = \rk(\cG_Z^{\vee} \otimes E) + c_1(\cG_Z^{\vee} \otimes E)\cdot C 
= 4 \rk(E) + 2 c_1(E) \cdot C, 
\end{equation*}
where for the equality we used that $\rk(\cG^{\vee}) = 2$ and $c_1(\cG^{\vee}) = H$, as follows from \eqref{G}. 
It remains to observe that the intersection number of any integral divisor on $Z$ with $C$ is even, which follows from Lemma~\ref{lemma-PicZ}. 
\end{proof}

\begin{lemma}
\label{lemma_serrefunctfixch}
If $E \in \Coh(\Sigma, \Cl_0)$ is a slope stable torsion free sheaf, then 
$\chi_{\Cl_0}(E, E) \leq 2$. 
\end{lemma} 

\begin{proof}
By Lemma~\ref{lemma-serre-Cl0}\eqref{Coh-SigmaCl0-hd2} we have 
\begin{equation*}
\chi_{\Cl_0}(E, E) = \dim \Hom_{\Cl_0}(E, E) - \dim \Hom_{\Cl_0}^1(E,E) + \dim \Hom^2_{\Cl_0}(E,E). 
\end{equation*} 
By stability the first term is $1$. 
On the other hand, by Lemma~\ref{lemma-serre-Cl0}\eqref{S_Cl_0} we have 
\begin{equation*}
\Hom^2_{\Cl_0}(E,E) \cong \Hom_{\Cl_0}(E, E \otimes_{\Cl_0} \Cl_1) .
\end{equation*}
By parts~\eqref{Cl1-preserve-ch} and~\eqref{Cl1-preserve-stability} of 
Lemma~\ref{lemma-serre-Cl0},  
the object $E \otimes_{\Cl_0} \Cl_1$ is a stable torsion free sheaf of the same slope as $E$, 
so $\Hom_{\Cl_0}(E, E \otimes_{\Cl_0} \Cl_1)$ is either $1$-dimensional or vanishes, according to whether $E \otimes_{\Cl_0} \Cl_1$ is isomorphic to $E$ or not. 
In either case, the third term in the above sum is at most $1$. 
\end{proof}

\begin{lemma}
\label{lemma_hopeistrue}
There are no objects $E \in \Db(\Sigma,\Cl_0)$ with $\rk(E)=4$ and $\chi_{\Cl_0}(E,E)=2$.
\end{lemma}

\begin{proof}
Suppose for sake of contradiction that we have such an object $E$, with 
\begin{equation*}
\ch(E) = 4 + b_1h_1+b_2h_2 + ch_1h_2 .
\end{equation*}
Note that $b_1, b_2, c$ are integers, because on $\Sigma$ the second Chern character of any object is always integral. 
By \eqref{eq_chiSigmaCl0}, $\chi_{\Cl_0}(E,E)=2$ implies that $8c-2b_1b_2=12$. 
Thus $b_1b_2$ is divisible by $2$ but not divisible by $4$,
and hence $b_1$ and $b_2$ are of different parities.

On the other hand, by Lemma \ref{lemma-Ra-Rb} the object $\cR_c|_\Sigma \in \Db(\Sigma, \Cl_0)$ 
has 
\begin{equation*}
\ch(\cR_c|_\Sigma) = 4 - 2h + c'h_1h_2 
\end{equation*} 
for some $c' \in \bZ$. 
By Lemma \ref{lemma:HRR} we have
\begin{equation*}
\chi_{\Cl_0}(E,\cR_c|_\Sigma)=-1 + c+c' + \frac{1}{2}(b_1+b_2), 
\end{equation*}
which is not an integer because $b_1$ and $b_2$ are of different parities. 
This contradiction finishes the proof.
\end{proof}

\begin{proof}[Proof of Proposition~\ref{prop_BIsurfacecase}]
First assume that $E$ is stable. Using \eqref{eq_chiSigmaCl0}, we get 
$$\frac{1}{4}\Delta_{\Cl_0}(E) = \frac{1}{16}\text{rk}(E)^2-\chi_{\Cl_0}(E,E).$$
By Lemma \ref{lemma_rankCl0modules} the first term on the right is an integer 
satisfying $\frac{1}{16}\text{rk}(E)^2\geq 1$, 
and by Lemma \ref{lemma_serrefunctfixch} we have $\chi_{\Cl_0}(E,E) \leq 2$. 
Now Lemma \ref{lemma_hopeistrue} implies that the equalities cannot be achieved at the same time, so we conclude that $\Delta_{\Cl_0}(E) \geq 0$. 

If $E$ is strictly semistable, then we consider its Jordan-H\"{o}lder stable factors $E_1, \dots, E_m$. By the previous part, we have $\Delta_{\Cl_0}(E_i) \geq 0$ for every $i=0,\dots,m$. By \cite[Lemma A.6]{BMS:stabCY3s}, this implies $\Delta_{\Cl_0}(E) \geq 0$.
\end{proof}

\subsection{Induction argument}

Let us come back to the case of the threefold $Y$. Before proving Theorem \ref{thm_BI}, we need the following lemma.

\begin{lemma}
\label{lemma_BIblowupconic}
Let $q: \wtilde{Y} \to Y$ be the blowup of $Y$ in a conic $C \subset Y$. Assume that for every $\mu_h$-semistable $\Cl_0$-module $F \in \emph{Coh}(Y,\Cl_0)$, we have $\Delta_{\Cl_0}(F) \geq 0$. Then for every $E \in \emph{Coh}(\wtilde{Y},q^*\Cl_0)$ which is $\mu_{q^*h}$-semistable, we have $\Delta_{q^*\Cl_0}(E) \geq 0$. 
\end{lemma}
\begin{proof}
By \cite[pag.\ 608-609]{GH:principlesofag}, we have that
$$c_1(\wtilde{Y})=q^*c_1(Y)-e \quad \text{and} \quad c_2(\wtilde{Y})=q^*(c_2(Y)+[C])-q^*c_1(Y)\cdot e,$$
where $e$ denotes the class of the exceptional divisor of $q$ and $[C] \in H^4(Y,\bZ)$ is the class of the conic $C$. The relative Todd class of $q$ is 
$$\td(\cT_q)=1-\frac{e}{2}+\frac{1}{12}(e^2+q^*[C]).$$
If $f$ is the class of a fiber of $q|_E: E \to C$, by \cite[Lemma 2.2.14]{ParShafa:algebraicgeom} we have
$$e^2=-q^*[C]+\text{deg}(\cN_{C/X})f=-q^*[C] +2f.$$
In particular, it follows that
$$\td(\cT_q)=1-\frac{e}{2}+\frac{1}{6}f.$$

Now consider a $\mu_{q^*h}$-semistable $q^*\Cl_0$-module $E$ on $\wtilde{Y}$. We can write
$$\ch_1(E)=q^*l + a e \quad \text{and} \quad \ch_2(E)=q^*m +bf,$$
where $l \in H^{1,1}(Y,\bZ)$ and $m \in H^{2,2}(Y,\bZ)$. Using Grothendieck-Riemann-Roch and $q_*e=q_*f=0$, we get
$$\ch(q^*q_*E)_{\leq 2}=q^*q_*(\ch(E)_{\leq 2} \td(\cT_q))=(\rk(E), q^*l, q^*m+\frac{a}{2}q^*[C]).$$
In particular, since $q^*h \cdot q^*[C]=2$, we have
$$\Delta_{q^*\Cl_0}(q^*q_*E)=q^*h \cdot q^*l^2-2\rk(E) \left( q^*h \cdot q^*m+a - \frac{1}{4}\rk(E) \right).$$

Up to twisting by a power of $\cO_{\wtilde{Y}}(e)$, we may assume that $0 \leq a < \rk(E)$. Then we have
$$\Delta_{q^*\Cl_0}(q^*q_*E) \leq q^*h \cdot q^*l^2 -2a^2 -2\rk(E) \left( q^*h \cdot q^*m - \frac{1}{4}\rk(E) \right)=\Delta_{q^*\Cl_0}(E),$$
since $q^*h \cdot e^2=-q^*h \cdot q^*[C]$.

On the other hand, by hypothesis and \cite[Lemma 8.5]{BLMS}, we deduce that
$$\Delta_{q^*\Cl_0}(q^*q_*E) \geq \Delta_{\Cl_0}(R^0q_*E) \geq 0.$$
This ends the proof of our claim.
\end{proof}

Now we have all the ingredients to apply the argument in \cite[Section 8]{BLMS} to prove Theorem \ref{thm_BI} in this setting. We explain here only the parts that differ from the cubic fourfolds case, for sake of completeness. 

Consider the following statement.

\begin{theorem}[$\rk(E)=r$]
\label{thm_restriction_rankr}
Let $E \in \emph{Coh}(Y,\Cl_0)$ be a rank $r$ $\mu_h$-slope semistable torsion-free sheaf. Assume that the restriction $E|_\Sigma \in \emph{Coh}(D,\Cl_0|_\Sigma)$ of $E$ to a general divisor $\Sigma \in |h|$ is not slope semistable with respect to $h|_\Sigma$. Then
$$2\sum_{i<j}r_ir_j(\mu_i-\mu_j)^2 \leq  \Delta_{\Cl_0}(E),$$
where the $\mu_i$ (resp.\ $r_i$) denote the slopes (resp.\ the ranks) of the Harder-Narasimhan factors of $E|_\Sigma$.
\end{theorem}

Similar to the proof of \cite[Theorem 8.10]{BLMS}, we argue by induction on the rank; in particular, we prove that Theorem \ref{thm_restriction_rankr}($r$) implies Theorem \ref{thm_BI}($r$) and Theorem \ref{thm_BI}($r-4$) implies Theorem \ref{thm_restriction_rankr}($r$). Note that Theorem \ref{thm_restriction_rankr}($r=4$) holds since $4$ is the minimal rank by Lemma \ref{lemma_rankCl0modules}, giving the base step of the induction.

\begin{proof}[Theorem \ref{thm_restriction_rankr}($r$) implies Theorem \ref{thm_BI}($r$)]
Assume that $E \in \Coh(Y,\Cl_0)$ is $\mu_h$-semistable of rank $r$ with $\Delta_{\Cl_0}(E) <0$. Then Theorem \ref{thm_restriction_rankr}($r$) implies that $E|_\Sigma$ is semistable for any general divisor $\Sigma \in |h|$, in contradiction with Proposition \ref{prop_BIsurfacecase}. This implies the statement.
\end{proof}

\begin{proof}[Theorem \ref{thm_BI}($r-4$) implies Theorem \ref{thm_restriction_rankr}($r$)]
Let $\Pi \subset |h|$ be a general pencil of hypersurfaces and consider the incidence variety
$$\wtilde{Y}:=\lbrace (\Sigma,y) \in \Pi \times Y: y \in \Sigma \rbrace.$$
Denote by $p: \wtilde{Y} \to \Pi$ and $q: \wtilde{Y} \to Y$ the standard projections, where $q$ is the blow-up of the base locus of $\Pi$ which is a smooth conic curve in $Y$. We denote by $f$ the class of a fiber of $p$. 

Assume that $E$ is as in the statement. The fact that $E|_\Sigma$ is not $\mu_{h|_\Sigma}$-semistable, implies that $q^*E$ is not $\mu_{h,f}$-semistable. Then consider the Harder-Narasimhan filtration of $q^*E$ with respect to $\mu_{h,f}$:
$$0 \subset E_1 \subset \dots  \subset E_m=q^*E.$$
The quotients $F_i:=E_i/E_{i-1}$ are $\mu_{h,f}$-semistable of rank $\leq r-4$. By Theorem \ref{thm_BI}($\leq r-4$), Lemma \ref{lemma_BIblowupconic} and \cite[Proposition 8.9]{BLMS}, we deduce that $\Delta_{q^*\Cl_0}(F_i) \geq 0$.

On the other hand, we have $\ch_1(F_i)= q^*l_i + a_ie$, with $l_i \in H^{1,1}(Y,\bZ)$ and $a_i \in \bZ$. Since $f=h-e$, $q^*h \cdot e^2=-2$ and $h^3=2$, we deduce that
\begin{equation}
\label{eq_mui}
\mu_i:=\mu_{h,f}(F_i)=\frac{h^2 \cdot l_i+2a_i}{2r_i}.    
\end{equation}
By \cite[Lemma 8.5]{BLMS}, as $R^0q_*E_i \subset E$, we obtain that
\begin{equation}
\label{eq_boundmuE}
\frac{\sum_{j \leq i} h^2 \cdot l_j}{\sum_{j \leq i} 2r_j} \leq 
 \mu_h(E).    
\end{equation}
By \eqref{eq_mui} and \eqref{eq_boundmuE}, we deduce that
\begin{equation*}
\sum_{j \leq i} r_j(\mu_j-\mu_h(E)) \leq \sum_{j \leq i} a_j.    
\end{equation*}
Following the same computation of \cite[Section 3.9]{Lang:ssposchar} and of \cite[page 37]{BLMS}, we deduce the desired statement.
\end{proof}


\section{Stability conditions on $\Ku(X)$}
\label{section-stability}

In this section, we prove the first main result of this paper, Theorem~\ref{main-theorem}, 
asserting the existence of stability conditions on the Kuznetsov component of any GM variety, 
as well as Corollary~\ref{corollary-DbX-stability} giving stability conditions on the derived category 
of any GM variety. 
We also prove Theorem~\ref{theorem-Stab-dagger}, which asserts that 
for a GM fourfold or sixfold over $\bC$ our construction gives full numerical stability conditions. 

First we prove the theorems in the case of an ordinary GM fourfold with smooth canonical quadric, 
in Sections~\ref{section-main-theorem-fourfold} and~\ref{section-support-fourfold}. 
Then in Section~\ref{section-main-theorem-general} we handle the general case, by using the 
duality conjecture proved in \cite{KuzPerry:cones} to reduce it to the known cases. 

\subsection{Stability conditions} 
We refer to \cite[Section 2]{BLMS} for background on the notions of tilting and (weak) stability conditions. For the reader's convenience, we briefly recall here the definition of a (weak) stability condition on a triangulated category $\cT$.

\begin{definition}
The \emph{heart of a bounded t-structure} on $\cT$ is a full additive subcategory $\cA \subset \cT$ such that: 
\begin{enumerate}
\item For $E$, $F \in \cA$ and $n <0$, we have $\Hom(E,F[n])=0$. 
\item For every $E \in \cT$, there exists a sequence of morphisms
\begin{equation*}
0=E_0 \xrightarrow{\, f_1 \,} E_1 \xrightarrow{\, f_2 \,} \cdots \xrightarrow{\, f_{m-1} \,} E_{m-1} \xrightarrow{\, f_m \,} E_m=E 
\end{equation*}
such that the cone of $f_i$ is contained in the subcategory $\cA[k_i] \subset \cT$ for some sequence of integers $k_1 > k_2 > \dots > k_m$. 
\end{enumerate}
\end{definition}

It is easy to see that the heart of a bounded t-structure is in fact an abelian category. 
For an abelian category $\cA$ or a triangulated category $\cT$, we write $\rK(\cA)$ or $\rK(\cT)$ for the Grothendieck group. 

\begin{definition}
A \emph{weak stability function} on an abelian category $\cA$ is a group homomorphism 
$Z \colon \rK(\cA) \rightarrow \bC$ such that for all $E \in \cA$ the complex number $Z(E)$ is contained in the set $\set{ z \in \bC \st \Im z > 0, \text{ or } \Im z = 0 \text{ and } \Re z \leq 0 }$. 
We say $Z$ is a \emph{stability function} if for all $E \in \cA$, $Z(E)$ is contained in the smaller set 
 $\set{ z \in \bC \st \Im z > 0, \text{ or } \Im z = 0 \text{ and } \Re z < 0 }$. 
\end{definition}

\begin{definition}
Let $\Lambda$ be a finite rank free abelian group with a surjective group homomorphism $v \colon \rK(\cT) \twoheadrightarrow \Lambda$.
A \emph{weak stability condition} on $\cT$ with respect to $\Lambda$ is the data of a pair $\sigma=(\cA,Z)$, 
where $\cA$ is the heart of a bounded t-structure on $\cT$ and $Z \colon \Lambda \to \bC$ is a group homomorphism, satisfying the following properties: 
\begin{enumerate}
\item The composition $\rK(\cA)=\rK(\cT) \xrightarrow{\, v \,} \Lambda \xrightarrow{\, Z \,} \bC$ is a weak stability function on $\cA$. 
We write $Z(-)$ instead of $Z(v(-))$ for simplicity, and 
for any object $E \in \cA$ define the slope with respect to $Z$ by
\begin{equation*}
\mu_{\sigma}(E)= 
\begin{cases}
-\frac{\Re Z(E)}{\Im Z(E)} & \text{if } \Im Z(E) > 0 \\
+ \infty & \text{otherwise.}
\end{cases}
\end{equation*} 
An object $E \in \cA$ is $\sigma$-semistable (resp.\ $\sigma$-stable) if for every proper subobject $F$ of $E$, we have $\mu_{\sigma}(F) \leq \mu_{\sigma}(E)$ (resp.\ $\mu_{\sigma}(F) < \mu_{\sigma}(E/F))$. 

\item Any object of $\cA$ has a Harder-Narasimhan filtration with $\sigma$-semistable factors. 
\item (Support property) There exists a quadratic form $Q$ on $\Lambda \otimes \bR$ such that the restriction of $Q$ to $\ker Z$ is negative definite and $Q(E) \geq 0$ for all $\sigma$-semistable objects $E$ in $\cA$.
\end{enumerate}
If in addition $Z$ is a stability function, then $\sigma$ is called a \emph{Bridgeland stability condition} with respect to $\Lambda$.
\end{definition}

Below, we shall often give references to results in the literature that are stated for stability conditions on derived categories of varieties, but whose proof works the same in the setting of the derived category of $\Cl_0$-modules on $Y$.

\subsection{Theorem~\ref{main-theorem} for ordinary fourfolds with smooth canonical quadric} 
\label{section-main-theorem-fourfold} 
In this section, $X$ denotes an ordinary GM fourfold with smooth canonical quadric; 
we use the notation in this setting from Sections~\ref{section-conic-fibration} and~\ref{section-BI}. 
Recall that by Theorem~\ref{theorem-KX-Cl}, the Kuznetsov component 
$\Ku(X)$ embeds into the derived category of a noncommutative quadric threefold $(Y, \Cl_0)$ as the orthogonal to four exceptional objects. 
The methods of \cite{BLMS} show that if we can produce a suitable weak stability condition on $\Db(Y, \Cl_0)$, then there is an induced stability condition on the semiorthogonal 
component $\Ku(X)$. 
Thus most of this section is dedicated to constructing such a stability condition on $\Db(Y, \Cl_0)$, as a double tilt of slope stability. 

We define a modified Chern character for objects in $E \in \Db(Y, \Cl_0)$ by  
\begin{equation*}
\ch_{\Cl_0}(E) = \ch(E) \cdot \left( 1 -\frac{1}{8}h^2 \right), 
\end{equation*} 
so that the discriminant appearing in the Bogomolov inequality of Theorem~\ref{thm_BI} can be written in the familiar form 
\begin{equation}
\label{DeltaCl0}
\Delta_{\Cl_0}(E)= h \cdot \left( \ch_{\Cl_0,1}^2(E) -2\rk(E) \ch_{\Cl_0,2}(E) \right) . 
\end{equation} 
Here $\ch_{\Cl_0,i}(E)$ denotes the term of degree $i$ of $\ch_{\Cl_0}(E)$.

We define $\Lambda_{\Cl_0}$ as the rank $3$ lattice generated by 
the vectors 
\begin{equation*} 
(h^{3} \cdot \ch_{\Cl_0,0}(E), h^2 \cdot \ch_{\Cl_0,1}(E), h \cdot \ch_{\Cl_0,2}(E))\in \bQ^3 
\end{equation*} 
for $E \in \Db(Y, \Cl_0)$.  
Note that by definition this lattice receives a surjective homomorphism 
\begin{equation*}
\rK(Y,\Cl_0) \twoheadrightarrow \Lambda_{\Cl_0}    
\end{equation*} 
from the Grothendieck group of $\Db(Y, \Cl_0)$.

For $\beta \in \bR$, we denote by $\Coh^{\beta}(Y, \Cl_0)$ the heart of a bounded t-structure obtained by tilting $\Coh(Y, \Cl_0)$ with respect to slope stability at the slope $\mu = \beta$ (see \cite{HapReiSm}). 
We consider the twisted Chern character 
\begin{equation*}
\ch_{\Cl_0}^\beta=e^{-\beta h }\ch_{\Cl_0},  
\end{equation*}
and write $\ch_{\Cl_0,i}^\beta$ for its degree $i$ term. 

\begin{lemma}
For any $(\alpha, \beta)\in \bR_{>0}\times \bR$,
the pair 
$$\sigma_{\alpha, \beta}=(\emph{Coh}^{\beta}(Y,\Cl_0), Z_{\alpha, \beta})$$
with
$$
Z_{\alpha, \beta}(E)
:=\frac{1}{2}\alpha^2 \emph{ch}_{\Cl_0,0}^{\beta}(E)h^{3} -h\cdot \emph{ch}_{\Cl_0,2}^{\beta}(E)
+\sqrt{-1} h^2 \cdot \emph{ch}_{\Cl_0,1}^{\beta}(E)
$$
defines a weak stability condition on $\emph{D}^b(Y,\Cl_0)$ with respect to $\Lambda_{\Cl_0}$. 
The quadratic form on $\Lambda_{\Cl_0} \otimes \bQ$ required by the support property can be given by the discriminant $\Delta_{\Cl_0}$.
Moreover, these weak stability conditions vary continuously as $(\alpha, \beta) \in \bR_{>0}\times \bR$ varies.
\end{lemma}

\begin{proof}
This follows from the Bogomolov inequality of Theorem~\ref{thm_BI} by standard arguments, cf. \cite[Proposition 2.12]{BLMS} 
\end{proof} 

To induce a stability condition on $\Ku(X)$ from one on $(Y, \Cl_0)$, 
we will need stability properties of the exceptional objects $\Cl_0, \Cl_1, \cR_c, \cR_d$ 
appearing in Theorem~\ref{theorem-KX-Cl} and of their Serre duals; 
we prove the required result in Lemma~\ref{lemma-exceptional-stable} below after 
some preliminaries.  
\begin{lemma}
\label{lemma-serre-Y-Cl0}
The category $\Db(Y, \Cl_0)$ has a Serre functor given by
$S_{\Cl_0}(-) = - \otimes_{\Cl_0} \Cl_1(-h)[3]$. 
\end{lemma} 

\begin{proof}
This follows by the same argument as in Lemma~\ref{lemma-serre-Cl0}\eqref{S_Cl_0}. 
\end{proof} 

\begin{remark}
The isomorphism $\Cl_1 \otimes_{\Cl_0} \Cl_1 \cong \Cl_0$, given by \cite[Corollary 3.9]{kuznetsov08quadrics}, is sometimes useful for computing the action of $S_{\Cl_0}$, e.g. $S_{\Cl_0}(\Cl_1) \cong \Cl_0(-h)[3]$. 
\end{remark}

\begin{lemma}
\label{lemma-rank-YCl0}
The rank of any object of $\Db(Y, \Cl_0)$ is divisible by $4$.
\end{lemma}

\begin{proof}
This follows from Lemma~\ref{lemma_rankCl0modules} by considering the restriction of objects to a generic hyperplane section. 
\end{proof}

We will only be concerned with the terms of $\ch_{\Cl_0}$ of degree at most $2$; 
we denote by $\ch_{\Cl_0, \leq 2} = \ch_{\Cl_0, 0} + \ch_{\Cl_0, 1} + \ch_{\Cl_0, 2}$ the sum of these 
terms. 

\begin{lemma}
\label{lemma-chleq2}
The objects 
\begin{equation}
\label{exceptional-list-sheaves}
\Cl_0, \, \Cl_1, \, \cR_c, \, \cR_d, \, \Cl_1(-h), \, \Cl_0(-h), \, \cR_c \otimes_{\Cl_0} \Cl_1(-h), \, \cR_d \otimes_{\Cl_0} \Cl_1(-h) 
\end{equation} 
are slope stable $\Cl_0$-modules, with truncated Chern characters given by 
\begin{alignat*}{2}
& \ch_{\Cl_0, \leq 2}(\Cl_0)  =  \ch_{\Cl_0, \leq 2}(\Cl_1)  = 4 - 4h + 2h^2 ,  \\ 
& \ch_{\Cl_0, \leq 2}(\cR_c)  = \ch_{\Cl_0, \leq 2}(\cR_d)  = 4 - 2h + \frac{1}{2}h^2 , \\ 
& \ch_{\Cl_0, \leq 2}(\Cl_1(-h))  = \ch_{\Cl_0, \leq 2}(\Cl_0(-h))  = 4 - 8h + 8h^2 , \\ 
& \ch_{\Cl_0, \leq 2}(\cR_c \otimes_{\Cl_0} \Cl_1(-h))  = \ch_{\Cl_0, \leq 2}(\cR_d \otimes_{\Cl_0} \Cl_1(-h))  = 4 - 6h + \frac{9}{2}h^2. 
\end{alignat*} 
In particular, the discriminant $\Delta_{\Cl_0}$ vanishes on all of these objects. 
\end{lemma}

\begin{proof}
A direct computation gives the truncated Chern characters of all the objects except for $\cR_c \otimes_{\Cl_0} \Cl_1(-h)$ and $\cR_d \otimes_{\Cl_0} \Cl_1(-h)$. For these, we note that for any object $E \in \Db(Y, \Cl_0)$, the truncated Chern character $\ch_{\Cl_0, \leq 2}(E)$ is determined by the Chern character of the derived restriction of $E$ to $\Sigma$, where $\Sigma \subset Y$ is a smooth hyperplane section; in particular, by  Lemma~\ref{lemma-serre-Y-Cl0}\eqref{Cl1-preserve-ch} we find $\ch_{\Cl_0, \leq 2}(E \otimes_{\Cl_0} \Cl_1(-h)) = \ch_{\Cl_0, \leq 2}(E(-h))$. Using this observation, the computation for $\cR_c \otimes_{\Cl_0} \Cl_1(-h)$ and $\cR_d \otimes_{\Cl_0} \Cl_1(-h)$ is straightforward. 

Finally, the stability of the objects~\eqref{exceptional-list-sheaves} follows because they are torsion free and, by Lemma~\ref{lemma-rank-YCl0}, of minimal rank. 
\end{proof}

\begin{lemma}
\label{lemma-exceptional-stable} 
For $-\frac{3}{2} \leq \beta < -1$ the objects 
\begin{equation}
\label{exceptional-list}
\Cl_0, \, \Cl_1, \, \cR_c, \, \cR_d, \, 
S_{\Cl_0}(\Cl_0)[-2], \, S_{\Cl_0}(\Cl_1)[-2], \, S_{\Cl_0}(\cR_c)[-2], \, S_{\Cl_0}(\cR_d)[-2], 
\end{equation} 
are contained in $\Coh^{\beta}(Y, \Cl_0)$, and for $\alpha >0$ they are $\sigma_{\alpha,\beta}$-stable. 
\end{lemma}

\begin{proof}
Note that by Lemma~\ref{lemma-serre-Y-Cl0}, the last four objects in~\eqref{exceptional-list} are the sheaves 
\begin{equation*}
\Cl_1(-h)[1], \, \Cl_0(-h)[1], \, \cR_c \otimes_{\Cl_0} \Cl_1(-h)[1], \, \cR_d \otimes_{\Cl_0} \Cl_1(-h)[1]
\end{equation*} 
 and by Lemma~\ref{lemma-chleq2} we have 
\begin{align*}
-2=\mu_h(\Cl_1(-h)) = \mu_h(\Cl_0(-h))\  
& < \mu_h(\cR_c \otimes_{\Cl_0} \Cl_1(-h)) = \mu_h(\cR_d \otimes_{\Cl_0} \Cl_1(-h)) = -\frac{3}{2} \\ 
-1 = \mu_h(\Cl_0) =  \mu_h(\Cl_1) & < \mu_h(\cR_c) = \mu_h(\cR_d)= -\frac{1}{2} . 
\end{align*}  
Therefore, by the slope stability of the objects~\eqref{exceptional-list-sheaves} from Lemma~\ref{lemma-chleq2}, 
the objects~\eqref{exceptional-list} are contained in $\Coh^{\beta}(Y, \Cl_0)$ as claimed. 

Note that the $\cO_Y$-modules underlying the objects $\Cl_0, \Cl_1, \cR_c, \cR_d, \Cl_1(-h), \Cl_0(-h)$ are vector bundles, by the definitions~\eqref{Cl_0} and~\eqref{Cl_1} and Lemma~\ref{lemma-Ra-Rb}. 
Since by Lemma~\ref{lemma-chleq2} they also have discriminant $\Delta_{\Cl_0}=0$, 
we conclude by \cite[Proposition 7.4.1]{BMT:3folds-BG} that the objects 
$\Cl_0, \Cl_1, \cR_c, \cR_d, \Cl_1(-h)[1], \Cl_0(-h)[1]$ are $\sigma_{\alpha,\beta}$-stable for $\alpha > 0$, as claimed. 

By the same argument, to show that $\cR_c \otimes_{\Cl_0} \Cl_1(-h)[1]$ and $\cR_d \otimes_{\Cl_0} \Cl_1(-h)[1]$ are $\sigma_{\alpha,\beta}$-stable for $\alpha > 0$, 
it suffices to show that the $\cO_Y$-modules underlying $\cR_c \otimes_{\Cl_0} \Cl_1$ and $\cR_d \otimes_{\Cl_0} \Cl_1$ are vector bundles. 
In fact, we claim that tensoring by $\Cl_1$ over $\Cl_0$ either fixes $\cR_c$ and $\cR_d$ or swaps them. 
Indeed, tensoring by $\Cl_1$ the semiorthogonal decomposition of Theorem \ref{theorem-KX-Cl}, we get
\begin{equation*}
\Db(Y, \Cl_0) = 
\llangle \Psi(\Ku(X)), \Cl_0, \Cl_1, \cR_c \otimes_{\Cl_0} \Cl_1, \cR_d \otimes_{\Cl_0} \Cl_1 \rrangle    . 
\end{equation*}
Since $\Cl_0, \Cl_1$ are completely orthogonal, comparing the two semiorthogonal decompositions we get $\llangle \cR_c,\cR_d \rrangle=\llangle \cR_c \otimes_{\Cl_0} \Cl_1, \cR_d \otimes_{\Cl_0} \Cl_1 \rrangle$. This implies our claim and finishes the proof of the lemma.
\end{proof}

We need to consider one further tilt of the weak stability condition $\sigma_{\alpha, \beta}$. 
Let $\Coh^{0}_{\alpha, \beta}(Y, \Cl_0)$ be the heart of a bounded t-structure obtained by tilting $\Coh^{\beta}(Y, \Cl_0)$ with respect to $\sigma_{\alpha, \beta}$-stability at the slope $\mu = 0$. 

\begin{lemma}[{\cite[Proposition 2.15]{BLMS}}]
For any $(\alpha, \beta)\in \bR_{>0}\times \bR$, the pair 
\begin{equation*}
\sigma_{\alpha, \beta}^0 = (\Coh^{0}_{\alpha, \beta}(Y, \Cl_0), Z_{\alpha, \beta}^0) 
\end{equation*}
with $Z^0_{\alpha, \beta} = -\sqrt{-1}  \cdot Z_{\alpha,\beta}$ defines a weak stability condition on $\Db(Y, \Cl_0)$ with respect to $\Lambda_{\Cl_0}$. 
\end{lemma} 

\begin{lemma}
\label{lemma-double-tilt}
For $\beta = -\frac{5}{4}$ and $0 < \alpha < \frac{1}{4}$, the weak stability condition 
$\sigma_{\alpha, \beta}^0$ satisfies the following properties: 
\begin{enumerate}
\item \label{Ei-double-tilt}
$\Cl_0, \Cl_1, \cR_c, \cR_d \in \Coh_{\alpha, \beta}^0(Y, \Cl_0)$. 
\item \label{SEi-double-tilt}
$S_{\Cl_0}(\Cl_0), S_{\Cl_0}(\Cl_1), S_{\Cl_0}(\cR_c), S_{\Cl_0}(\cR_d) \in \Coh_{\alpha, \beta}^0(Y, \Cl_0)[1]$. 
\item \label{Z0-nonzero}
$Z^0_{\alpha, \beta}(\Cl_0), Z^0_{\alpha, \beta}(\Cl_1), Z^0_{\alpha, \beta}(\cR_c) , Z^0_{\alpha, \beta}(\cR_d)$ are all nonzero. 
\end{enumerate}
\end{lemma} 

\begin{proof}
A direct computation using Lemma~\ref{lemma-chleq2} shows
\begin{align*}
\mu_{\alpha,\beta}(S_{\Cl_0}\Cl_0[-2]) =\mu_{\alpha,\beta}(S_{\Cl_0}\Cl_1[-2]) & <\mu_{\alpha,\beta}(S_{\Cl_0}\cR_c[-2])=\mu_{\alpha,\beta}(S_{\Cl_0}\cR_d[-2]) < 0 \\
0 < \mu_{\alpha,\beta}(\Cl_0)=\mu_{\alpha,\beta}(\Cl_1) & < \mu_{\alpha,\beta}(\cR_c)=\mu_{\alpha,\beta}(\cR_d).
\end{align*}
In particular~\eqref{Z0-nonzero} holds, 
and~\eqref{Ei-double-tilt} and~\eqref{SEi-double-tilt} follow from Lemma~\ref{lemma-exceptional-stable}. 
\end{proof} 

Finally, we are ready to produce stability conditions on $\Ku(X)$. 
Note that there is a homomorphism 
\begin{equation*}
\rK(\Ku(X)) \to \rK(Y, \Cl_0) \to \Lambda_{\Cl_0} 
\end{equation*}
where the first arrow is the injection on Grothendieck groups induced by the embedding $\Psi \colon \Ku(X) \to \Db(Y, \Cl_0)$ from Theorem~\ref{theorem-KX-Cl} and the second arrow is the canonical surjection. 
We define $\Lambda_{\Cl_0, \Ku(X)}$ to be the image of this homomorphism, so that there is a surjection 
\begin{equation}
\label{LambdaCl0Ku}
\rK(\Ku(X)) \twoheadrightarrow \Lambda_{\Cl_0, \Ku(X)} . 
\end{equation} 

\begin{theorem}
\label{theorem-stab-Ku-good-4fold}
Let $X$ be an ordinary GM fourfold with smooth canonical quadric. 
Then $\Ku(X)$ has a Bridgeland stability condition with respect to the lattice $\Lambda_{\Cl_0, \Ku(X)}$. 
\end{theorem} 

\begin{proof}
By Theorem~\ref{theorem-KX-Cl} we have a semiorthogonal decomposition 
\begin{equation*}
\Db(Y,\Cl_0)= \llangle \Psi(\Ku(X)), \Cl_1, \Cl_0, \cR_c, \cR_d \rrangle . 
\end{equation*} 
We claim that any weak stability condition $\sigma^{0}_{\alpha, \beta}$ as in Lemma~\ref{lemma-double-tilt} induces a stability condition on $\Ku(X)$ with respect to $\Lambda_{\Cl_0, \Ku(X)}$, with heart given by $\Psi^{-1}(\Ku(X) \cap \Coh_{\alpha, \beta}^0(Y, \Cl_0))$  
and central charge $Z_{\alpha, \beta}^0 \circ \Psi \colon \rK(\Ku(X)) \to \bC$. 
Indeed, it suffices to apply \cite[Proposition 5.1]{BLMS}. 
The hypotheses of the cited proposition are satisfied due to Lemma~\ref{lemma-double-tilt} and the following observation: 
if $0 \neq F \in \Coh_{\alpha, \beta}^0(Y, \Cl_0)$ and $Z_{\alpha, \beta}^0(F) = 0$, then $\Forg(F)$ is a torsion sheaf with $0$-dimensional support, 
hence $\Hom_{\Cl_0}(\Cl_0, F) = \Hom(\cO_Y, \Forg(F)) \neq 0$ and in particular $F \notin \Psi(\Ku(X))$. 
\end{proof}

\subsection{Full numerical stability conditions} 
\label{section-support-fourfold}

In this section, we show the stability conditions constructed in Theorem~\ref{theorem-stab-Ku-good-4fold} 
satisfy the support property with respect to a larger lattice, the numerical Grothendieck group of $\Ku(X)$. 
This can be formulated in terms of the Mukai Hodge structure of $\Ku(X)$, which we review first. 
For this section, we work over $k = \bC$. 

\subsubsection{The Mukai Hodge structure} 
\label{section-Mukai-HS}
Following Addington--Thomas \cite{addington-thomas}, for any GM fourfold or sixfold $X$ we define the abelian subgroup
$$\tH(\Ku(X),\bZ):= \lbrace \kappa \in \rK_{\text{top}}(X) \st \chi([\cO_X(i)], \kappa)= \chi([\cU_X^\vee], \kappa)=0, \, \forall i=0, \dots, \dim (X)-3  \rbrace$$
of the topological K-theory $\rK_{\text{top}}(X)$ of $X$, where $\chi$ denotes the Euler pairing. It is equipped with a nondegenerate symmetric pairing $(-,-):=-\chi(-,-)$ and there is a canonical homomorphism 
\begin{equation*} 
v \colon \rK(\Ku(X)) \to \tH(\Ku(X), \bZ) , 
\end{equation*} 
called the Mukai vector. 
By pulling back the Hodge structure on the cohomology ring of $X$, $\tH(\Ku(X), \bZ)$ can be endowed with a weight $2$ Hodge structure with Hodge numbers  
\begin{equation*}
h^{2,0} = 1, ~ h^{1,1} = 22, ~ h^{0,2} = 1. 
\end{equation*} 
In case $\Ku(X) \simeq \Db(S)$ for a K3 surface $S$, then $\tH(\Ku(X), \bZ)$ is isomorphic to the 
usual Mukai Hodge structure of $S$. 
The Mukai Hodge structure $\tH(\Ku(X), \bZ)$ was defined and studied for a GM fourfold in \cite{Pert}; more generally, in \cite{IHC-CY2} a Hodge structure is associated to any (suitably enhanced) category, which agrees with the Mukai Hodge structure in the case of $\Ku(X)$. 

The Mukai vector $v$ factors through the lattice of integral Hodge classes 
\begin{equation*} 
\tH^{1,1}(\Ku(X), \bZ) \subset \tH(\Ku(X), \bZ) . 
\end{equation*} 
In fact, we have the following result, which should be thought of as asserting a version of the integral Hodge conjecture for $\Ku(X)$. 

\begin{theorem}[\cite{IHC-CY2}]
Let $X$ be a GM fourfold or sixfold. Then $v \colon \rK(\Ku(X)) \to \tH(\Ku(X), \bZ)$ induces an isomorphism of lattices $\cN(\Ku(X))(-1) \cong \tH^{1,1}(\Ku(X), \bZ)$, where $\cN(\Ku(X))(-1)$ denotes the numerical Grothendieck group of $\Ku(X)$ equipped with the pairing given by the negative of the Euler form. 
\end{theorem}

Now let $X$ be a GM fourfold. 
We recall a relation which we will need between $\tH(\Ku(X), \bZ)$ and the usual 
Hodge structure on $X$. 
By \cite[Lemma 2.27]{KuzPerry:dercatGM} there are two classes 
$\lambda_1$ and $\lambda_2$ in $\tH^{1,1}(\Ku(X), \bZ)$ generating a canonical sublattice with intersection form
\begin{equation*}
A_1^{\oplus 2} =\begin{pmatrix}
 2 & 0 \\ 
 0 & 2
 \end{pmatrix}.
\end{equation*} 
By \cite[Equation (4)]{Pert}, the Chern characters of $\lambda_1$ and $\lambda_2$ in $\text{CH}(X) \otimes \bQ$ are
\begin{equation}
\label{deflambda12}    
\ch(\lambda_1)=-2 + \gamma_X^*\sigma_{1,1} -\frac{1}{2} \quad \text{and} \quad \ch(\lambda_2)=-4 + 2H -\frac{1}{6}H^3
\end{equation}
where $\gamma_X^*\sigma_{1,1}$ is the pullback to $X$ of the Schubert cycle $\sigma_{1,1}$ along the canonical morphism $\gamma_X \colon X \to \Gr(2,V_5)$. We denote by $\tH(\Ku(X), \bZ)_{0}$ the orthogonal to $A_1^{\oplus 2} \subset \tH(\Ku(X), \bZ)$. 
There is also a related Hodge structure, 
the vanishing cohomology $\rH^4(X, \bZ)_{0}$, defined as the orthogonal to the sublattice $\gamma_X^*\rH^4(\Gr(2,5), \bZ) \subset \rH^4(X, \bZ)$ 
with respect to the intersection pairing. 

\begin{proposition}[{\cite[Proposition 3.1]{Pert}}]
\label{proposition-HKu-HX}
There is an isometry of weight $2$ Hodge structures 
\begin{equation*} 
\tH(\Ku(X), \bZ)_0 \cong \rH^4(X, \bZ)_{0}(1) , 
\end{equation*} 
where $(1)$ on the right side denotes a Tate twist. 
The isometry is induced by the second Chern class $c_2 \colon \tH(\Ku(X), \bZ) \to \rH^4(X, \bZ)$, 
which is also equal to the full Chern character $\ch$ on $\tH(\Ku(X), \bZ)_0$. 
\end{proposition} 

\subsubsection{The support property} 
Now we turn to the support property for the stability conditions constructed in Theorem~\ref{theorem-stab-Ku-good-4fold}. 

\begin{definition}
Let $X$ be a GM fourfold or sixfold. 
A \emph{full numerical stability condition on $\Ku(X)$} is a Bridgeland stability condition on $\Ku(X)$ with respect to the lattice $\wtilde{\rH}^{1,1}(\Ku(X),\bZ)$ and the surjective morphism $v: \rK(\Ku(X)) \twoheadrightarrow \tH^{1,1}(\Ku(X),\bZ)$ induced by the Mukai vector.
\end{definition} 

Let $X$ be a GM fourfold with smooth canonical quadric. 
Note that the homomorphism~\eqref{LambdaCl0Ku} factors as 
\begin{equation*}
\rK(\Ku(X)) \xrightarrow{\, v \,}  \cN(\Ku(X)) \xrightarrow{\, u \,} \Lambda_{\Cl_0, \Ku(X)}, 
\end{equation*} 
where $u$ is the surjection given by the map induced by $\Psi \colon \Ku(X) \to \Db(Y, \Cl_0)$ on numerical Grothendieck groups followed by the projection to $\Lambda_{\Cl_0}$. 

\begin{proposition}
\label{proposition-full-numerical}
The stability conditions constructed in  Theorem~\ref{theorem-stab-Ku-good-4fold} are 
full numerical stability conditions on $\Ku(X)$. 
\end{proposition} 

Let $\sigma$ be a stability condition on $\Ku(X)$ 
constructed in Theorem~\ref{theorem-stab-Ku-good-4fold}, with central charge $Z \colon \Lambda_{\Cl_0, \Ku(X)} \to \bC$. 
Define $\eta(\sigma) \in \tH^{1,1}(\Ku(X), \bC)$ to be the element in the complexification of the Mukai Hodge structure such that 
\begin{equation*}
(Z \circ u)(-) = (\eta(\sigma), -) . 
\end{equation*} 
Recall the subsets $\cP_0(\Ku(X)) \subset \cP(\Ku(X)) \subset \tH^{1,1}(\Ku(X), \bC)$ defined in~\eqref{cP0}. 
By \cite[Lemma 9.7]{BLMS} (or rather its direct analog in our setup), 
if $\eta(\sigma)$ is in $\cP_0(\Ku(X))$, 
then $\sigma$ is a full numerical stability condition on $\Ku(X)$. 
Therefore, Proposition~\ref{proposition-full-numerical} 
is a consequence of the next lemma, similar to \cite[Proposition 9.10]{BLMS}. 

\begin{lemma}
\label{lemma-eta-sigma}
If $\sigma$ is a stability condition on $\Ku(X)$ as
constructed in Theorem~\ref{theorem-stab-Ku-good-4fold}, then 
\begin{equation*}
\eta(\sigma) \in (A_1^{\oplus 2} \otimes \bC) \cap \cP(\Ku(X)) \subset \cP_0(\Ku(X)). 
\end{equation*}
\end{lemma} 

\begin{proof}
Let $V \subset \wtilde{\rH}^{1,1}(X,\bR)$ be the subspace generated by the real and imaginary parts of $\eta(\sigma)$. 
We will show that $V = A_1^{\oplus 2}$, which in particular means $\eta(\sigma) \in \cP(\Ku(X))$. 

First we claim that $V$ has real dimension $2$. 
By the construction of Theorem~\ref{theorem-stab-Ku-good-4fold}, the central charge of $\sigma$ is given by $Z_\alpha:=Z^0_{\alpha,-\frac{5}{4}} \circ \Psi$ for some $0 < \alpha < \frac{1}{4}$. 
A computation using \eqref{deflambda12} shows 
\begin{equation*} 
\ch_{\Cl_0,\leq 2}(\Psi(\lambda_1))= 8 + 2h -8h^2 \quad \text{and} \quad \ch_{\Cl_0,\leq 2}(\Psi(\lambda_2))= -8 -4h +h^2, 
\end{equation*} 
and therefore 
\begin{equation*} 
Z_\alpha({\lambda}_1)=24+\sqrt{-1} \left(-8\alpha^2+\frac{3}{2} \right) \quad \text{and} \quad Z_\alpha({\lambda}_2)=-28+\sqrt{-1} \left(8\alpha^2-\frac{29}{2} \right).
\end{equation*} 
Since $Z_\alpha({\lambda}_1)$ and $Z_\alpha({\lambda}_2)$ are linearly independent, we deduce our claim. 

To show $V = A_1^{\oplus 2}$, it remains to show 
$\eta(\sigma) \in (A_1^{\oplus 2})_{\bC}$. 
This is equivalent to showing that for any $F \in \Ku(X)$ with $v(F) \in (A_1^{\oplus 2})^\perp \subset \tH(\Ku(X), \bZ)$, 
we have $Z_{\alpha}(F) = 0$. 
By definition, $Z_{\alpha}$  only depends on $\ch_i(\Psi(F))$ for $0 \leq i \leq 2$, so 
if $i_{\Sigma} \colon \Sigma \hookrightarrow Y$ denotes a smooth hyperplane section, 
it suffices to show the Chern classes of $i_{\Sigma*}i^*_{\Sigma} \Psi(F)$ vanish. 

Recall that by definition~\eqref{Psi} we have $\Psi = \Xi \circ \rL_{\cO_{\tX}(-h)} \circ b^* \circ \rL_{\cU_X} \colon \Ku(X) \to \Db(Y, \Cl_0)$. 
As in~\eqref{diagram-Z}, let $i_{Z} \colon Z \hookrightarrow \tX$ be the preimage of $\Sigma$ under the conic fibration $\pi \colon \tX \to Y$. Then using the definition~\eqref{Xi} of $\Xi \colon \Db(\tX) \to \Db(Y, \Cl_0)$, base change, and the projection formula, we find 
\begin{equation*}
i_{\Sigma*} \circ i^*_{\Sigma} \circ \Xi \simeq \Xi \circ i_{Z*} \circ i_{Z}^* , 
\end{equation*} 
and thus 
\begin{equation*}
i_{\Sigma*}i^*_{\Sigma} \Psi(F) \cong \Xi (i_{Z*}i_Z^* \rL_{\cO_{\tX}(-h)}  b^*  \rL_{\cU_X}(F)). 
\end{equation*} 
Therefore, to prove the Chern classes of $i_{\Sigma*}i^*_{\Sigma} \Psi(F)$ vanish, it suffices 
to show the class 
\begin{equation}
\label{iZiZmutationF}
i_{Z*}i_Z^* \rL_{\cO_{\tX}(-h)}  b^*  \rL_{\cU_X}(F)
\end{equation}
vanishes in the numerical Grothendieck group of $\tX$. 

First we claim that the product of the $\ch(b^*F)$ with any positive power of the classes $H, \gamma_X^* \sigma_{1,1}, h, E$ vanishes, where $\gamma_X^*\sigma_{1,1}$ is the pullback of the Schubert cycle $\sigma_{1,1}$ via the canonical map $\gamma_X \colon X \to \Gr(2,V_5)$. 
Indeed, by Proposition~\ref{proposition-HKu-HX} we have $\ch(b^*F) = b^*\ch_2(F)$ and $\ch_2(F)$ is orthogonal 
to any power of $H, \gamma_X^* \sigma_{1,1}$. 
Since $h = H - E$, it remains to show that $b^*\ch_2(F)$ is orthogonal to any power of $E$. 
We have 
\begin{equation*}
E \cdot b^*\ch_2(F)= i_{E*}(b_{E}^* i_{T}^*\ch_2(F)) ,  
\end{equation*} 
where $i_E, b_E$, and $i_T$ are as in diagram~\eqref{main-diagram}; this vanishes because $i_T^* \ch_2(F) = 0$, 
as the class of $T$ is a combination of $H^2$ and $\gamma_X^*\sigma_{1,1}$. 
Similarly, we have 
\begin{equation*}
E^2 \cdot b^*\ch_2(F)=b_*E^2 \cdot \ch_2(F)=(2 \gamma_X^*\sigma_{1,1}-H^2)\cdot \ch_2(F)=0. 
\end{equation*} 

Next we claim that $\rL_{\cO_{\tX}(-h)}  b^*  \rL_{\cU_X}(F) \cong \rL_{\cO_{\tX}(-h)} \rL_{\cU_{\tX}} b^*F $ has the same class as $b^*F$ in the Grothendieck group of $\tX$. 
Indeed, it is enough to show that 
\begin{equation*}
\chi(\cU_{\tX}, b^*F) = \chi(\cO_{\tX}(-h), b^*F) = 0, 
\end{equation*} 
which follows from Hirzebruch--Riemann--Roch and the previous paragraph,   
because all of $\ch(\cU_{\tX}), \ch(\cO_{\tX}(-h)),$ and $\td(\tX)$ are combinations of powers of $H, \gamma_X^* \sigma_{1,1}, h, E$. 

Now we can prove \eqref{iZiZmutationF} vanishes in the numerical Grothendieck. 
By the above, it is enough to show the class of $i_{Z*}i_Z^*b^*F$, which equals $b^*F - b^*F(-h)$, vanishes; 
this holds because $\ch(b^*F) = \ch(b^*F(-h))$ by the orthogonality of $\ch(b^*F)$ to powers of $h$. 

Finally, by \cite[Theorem 1.3]{OG} and \cite[Theorem 5.1]{DebKuz:periodGM}, the image of the period map of smooth GM fourfolds does not intersect the divisor $\cD_8$, i.e.\ there do not exist classes of square $-2$ in $\rH^4(X,\bZ)_{0}(1)$. We conclude that  $(A_1^{\oplus 2})_\bC \cap \cP \subset \cP_0$, as stated.
\end{proof}

\subsection{The general case}
\label{section-main-theorem-general}

The key ingredient in reducing Theorem~\ref{main-theorem} for arbitrary GM varieties to the case of 
GM fourfolds considered above is the duality conjecture proved in \cite{KuzPerry:cones}, 
which gives equivalences between Kuznetsov components of GM varieties of varying dimensions. 
In particular, we have the following. 

\begin{theorem}
\label{theorem-duality} 
\begin{enumerate}
\item \label{GM4-dual}
If $X$ is a GM fourfold or sixfold, then there exists an ordinary GM fourfold $X'$ with smooth canonical quadric and an equivalence $\Ku(X) \simeq \Ku(X')$. 
\item \label{GM5-dual}
If $X$ is a GM fivefold, then there exists an ordinary GM threefold $X'$ and an equivalence $\Ku(X) \simeq \Ku(X')$. 
\end{enumerate}
\end{theorem}

\begin{proof}
The description of generalized partners and duals in \cite[Lemma 3.8]{KuzPerry:dercatGM} together with the duality conjecture \cite[Conjecture 3.7]{KuzPerry:dercatGM} proved in \cite[Theorem 1.6]{KuzPerry:cones} shows that claim~\eqref{GM5-dual} holds, and that~\eqref{GM4-dual} holds modulo the condition that $X'$ has smooth canonical quadric. 

We will prove the full claim~\eqref{GM4-dual} by a careful choice of a generalized partner $X'$ for $X$.  
To explain this, we freely use the notation and terminology on EPW sextics 
introduced in \cite[\S 3]{DebKuz:birGM} (see also \cite[\S 3]{KuzPerry:dercatGM}). 
First we observe that by \cite[Proposition 4.5]{DebKuz:birGM} a smooth ordinary GM fourfold $X$ as in~\eqref{GM4fold} has smooth canonical quadric if and only if the point $\Sigma_1(X) := \bP(V_1) \in \bP(V_5)$ (where $V_1$ is as defined in Section~\ref{subsection-conicfibr}) lies in the EPW stratum $Y^1_{\sA(X)}$. 
In turn, by \cite[Lemma 2.1, Remark B.4, and Equation (3) of Section 2.3]{DebKuz:periodGM}, this condition on $\Sigma_1(X)$ holds if the Pl\"{u}cker point $\mathbf{p}_X$ of $X$ does not lie in the projective dual $(Y_{\sA(X)}^{\geq 2})^{\vee}$ of the EPW stratum $Y_{\sA(X)}^{\geq 2}$. 

Now let $X$ be any GM fourfold or sixfold. 
Choose a point $\mathbf{p} \in Y_{\sA(X)^{\perp}}^1$ in the top stratum of the dual EPW sextic 
which does not lie in $(Y_{\sA(X)}^{\geq 2})^{\vee}$. 
Let $X'$ be the ordinary GM fourfold corresponding to the pair $(\sA(X), \mathbf{p})$ (see \cite[Theorem 3.10]{DebKuz:birGM} or \cite[Theorem 3.1]{KuzPerry:dercatGM}). 
Then by construction $X'$ is a period partner of $X$ whose Pl\"{u}cker point $\mathbf{p}_{X'}$ does not lie in $(Y_{\sA(X')}^{\geq 2})^{\vee}$, and hence $X'$ has smooth canonical quadric by the previous paragraph. 
This finishes the proof, since there is an equivalence $\Ku(X) \simeq \Ku(X')$ by the duality conjecture (\cite[Theorem 1.6]{KuzPerry:cones}).  
\end{proof}

\begin{proof}[Proof of Theorem \ref{main-theorem}]
By Theorem~\ref{theorem-duality}, 
Theorem~\ref{main-theorem} is an immediate consequence of Theorem~\ref{theorem-stab-Ku-good-4fold} and the construction of stability conditions on Kuznetsov components of GM threefolds \cite[Theorem 6.9]{BLMS}. 
Similarly, it follows from Proposition~\ref{proposition-full-numerical} that over $\bC$, this gives full numerical stability conditions on the Kuznetsov components of GM sixfolds.
\end{proof}
 
Finally, we note that as a consequence of these results and \cite{CollPoli}, similarly to \cite[Proposition 5.13]{BLMS}, we obtain Corollary~\ref{corollary-DbX-stability} from the introduction:  

\begin{corollary}
Let $X$ be a GM variety over $k$. 
Then the category $\Db(X)$ has a Bridgeland stability condition. 
Moreover, if $k = \bC$, then Bridgeland stability conditions exist which satisfy the support property with respect to the image of the Chern character in $\rH^*(X,\bQ)$. 
\end{corollary}
\begin{proof} 
Let $\sigma=(\cA,Z)$ be a stability condition on $\Ku(X)$. By \cite[Proposition 3.3]{CollPoli}, as explained for instance in \cite[Corollary 3.8]{BMMS:categoricalcubic3}, if for every exceptional object $E$ in the semiorthogonal decomposition \eqref{equation-Ku}, there exists an index $j$ such that 
$\Hom^{\leq j}(A,E)=0$ for every $A \in \cA$, then 
a stability condition exists on $\Db(X)$. 
In order to check this condition, we denote by $i: \Ku(X) \to \Db(X)$ the inclusion functor of $\Ku(X)$ in $\Db(X)$ and by $i^!$ its right adjoint, which exists since $\Ku(X)$ is admissible. By adjunction, we have
\begin{equation*}
\Hom^k(A,E) = \Hom^k(i(A),E)= \Hom^k_{\Ku(X)}(A,i^!E).
\end{equation*}
Since $\cA$ is the heart of a bounded t-structure, $i^!E$ has a finite number of cohomology objects $A_1[k_1], \dots ,A_m[k_m]$ with $A_s[k_s] \in \cA[k_s]$ for $1 \leq s \leq m$ and $k_1 > \cdots >k_m$. Then we can choose $j<-k_1$, so that
$\Hom(A,A_s[k_s+k])=0$ for every $1 \leq s \leq m$ and $k \leq j$. This gives the desired vanishing and implies the statement.
\end{proof}


\section{Applications}
\label{section_applications}
In this section, we prove the results stated in Sections~\ref{intro-stab-manifold}, \ref{intro-moduli}, and \ref{intro-associated-K3}. 
The proofs are deformation theoretic and require understanding Kuznetsov components of GM varieties and stability conditions on them in families. 
Thus we start in Section~\ref{families-Ku} by defining a version of the Kuznetsov component for a family of GM varieties, and proving a relative version of the key derived category result 
Theorem~\ref{theorem-KX-Cl}
from our construction of stability conditions on a fixed Kuznetsov component. 
In Section~\ref{section-stab-over-curve}, which contains the bulk of our work, we construct well-behaved relative stability conditions on the Kuznetsov components of GM fourfolds over a curve, where a special fiber is equivalent to the derived category of a twisted K3 surface. 
In Section~\ref{section-proofs-applications}, we combine this with the arguments of \cite[Part VI]{stability-families} to prove the promised applications. 
Finally, in Section~\ref{subsection_examples} we discuss some low-dimensional examples of the hyperk\"{a}hler varieties given by Theorem~\ref{theorem-unirational-families}. 

We assume $k = \bC$ is the complex numbers throughout this section. 

\subsection{Families of Kuznetsov components} 
\label{families-Ku}
In this section, we discuss a family version of Theorem~\ref{theorem-KX-Cl}, 
which gave an embedding of the Kuznetsov 
component of an ordinary GM fourfold with smooth canonical quadric into the 
derived category of Clifford modules on a quadric threefold. 

First, we need to define the Kuznetsov component in families. 
By a family of GM varieties over a base scheme $S$, 
we mean a smooth proper morphism $f \colon \cX \to S$ equipped with a line bundle on 
$\cO_{\cX}(1)$ on $\cX$,  
such that for every geometric point $s \in S$ the pair $(\cX_s, \cO_{\cX_s}(1))$ is a polarized GM variety, 
i.e. $\cX_s$ is isomorphic to an intersection as in Definition~\ref{definition-GM} and $\cO_{\cX_s}(1)$ corresponds 
to the Pl\"{u}cker line bundle.
\begin{remark}
There is a slightly more general notion of a family of GM varieties, 
where instead of a line bundle $\cO_{\cX}(1)$ one considers an element of 
$\Pic_{\cX/S}(S)$. 
This definition is better suited for defining the moduli stack of GM varieties, see \cite{KuzPerry:dercatGM, DebKuz:moduli}, 
but in practice there is little difference, since \'{e}tale locally on $S$ any element of $\Pic_{\cX/S}(S)$ is a line bundle.
\end{remark} 

The results of \cite{DebKuz:birGM} show that for any family of GM varieties, there is a canonically determined 
rank $5$ vector bundle $\cV_{5}$ on $S$ and a morphism $\cX \to \Gr_S(2, \cV_5)$ which 
fiberwise recovers the usual map from $\cX_s$ to $\Gr(2,5)$.  
We denote by $\cU_{\cX}$ the pullback to $\cX$ of the tautological rank $2$ subbundle on $\Gr_S(2, \cV_5)$. 

In the following result we use the notion 
of a strong $S$-linear semiorthogonal decomposition of finite cohomological amplitude, 
and the base change of such a decomposition along a point $s \in S$, 
see \cite[Section 3]{stability-families}. 
Note that for any integer $i \in \bZ$, the objects $\cO_{\cX}(i), \cU_{\cX}(i)$ are relative exceptional objects on $\cX$, 
i.e. they are fiberwise exceptional objects. 
In general, if $E \in \Db(\cX)$ is a perfect complex which is relatively exceptional, then 
the functor $\Db(S) \to \Db(\cX)$ given by $F \mapsto f^*(F) \otimes E$ is fully faithful with image 
$f^*(\Db(S)) \otimes E \subset \Db(\cX)$ an admissible $S$-linear subcategory \cite[Lemma 3.23]{stability-families}. 

\begin{lemma}
Let $f \colon \cX \to S$ be a family of $n$-dimensional GM varieties over a noetherian base scheme $S$. 
Then there is an admissible $S$-linear subcategory $\Ku(\cX) \subset \Db(\cX)$ and a strong $S$-linear semiorthogonal 
decomposition of finite cohomological amplitude 
\begin{multline*}
\Db(\cX)  = \langle \Ku(\cX), f^*(\Db(S)) \otimes \cO_{\cX}, f^*(\Db(S)) \otimes \cU_{\cX}^{\vee},  
\dots \\ 
\dots, f^*(\Db(S)) \otimes \cO_{\cX}(n-3), f^*(\Db(S)) \otimes \cU_{\cX}^{\vee}(n-3) \rangle . 
\end{multline*} 
Moreover, if $S$ has affine diagonal, then for any point $s \in S$ the base change of 
$\Ku(\cX)$ satisfies $\Ku(\cX)_s \simeq \Ku(\cX_s)$ where the right side is defined by~\eqref{equation-Ku}. 
\end{lemma} 

\begin{proof}
This follows from the semiorthogonal decomposition~\eqref{equation-Ku} on fibers 
together with \cite[Lemma 3.25 and Theorem 3.17]{stability-families}. 
\end{proof} 

The results of \cite{DebKuz:moduli} give a precise description of families of 
GM varieties in terms of certain collections of vector bundles and morphisms, called GM data. 
In the case of a family of ordinary GM fourfolds $f \colon \cX \to S$, there is a 
canonically associated collection $(\cW, \cV_5, \cL, \mu, q)$ where: 
\begin{itemize}
\item $\cW$, $\cV_5$, and $\cL$ are vector bundles of ranks $9, 5$, and $1$ on $S$; 
\item $\mu \colon \cW \to ( \wedge^2\cV_5 ) \otimes \cL$ is an embedding of vector bundles; and 
\item $q \in \rH^0(M_{\cX}, \cO_{\bP_S(\cW)}(2) \otimes \det(\cV_5) \otimes \cL)$ where 
$M_{\cX} =  \Gr_S(2, \cV_5)  \times_{\bP_S(\wedge^2 \cV_5)} \bP_S(\cW)$; 
\end{itemize} 
such that $\cX$ is isomorphic to the zero locus of $q$ in $M_{\cX}$ and $\cO_{\cX}(1)$ is 
the restriction of $\cO_{\bP_S(\cW)}(1)$. 
This is a relative version of the description~\eqref{GM4fold} of a fixed GM fourfold. 
Using this, the results of Section~\ref{section-conic-fibration} can be upgraded to the family setting as follows. 

Let $\cQ$ be the line bundle defined by the exact sequence 
\begin{equation*} 
0 \to \cW \to ( \wedge^2\cV_5 ) \otimes \cL \to \cQ \to 0 . 
\end{equation*} 
The surjective arrow corresponds to a morphism $\cV_5 \to \cV_5^{\vee} \otimes \cL^{\vee} \otimes \cQ$. 
It follows from the smoothness of $f \colon \cX \to S$ that the kernel of this morphism is a 
line subbundle $\cV_1 \subset \cV_5$. 
Let $\bP_{\cW}^3 = \bP_S(\cV_5/\cV_1)$, and note that there is a natural embedding $\bP_{\cW}^3 \hookrightarrow M_{\cX}$. 
Let 
\begin{equation*}
\cT = \bP_{\cW}^3  \times_{M_{\cX}} \cX \to S 
\end{equation*} 
be the family of quadric surfaces over $S$ cut out by the restriction to $\bP_{\cW}^3$ of the 
section $q$ defining $\cX$ in $M_{\cX}$. 
We call $\cT \to S$ the \emph{family of canonical quadrics} of $\cX \to S$. 
Further, let $\cV_4 = \cV_5/\cV_1$, let $\cW' \subset ( \wedge^2\cV_4 ) \otimes \cL$ be the subbundle given by the image of $\cW$, and let 
\begin{equation*} 
g \colon \cY = \Gr_S(2,\cV_4) \times_{\bP(\wedge^2\cV_4 )} \bP_S(\cW') \to S, 
\end{equation*}
which is a family of quadric threefolds over $S$. 
Let $b \colon \wtilde{\cX} \to \cX$ be the blowup of $\cX$ in $\cT$. 
Then, as in Lemma~\ref{lemma-conic-bundle}, linear projection from $\bP_{\cW}^{3} \subset \bP_S(\cW)$ 
induces a conic fibration $\pi \colon \wtilde{\cX} \to \cY$, and if $\cT \to S$ is smooth then so are $\wtilde{\cX} \to S$ and $\cY \to S$. 

As in Section~\ref{subsection-Ku-in-Cl}, there are sheaves $\Cl_0$ and $\Cl_1$ on $\cY$ of even and odd parts of the Clifford algebra associated to the conic fibration $\pi \colon \wtilde{\cX} \to \cY$. We note that $\Cl_0$ and $\Cl_1$ are relative exceptional objects in the $S$-linear category $\Db(\cY, \Cl_0)$. 
Repeating the proof of Theorem~\ref{theorem-KX-Cl} in this setting shows the following. 

\begin{proposition}
\label{proposition-Ku-Cl-relative}
Let $f \colon \cX \to S$ be a family of ordinary GM fourfolds over a noetherian base $S$ with affine diagonal. 
Assume:
\begin{enumerate} 
\item \label{assumption-cT-smooth}
The family of canonical quadrics $\cT \to S$ is smooth. 
\item \label{assumption-Pic-lt}
The relative Picard group of $\cT \to S$ is a trivial rank $2$ local system.  
\end{enumerate}
Then there is a fully faithful $S$-linear functor $\Psi \colon \Ku(\cX) \to \Db(\cY, \Cl_0)$ and relative exceptional objects 
$\cR_{c}, \cR_{d} \in \Db(\cY, \Cl_0)$,  
such that there is a 
strong $S$-linear semiorthogonal decomposition of finite cohomological amplitude 
\begin{equation*}
\Db(\cY, \Cl_0) \hspace{-1mm}  = \hspace{-1mm} \llangle \Psi(\Ku(\cX)) , g^*(\Db(S)) \otimes \Cl_1, g^*(\Db(S)) \otimes \Cl_0, g^*(\Db(S)) \otimes \cR_{c}, g^*(\Db(S)) \otimes \cR_{d} \rrangle 
\end{equation*} 
whose base change along any point $s \in S$ recovers the decomposition of $\Db(\cY_s, \Cl_{0,s})$ given by Theorem~\ref{theorem-KX-Cl}. 
\end{proposition}

\begin{remark}
If $f \colon \cX \to S$ is a family of ordinary GM fourfolds for which assumption~\eqref{assumption-cT-smooth} in Proposition~\ref{proposition-Ku-Cl-relative} holds, then the relative Picard group of $\cT \to S$ is a rank $2$ local system on $S$. 
In general, this local system may not be trivial, but there always exists a degree $2$ \'{e}tale cover $S' \to S$ so that the base changed family $\cX' \to S'$ satisfies assumptions~\eqref{assumption-cT-smooth} and~\eqref{assumption-Pic-lt}. 
\end{remark}

\subsection{Stability conditions over a curve} 
\label{section-stab-over-curve}
This section contains the key technical ingredient for the proofs of the applications: 
a specialization over a curve from any GM fourfold to one whose Kuznetsov component is geometric, 
and the construction of a relative stability condition on the Kuznetsov component over the curve 
with well-behaved relative moduli spaces of stable objects. 

In this section, for an ordinary GM fourfold $X$ with smooth canonical quadric, we consider stability conditions on $\Ku(X)$ contained in an open subset $\Stab^{\dagger}(\Ku(X)) \subset \Stab(\Ku(X))$ of the space of full numerical stability conditions on $\Ku(X)$, defined as follows. 
In the notation of Section~\ref{intro-stab-manifold}, the map $\eta \colon \eta^{-1}(\cP_0(\Ku(X))) \subset \Stab(\Ku(X)) \to \cP_0(\Ku(X))$ is a covering map, 
by the argument of \cite[Proposition 8.3]{bridgeland-K3}. 
By Lemma~\ref{lemma-eta-sigma}, the stability conditions on $\Ku(X)$ constructed in the proof of Theorem~\ref{theorem-stab-Ku-good-4fold} lie in $\eta^{-1}(\cP_0(\Ku(X)))$; we let $\Stab^{\dagger}(\Ku(X))$ be the connected component of $\eta^{-1}(\cP_0(\Ku(X)))$ containing them. 
In the proof of Theorem~\ref{theorem-Stab-dagger} in the next section, we will see that $\Stab^{\dagger}(\Ku(X))$ actually forms a connected component of $\Stab(\Ku(X))$. 

Given a family of categories $\cD$ over a base $S$ 
(e.g. $\cD = \Ku(\cX)$ for $\cX \to S$ a family of GM fourfolds), the notion of a 
\emph{stability condition on $\cD$ over $S$} was introduced in \cite{stability-families}. 
This consists of a collection $\usigma = (\sigma_s)_{s \in S}$ of stability conditions $\sigma_s$ 
on the fibers $\cD_s$, $s \in S$, satisfying certain axioms, and comes with a notion of 
relative moduli spaces of stable objects. 
We emphasize that in the case where the base is a point $S = \Spec(\bC)$, a stability condition on $\cD$ over $S$ is specified by a usual stability condition on $\cD$ satisfying certain properties --- roughly, the existence of bounded moduli spaces --- which a priori may not be satisfied by an arbitrary stability condition on $\cD$. 
On the other hand, the results of \cite{stability-families} show that the known constructions of stability conditions via tilting actually give stability conditions over $\Spec(\bC)$. In particular, we have the following result (which is actually a special case of Proposition~\ref{proposition-stability-over-curve} proved below, but which we state first for psychological reasons). 

\begin{lemma}
\label{sigma-family-stability}
Let $X$ be an ordinary GM fourfold $X$ with smooth canonical quadric. 
Then any $\sigma \in \Stab^{\dagger}(\Ku(X))$ is a stability condition on $\Ku(X)$ over $\Spec(\bC)$. 
In particular, if $v \in \tH^{1,1}(\Ku(X), \bZ)$, then the moduli stack of $\sigma$-semistable objects in $\Ku(X)$ of class $v$ (see \cite[Definition 21.11]{stability-families}) admits a good moduli space $M_{\sigma}(\Ku(X), v)$, which is a proper algebraic space over $\bC$. 
In case $\sigma$ is $v$-generic, then $M_{\sigma}(\Ku(X), v)$ is a coarse moduli space which is smooth and proper over $\bC$. 
\end{lemma}

\begin{proof}
The claim that any $\sigma \in \Stab^{\dagger}(\Ku(X))$ is a stability condition on $\Ku(X)$ over $\Spec(\bC)$ follows from our construction of stability conditions in $\Stab^{\dagger}(\Ku(X))$ from Section~\ref{section-stability} and the results of \cite[Parts IV and V]{stability-families}, cf. the proof of \cite[Proposition 30.4]{stability-families} and \cite[Remark 30.5]{stability-families}. 
Then \cite[Theorem 21.24]{stability-families} gives the rest of the lemma, except the smoothness of $M_{\sigma}(\Ku(X), v)$ in the case $\sigma$ is $v$-generic. 
For smoothness, we use the algebraic space $sM_{\pug}(\Ku(X))$ parameterizing simple universally gluable objects in $\Ku(X)$ (see \cite[Section 9.3]{stability-families} for the precise definition). 
By Mukai's theorem in the form proved in \cite{IHC-CY2}, the space $sM_{\pug}(\Ku(X))$ is smooth.  
Since $\sigma$ is a stability condition on $\Ku(X)$ over $\bC$, we have that $M_{\sigma}(\Ku(X), v)$ is an open subspace of $sM_{\pug}(\Ku(X))$, and so smooth as well. 
\end{proof} 

Next we prove a $1$-parameter version of Lemma~\ref{sigma-family-stability}. 
In Theorem~\ref{theorem-relative-moduli} of the next section, we will discuss the case of higher-dimensional bases. 

\begin{proposition}
\label{proposition-stability-over-curve} 
Let $X$ be an ordinary GM fourfold with smooth canonical quadric, 
let $v$ be a primitive vector in $\tH^{1,1}(\Ku(X), \bZ)$, and 
let $\sigma \in \Stab^{\dagger}(\Ku(X))$ be a $v$-generic stability condition. 
Let $X'$ be another ordinary GM fourfold with smooth canonical quadric which is deformation equivalent to $X$ within the Hodge locus for $v$, 
i.e. there is smooth family of GM fourfolds over a connected quasi-projective base with fibers $X$ and $X'$ along which $v$ remains a Hodge class.  
Then there exists a family $\cX \to C$ of ordinary GM fourfolds over a  
smooth connected quasi-projective curve $C$ along which $v$ remains a Hodge class, 
complex points $0, 1 \in C(\bC)$, 
and a stability condition $\underline{\sigma}$ on $\Ku(\cX)$ over $C$, 
such that: 
\begin{enumerate}
\item \label{defo-fibers} 
$\cX_{0} = X$ and $\cX_{1} = X'$. 
\item $v$ is a primitive vector in $\tH^{1,1}(\Ku(\cX_c), \bZ)$ for all $c \in C$. 
\item $\sigma_c \in \Stab^{\dagger}(\Ku(\cX_c))$ is $v$-generic for all $c \in C$, 
and $\sigma_0$ is a small deformation of $\sigma$ so that $M_{\sigma_0}(\Ku(X), v) = M_{\sigma}(\Ku(X), v)$. 
\item \label{defo-moduli} 
The relative moduli space $M_{\underline{\sigma}}(\Ku(\cX), v)$ (given as the good moduli space for the moduli stack of $\usigma$-semistable objects in $\Ku(\cX)$ of class $v$, see \cite[Theorem 21.24]{stability-families}) is a smooth and proper algebraic space over $C$. 
\end{enumerate}
\end{proposition}

\begin{proof}
By assumption, there exists a family $\cX \to C$ of smooth GM fourfolds over a smooth connected 
quasi-projective curve $C$, such that $\cX_0 = X$, $\cX_1 = X'$ for some points $0, 1 \in C(\bC)$, and $v$ remains a Hodge class along $C$. 
We may assume all fibers of $\cX \to C$ are ordinary with smooth canonical quadric, since these properties 
are open in families of GM fourfolds. 
Thus our construction from Section~\ref{section-stability} gives stability conditions on any fiber of $\cX \to C$. 
Using this and Proposition~\ref{proposition-Ku-Cl-relative}, 
up to possibly replacing $C$ by a finite cover we obtain a 
family of stability conditions $\usigma$ on $\Ku(\cX)$ over $C$ satisfying properties~\eqref{defo-fibers}-\eqref{defo-moduli}, 
by the same argument as in the proof of \cite[Corollary 32.1]{stability-families} with the following modification: 
instead of \cite[Theorem 31.1]{stability-families} we appeal to the general Mukai theorem on smoothness 
of moduli spaces of simple universally gluable objects for families of CY2 categories proved in \cite{IHC-CY2}. 
\end{proof}

Proposition~\ref{proposition-stability-over-curve} shows that to prove deformation 
invariant properties about the moduli space $M_{\sigma}(\Ku(X), v)$, we may specialize 
the GM fourfold $X$ within the Hodge locus for $v$. 
We will use this observation by specializing to the case where $\Ku(X)$ is equivalent to the derived category of a twisted K3 surface $(S, \alpha)$, where many results about moduli spaces of stable objects are already known \cite{bayer-macri-projectivity}. 
There is a subtlety that moduli spaces of $\sigma$-stable objects in $(S, \alpha)$ are only well-understood when $\sigma$ lies in the connected component $\Stab^{\dagger}(S, \alpha)$ containing geometric stability conditions, i.e. those for which skyscraper sheaves of points are stable of the same phase. 

\begin{definition}
Let $X$ be an ordinary GM fourfold with smooth canonical quadric and let $(S, \alpha)$ be a 
twisted K3 surface. 
Then a \emph{$\dagger$-equivalence} $\Ku(X) \simeq \Db(S, \alpha)$ is an equivalence 
under which $\Stab^{\dagger}(\Ku(X))$ maps to $\Stab^{\dagger}(S, \alpha)$. 
\end{definition}

The following result thus gives a specialization of GM fourfolds of the type we require. 

\begin{proposition}
\label{proposition-specialize} 
Let $X$ be an ordinary GM fourfold and let $v \in \tH^{1,1}(\Ku(X), \bZ)$ be a Hodge class. 
Then there exists an ordinary GM fourfold $X'$ with smooth canonical quadric such that: 
\begin{enumerate}
\item $X$ is deformation equivalent to $X'$ within the Hodge locus for $v$. 
\item \label{Ku-twisted-K3-stability}
There exists a twisted K3 surface $(S', \alpha')$ and $\dagger$-equivalence $\Ku(X') \simeq \Db(S', \alpha')$. 
\end{enumerate}
\end{proposition} 

We will prove Proposition~\ref{proposition-specialize} at the end of this section, after some preliminary results. 

\begin{lemma}
\label{lemma-Ku-twisted-K3}
Let $X$ be an ordinary GM fourfold with smooth canonical quadric.  
Assume there exists a primitive $v \in \tH^{1,1}(\Ku(X), \bZ)$ with $(v,v) = 0$. 
Let $\sigma \in \Stab^{\dagger}(\Ku(X))$ be a $v$-generic stability condition, and 
assume there exists a $\sigma$-stable object in $\Ku(X)$ of class $v$. 
Then $S = M_{\sigma}(\Ku(X),v)$ is a smooth K3 surface and there is a Brauer class $\alpha \in \Br(S)$ such that there is 
a $\dagger$-equivalence $\Ku(X) \simeq \Db(S, \alpha)$. 
\end{lemma}

\begin{proof}
By Lemma~\ref{sigma-family-stability}, $S$ is a smooth proper algebraic space. 
Because $\Ku(X)$ is a CY2 category, standard arguments (see \cite[Section 2]{KM:symplectic}) 
show that $S$ has dimension $(v,v) + 2 = 2$ and is equipped with a symplectic form. 
Being a smooth proper $2$-dimensional algebraic space, we conclude that $S$ is in fact a smooth projective surface. By \cite[Proposition A.7]{BLMS}, it follows that $S$ is also connected. 
The existence of a symplectic form on $S$ then implies it is either a K3 or abelian surface. 
Let $\cE$ be a quasi-universal family over $S \times X$ and $\alpha \in \Br(S)$ the associated Brauer class. 
Standard arguments (cf. \cite[Lemma 32.3]{stability-families}) then show that the corresponding Fourier--Mukai functor $\Phi_{\cE} \colon \Db(S, \alpha) \to \Db(X)$ factors through an equivalence $\Db(S, \alpha) \simeq \Ku(X)$. 
Then since $\Ku(X)$ has the same Hochschild homology as a K3 surface, so does $\Db(S, \alpha)$. It follows that $S$ is not an abelian surface, and hence is a K3 surface. 
\end{proof}

Next we show that the existence of an $\dagger$-equivalence as in Proposition~\ref{proposition-specialize}\eqref{Ku-twisted-K3-stability} deforms along Hodge loci for square zero classes. 

\begin{lemma}
\label{lemma-deform-Ku-twisted-K3} 
Let $X$ and $X'$ be ordinary GM fourfolds with smooth canonical quadrics, 
which are deformation equivalent within the Hodge locus for a $v \in \tH^{1,1}(\Ku(X), \bZ)$ with $(v,v) = 0$. 
Then $\Ku(X)$ is $\dagger$-equivalent to the derived category of a twisted K3 surface if and only if $\Ku(X')$ is. 
\end{lemma} 

\begin{proof}
If $\sigma \in \Stab^{\dagger}(\Ku(X))$ is $v$-generic, then $M_{\sigma}(\Ku(X), v)$ is a smooth K3 surface. 
Indeed, this holds by our assumption and the analogous statement for twisted K3 surfaces \cite{bayer-macri-projectivity}. 
By Proposition~\ref{proposition-stability-over-curve}, it follows that for any $v$-generic $\sigma' \in \Stab^{\dagger}(\Ku(X'))$ the moduli space $M_{\sigma'}(\Ku(X'), v)$ is also a smooth K3 surface. 
Then we conclude by Lemma~\ref{lemma-Ku-twisted-K3}. 
\end{proof} 

The following result roughly says that given an equivalence $\Ku(X) \simeq \Db(S, \alpha)$, we can modify it (possibly by replacing $(S, \alpha)$ with a different twisted K3) to a $\dagger$-equivalence 
provided that $X$ admits deformations with suitable Hodge-theoretic properties. 

\begin{lemma}
\label{lemma-modify-Ku-twisted}
Let $X$ be an ordinary GM fourfold with smooth canonical quadric such that: 
\begin{enumerate}
\item There is an equivalence $\Ku(X) \simeq \Db(S, \alpha)$ for some twisted K3 surface $(S, \alpha)$. 
\item \label{lemma-modify-Ku-twisted-spherical}
There exists a class $v \in \tH^{1,1}(\Ku(X), \bZ)$ with $(v,v) = 0$ such that $X$ is deformation equivalent within the Hodge locus for $v$ to another ordinary GM fourfold $X'$ with smooth canonical quadric, with the property that $\tH^{1,1}(\Ku(X'), \bZ)$ contains no elements $\delta$ with $(\delta, \delta) = -2$. 
\end{enumerate}
Then there exists a twisted K3 surface $(T, \beta)$, possibly different from $(S, \alpha)$, and a $\dagger$-equivalence $\Ku(X) \simeq \Db(T, \beta)$. 
\end{lemma} 

\begin{proof}
Let $\cX \to C$ be a smooth family of ordinary GM fourfolds with smooth canonical quadrics over a smooth connected quasi-projective curve $C$, such that $\cX_0 = X$ and $\cX_1 = X'$ for some points $0,1 \in C(\bC)$, and $v$ remains a Hodge class along $C$. 
For a $v$-generic $\sigma \in \Stab^{\dagger}(S, \alpha)$, the space $M_\sigma(\Db(S, \alpha), v)$ is a K3 surface by \cite{bayer-macri-projectivity}. 
In particular, it follows that there exists a simple object $E \in \Ku(X)$ of class $v(E) = v \in \tH^{1,1}(\Ku(X), \bZ)$ with $\Ext^{< 0}(E,E) = 0$. 
By Mukai's theorem for the family of CY2 categories $\Ku(\cX)$ over $C$ \cite{IHC-CY2}, 
it follows that there exists a Zariski open subset $U \subset C$ such that for any $c \in U$ there is a simple object $E_c \in \Ku(\cX_c)$ of class $v(E_c) = v \in \tH^{1,1}(\Ku(\cX_c), \bZ)$ with $\Ext^{< 0}(E_c,E_c) = 0$, 
and in particular $\Ext^1(E_c, E_c) \cong \bC^2$. 

The condition that $\tH^{1,1}(\Ku(\cX_c), \bZ)$ contains no elements $\delta$ with $(\delta, \delta) = -2$ holds for a very general point $c \in C$, since it holds for $c = 1$. 
Therefore, up to replacing $X'$ with a different fiber of $\cX \to C$, 
we may assume there exists an object $E' \in \Ku(X')$ such that $\Ext^1(E',E') \cong \bC^2$. 
By \cite[Lemma A.4]{BLMS}, it follows that $E'$ is $\sigma$-stable for any $\sigma \in \Stab(\Ku(X'))$. 
Thus by Lemma~\ref{lemma-Ku-twisted-K3}, there exists a twisted K3 surface $(S', \alpha')$ and 
a $\dagger$-equivalence $\Ku(X') \simeq \Db(S', \alpha')$. 
From this, the result follows by applying Lemma~\ref{lemma-deform-Ku-twisted-K3}. 
\end{proof} 

Now we can explain the idea of the proof of Proposition~\ref{proposition-specialize}. 
The closure of the locus of GM fourfolds containing a quintic del Pezzo surface forms a divisor in the moduli space of GM fourfolds, such that any GM fourfold can be specialized along any Hodge locus into this divisor \cite{DIM4fold}, and $\Ku(X)$ is equivalent to the derived category of a K3 surface on (a Zariski open subset of) this divisor \cite{KuzPerry:dercatGM}. 
We will use Lemma~\ref{lemma-modify-Ku-twisted} to modify the equivalence to a $\dagger$-equivalence with 
a twisted K3 surface. 
To verify condition~\eqref{lemma-modify-Ku-twisted-spherical} of Lemma~\ref{lemma-modify-Ku-twisted} in this situation, we will need some input from the period morphism for GM fourfolds, which we review now. 

Let $\cM$ denote the moduli stack of GM fourfolds. 
This is a smooth, irreducible Deligne--Mumford stack of dimension $24$ \cite[Proposition A.2]{KuzPerry:dercatGM} (see also~\cite{DebKuz:birGM}).
We denote by 
\begin{equation*}
p \colon \cM \to \cD 
\end{equation*}
the period morphism, where the period domain $\cD$ is the $20$-dimensional quasi-projective variety 
classifying Hodge structures on the middle cohomology $\rH^4(X_0, \bZ)$ of a fixed GM fourfold $X_0$ for which the canonical sublattice $\rH^4(\Gr(2,5),\bZ) \subset \rH^4(X_0, \bZ)$ consists of Hodge classes  
(see \cite{DIM4fold} for details). 
Note that by Section~\ref{section-Mukai-HS}, the period domain $\cD$ can also be thought of as classifying Hodge structures on $\tH(\Ku(X_0),  \bZ)$ for which the canonical sublattice 
$A_1^{\oplus 2} \subset \tH(\Ku(X_0),  \bZ)$ consists of Hodge classes. 
We consider inside $\cD$ the locus parameterizing Hodge structures on $\tH(\Ku(X_0),  \bZ)$ 
for which there are ``extra'' Hodge classes, i.e. nonzero ones orthogonal to $A_1^{\oplus 2}$. 
By \cite{DIM4fold}, this locus is the union of divisors $\cD_d \subset \cD$ over positive integers $d$ satisfying $d \equiv 0, 4, \text{ or } 2 \pmod 8$, where $\cD_d$ is irreducible for $d \equiv 0 \text{ or } 4 \pmod 8$ and $\cD_d = \cD'_d \cup \cD''_d$ is the union of two irreducible divisors $\cD'_d$ and $\cD''_d$ for $d \equiv 2 \pmod 8$. 

Let $\cM^{\ord} \subset \cM$ denote the open subspace parameterizing ordinary GM fourfolds, and let $p^{\ord} \colon \cM^\ord \to \cD$ denote the restriction of the period morphism to this subspace. 

\begin{lemma}
\label{lemma-periods}
\begin{enumerate}
\item \label{p-dominant} 
The morphism $p^\ord \colon \cM^\ord \to \cD$ is smooth and dominant. 

\item \label{preimage-Z-irreducible}
For any irreducible closed subscheme $Z \subset \cD$, the preimage $(p^{\ord})^{-1}(Z)$ is irreducible. 

\item \label{dP5-dominant}
Let $\cM_{\dP5} \subset \cM$ be the locus parameterizing GM fourfolds containing a quintic del Pezzo surface. Then $\cM_{\dP5}$ is irreducible, and the restriction of the period morphism to $\cM_{\dP5}$ factors through a dominant morphism $\cM_{\dP5} \to \cD''_{10}$. 
\item \label{D10} 
There exists a vector $v \in \tH(\Ku(X_0),  \bZ)$ with $(v,v) = 0$ which is a Hodge class for all of the Hodge structures parameterized by $\cD''_{10}$. 
\item \label{Z-divisors} There are infinitely many divisors $Z$ among $\cD_d, \cD'_d$, and $\cD''_d$ for which the following conditions are satisfied: 
\begin{itemize}
\item There exists a $v \in \tH(\Ku(X_0),  \bZ)$ with $(v,v) = 0$ which is a Hodge class for all of the Hodge structures parameterized by $Z$. 
\item For a very general point of $Z$, the corresponding Hodge structure on $\tH(\Ku(X_0), \bZ)$ contains no Hodge classes $\delta$ with $(\delta, \delta) = -2$. 
\end{itemize}
\item \label{Z-D10} 
The intersections $Z \cap \cD''_{10}$ for $Z$ as in~\eqref{Z-divisors} give infinitely many distinct divisors in $\cD''_{10}$. 
\end{enumerate} 
\end{lemma}

\begin{proof}
\eqref{p-dominant} follows from \cite[Theorem 4.4]{DIM4fold}. 
For~\eqref{preimage-Z-irreducible}, first we observe that the fibers of the morphism $p^{\ord} \colon \cM^{\ord} \to \cD$ over closed points of $\cD$ are irreducible; 
indeed, the combined results of \cite{DebKuz:birGM, DebKuz:periodGM} show the any fiber is the smooth locus of an EPW sextic. 
Thus if $Z \subset \cD$ is an irreducible closed subscheme, then the morphism $(p^{\ord})^{-1}(Z) \to Z$ is smooth and in particular open by~\eqref{p-dominant}, and has irreducible fibers over closed points. 
It follows that $(p^{\ord})^{-1}(Z)$ is irreducible, see \cite[\href{https://stacks.math.columbia.edu/tag/004Z}{Tag 004Z}]{stacks-project}. 

\eqref{dP5-dominant} is \cite[Proposition 7.7]{DIM4fold}. 
\eqref{D10} follows from \eqref{dP5-dominant} combined with either \cite[Theorem 1.2]{Pert} or \cite[Theorem 4.1 and Lemma 4.4]{KuzPerry:dercatGM}. 

In order to prove \eqref{Z-divisors}, recall that by \cite[Theorem 1.1]{Pert}, the existence of an isotropic class in the Mukai lattice of $\Ku(X_0)$ is equivalent to the fact that $X_0$ has period point in a divisor $\cD_d$ with $d$ having prime factorization of the form 
\begin{equation*}\tag{$\ast\ast'$}
d=\prod_{i}p_i^{n_i} \text{ with } n_i \equiv 0 \pmod 2 \text{ for } p_i \equiv 3 \pmod 4.
\end{equation*}
Now assume that $d$ satisfies $(\ast\ast')$ and $32 \mid d$. We claim that very general points of $Z:=\cD_d$ parametrize Hodge structures on the Mukai lattice not containing square $-2$ Hodge classes. The proof of this claim is analogous to that of \cite[Proposition 2.15]{Huy}. Indeed, assume there is a class $\delta \in \wtilde{\rH}^{1,1}(\Ku(X_0),\bZ)$ with $\delta^2=-2$. As $32 \mid d$, we have $d \equiv 0 \pmod 8$. Thus we can write $\delta=w + kv$, with $w \in A_1^{\oplus 2}$ and $v \in (A_1^{\oplus 2})^\perp$ such that $v^2=-8s$. Then $-2=w^2 -8k^2s$, which implies $w^2 \equiv 6 \pmod 8$. But we would have $w^2/2 \equiv 3 \pmod 4$, in contradiction with $(\ast\ast')$. This implies \eqref{Z-divisors}.

It remains to prove \eqref{Z-D10}. Assume $Z:=\cD_d$ with $d=-32s$ for an integer $s>0$. On the other hand, consider the rank $4$ lattice with intersection form given by
\begin{equation*} 
L_s:=\begin{pmatrix}
2 & 0 & 0 & 0\\
0 & 2 & 0 & 1\\
0 & 0 & -8s & 0\\
0 & 1 & 0 & -2
\end{pmatrix}.
\end{equation*} 
By \cite[Corollary 1.12.3]{Nik} the lattice $L_s$ has a primitive embedding in the Mukai lattice of $\Ku(X_0)$. Denote by $\cL_s \subset \cD_d \cap \cD_{10}''$ the locus in $\cD$  
parameterizing Hodge structures on the Mukai lattice for which $L_s$ consists of Hodge classes. 

We claim that a very general element of $\cL_s$ does not belong to $\cL_{s'}$ for every $s' \neq s$. Indeed, denote by $\lambda_1, \lambda_2, \tau_1, \tau_2$ the basis for $L_s$ representing the intersection form as above. Then a simple computation shows that there is not a class $\tau \in L_s$ such that
$$\tau \cdot \lambda_1=\tau \cdot \lambda_2=\tau \cdot \tau_2=0, \quad 
\tau^2=-8s'$$
for every $s' \neq s$. As a consequence, we have $\cL_s \nsubseteq \cL_{s'}$, which implies $\cD_d \cap \cD_{10}'' $ is not contained in $\cD_{d'} \cap \cD_{10}'' $ for every $d'=-32s'\neq d$. In particular, we can consider the union $\cup_{d}(\cD_d \cap \cD_{10}'')$, where $d$ satisfies $(\ast\ast'')$ and $32 \mid d$, of countably many infinite divisors in $\cD_{10}''$ as required. 
\end{proof} 

Finally, we can prove Proposition~\ref{proposition-specialize}. 
\begin{proof}[Proof of Proposition~\ref{proposition-specialize}]
The Hodge class $v \in \tH^{1,1}(\Ku(X), \bZ)$ determines an irreducible closed subscheme 
$Z \subset \cD$ (equal to one of the irreducible divisors discussed above, or the whole period domain $\cD$ in the case $v \in A_{1}^{\oplus 2}$) such that $p(X) \in Z$ and $v$ is a Hodge class for the Hodge structures parameterized by $Z$. 
By Lemma~\ref{lemma-periods}\eqref{preimage-Z-irreducible} the locus $(p^{\ord})^{-1}(Z)$ is irreducible, and thus consists of those ordinary GM fourfolds which are deformation equivalent to $X$ within the Hodge locus for $v$. 
It follows from the construction of GM fourfolds in the proof of \cite[Theorem 8.1]{DIM4fold} 
that $(p^{\ord})^{-1}(Z)$ contains an ordinary GM fourfold $X'$ containing a so-called $\sigma$-plane. 
Further, up to changing $X'$ in the fiber of the period map $p^{\ord} \colon \cM^{\ord} \to \cD$, we may assume that $X'$ has smooth canonical quadric. 
Indeed, the proof of Theorem~\ref{theorem-duality} shows that $X'$ admits a period partner $X''$ 
which is an ordinary GM fourfold with smooth canonical quadric, 
and then the results of \cite{DebKuz:periodGM} show that $X'$ and $X''$ are contained in the same fiber of the 
period map $p^{\ord} \colon \cM^{\ord} \to \cD$. 

By \cite[Proposition 7.7]{DIM4fold}, $X'$ is contained in the closure of $\cM_{\dP5} \subset \cM$. 
Hence by Lemma~\ref{lemma-deform-Ku-twisted-K3} and Lemma~\ref{lemma-periods}\eqref{D10}, to finish the proof it suffices to show there exists an ordinary GM fourfold $Y$ in $\cM_{\dP5}$ with smooth canonical quadric, such that $\Ku(Y)$ is 
$\dagger$-equivalent to the derived category of a twisted K3 surface. 

Let $\cM^{\circ} \subset \cM$ be the open subspace parameterizing ordinary GM fourfolds with smooth canonical quadric. 
By \cite[Theorem 4.1 and Lemma 4.4]{KuzPerry:dercatGM}, the subspace $\cM_{\dP5}^{\circ} = \cM_{\dP5} \cap \cM^{\circ}$ 
parameterizes $Y$ such that $\Ku(Y)$ is equivalent to the derived category of a K3 surface. 
Consider the divisors $Z \subset \cD$ from Lemma~\ref{lemma-periods}\eqref{Z-divisors}. 
By parts \eqref{Z-D10} and \eqref{dP5-dominant} of Lemma~\ref{lemma-periods}, there 
are infinitely many such $Z$ for which $\cM_{\dP5}^{\circ}$ meets $p^{-1}(Z)$. 
Moreover, among these $Z$, by Lemma~\ref{lemma-periods}\eqref{p-dominant}  
there are infinitely many for which $\cM^{\circ} \cap p^{-1}(Z) \to Z$ is dominant. 
Fix such a $Z$, let $Y$ be a GM fourfold in $\cM_{\dP5}^{\circ} \cap p^{-1}(Z)$, and 
choose $Y'$ in $\cM^{\circ} \cap p^{-1}(Z)$ very general so that $\tH^{1,1}(\Ku(Y'), \bZ)$ contains 
no elements $\delta$ with $(\delta, \delta) = -2$. 
By construction, $Y$ and $Y'$ are deformation equivalent within the Hodge locus for a 
$v \in \tH^{1,1}(\Ku(Y), \bZ)$ with $(v,v) = 0$. 
Thus by Lemma~\ref{lemma-modify-Ku-twisted} we conclude that $Y$ is $\dagger$-equivalent to the 
derived category of a twisted K3 surface, finishing the proof.  
\end{proof}

\subsection{Proofs of the applications} 
\label{section-proofs-applications}
Using Propositions~\ref{proposition-stability-over-curve} and~\ref{proposition-specialize}, 
we can now prove the promised applications. 

\begin{proof}[Proof of Theorem~\ref{theorem-Ku-geometric}]
By Theorem~\ref{theorem-duality}, it suffices to consider the case of an ordinary GM fourfold $X$ with smooth canonical quadric. 
In this case, by \cite[Lemma 32.3]{stability-families} and Lemma~\ref{lemma-Ku-twisted-K3}, 
it suffices to show that if there exists a nonzero primitive $v \in \tH^{1,1}(\Ku(X), \bZ)$ with $(v,v) = 0$ 
and $\sigma \in \Stab^{\dagger}(\Ku(X))$ is $v$-generic, then $M_{\sigma}(\Ku(X), v)$ is a K3 surface 
(or even just nonempty).  
By combining Propositions~\ref{proposition-specialize} and~\ref{proposition-stability-over-curve}, 
this reduces to showing the analogous statement for a twisted K3 surface, which holds by \cite{bayer-macri-projectivity}. 
\end{proof}

\begin{remark}
\label{remark-dagger-oops}
The above argument is similar to the proof of the analogous result \cite[Proposition 32.2]{stability-families} for cubic fourfolds. 
In \cite{stability-families}, however, there is a gap in the argument: the equivalences $\Ku(Y) \simeq \Db(S, \alpha)$ for the special cubic fourfolds $Y$ used in the proof are not checked to be $\dagger$-equivalences, which is necessary to invoke \cite{bayer-macri-projectivity}. 
It is easy to see that our arguments in the proof of Proposition~\ref{proposition-stability-over-curve} also apply in the case of cubic fourfolds to fill this gap. 
\end{remark}

\begin{remark}
\label{rmk-HTK3versusHomologicalK3}
In \cite[\S3.3]{Pert} it is proved that there are examples of GM fourfolds $X$ with a hyperbolic plane primitively embedded in $\tH^{1,1}(\Ku(X), \bZ)$, but which cannot have a Hodge-theoretically associated K3 surface in the sense of \eqref{HT-associated-K3}. On the other hand, by Theorem \ref{theorem-Ku-geometric} we have that $X$ has a homological associated K3 surface $S$, i.e.\ $\Ku(X) \simeq \Db(S)$. These examples show a difference of GM fourfolds from cubic fourfolds, where having a Hodge-theoretic associated K3 surface and a homological associated K3 surface are equivalent conditions by \cite{addington-thomas, stability-families}. This makes particularly interesting the question of the relation between the two notions of an associated K3 surface and the rationality of a GM fourfold.
\end{remark}

Theorem~\ref{theorem-moduli-space} also follows by a deformation argument. 
The proof is the same as the analogous result \cite[Theorem 29.2]{stability-families} for cubic fourfolds (with the same caveat as in Remark~\ref{remark-dagger-oops}), so we omit it. 
Using this, Theorem~\ref{theorem-Stab-dagger} then follows as in the proof of the analogous result \cite[Theorem 29.1]{stability-families} for cubic fourfolds. \qed 

\medskip 
Corollary~\ref{corollary-hodge-isom} can be proved similarly to \cite[Theorem 1.2]{addington-thomas}, but using Theorem~\ref{theorem-Ku-geometric} we can give a slightly more direct argument: 

\begin{proof}[Proof of Corollary~\ref{corollary-hodge-isom}]
As in the proof of \cite[Proposition 5.1]{addington-thomas}, we can extend the given Hodge isometry $\varphi \colon K^{\perp} \xrightarrow{\sim} L^{\perp}(1)$ to a Hodge isometry 
$\widetilde{\varphi} \colon \tH(S, \bZ) \xrightarrow{\sim} \tH(\Ku(X), \bZ)$. 
Since $\tH^{1,1}(S, \bZ)$ contains a hyperbolic plane, 
Theorem~\ref{theorem-Ku-geometric} shows $\Ku(X) \simeq \Db(S')$ for a K3 surface $S'$. 
Using the derived Torelli theorem for K3 surfaces as in \cite[Proposition 5.1]{addington-thomas}, 
it follows that there is an equivalence $\Db(S) \simeq \Ku(X)$ which induces the Hodge isometry $\widetilde{\varphi}$. 
This implies that $\widetilde{\varphi}$, and hence also $\varphi$, is algebraic. 
\end{proof}

The proof of Theorem~\ref{theorem-unirational-families} relies on the following general existence result for relative moduli spaces.
To formulate this precisely, note that if $\cX \to S$ is a family of GM fourfolds over a complex variety, then the Mukai lattices of the fibers $\tH(\Ku(\cX_s), \bZ)$, $s \in S(\bC)$, form the fibers of a local system $\tH(\Ku(\cX)/S, \bZ)$ on $S$ (see \cite{IHC-CY2} for the construction of this local system for general families of CY2 categories). 

\begin{theorem}
\label{theorem-relative-moduli}
Let $\cX \to S$ be a family of ordinary GM fourfolds over a connected complex quasi-projective variety $S$, 
whose associated family of canonical quadrics is smooth. 
Let $v$ be a primitive section of the local system $\tH(\Ku(\cX)/S, \bZ)$ whose fibers are Hodge classes. 
Assume that for a very general point $s_0 \in S$, there exists a stability condition 
$\tau_{s_0} \in \Stab^{\dagger}(\Ku(\cX_{s_0}))$ that is generic with respect to $v$, and whose 
central charge $Z_{s_0} \colon \tH^{1,1}(\Ku(\cX_{s_0}), \bZ) \to \bC$ is invariant under the monodromy action. 
\begin{enumerate}
\item 
If $S = C$ is a curve, then there exists an algebraic space $\tM(v)$ and a smooth proper morphism $\tM(v) \to C$ that makes $\tM(v)$ a relative moduli space over $C$, i.e. 
the fiber over any point $c \in C$ is a coarse moduli space $M_{\sigma_c}(\Ku(\cX_c), v_c)$ of stable objects in the Kuznetsov component of the corresponding GM fourfold for some stability condition $\sigma_c$.
\item \label{M0v}
There exist a nonempty open subset $S^\circ \subset S$, a quasi-projective variety $M^\circ(v)$, and a smooth projective morphism $M^\circ(v) \to S^\circ$ making $M^\circ(v)$ a relative moduli space over $S^\circ$.
\item
There exist an algebraic space $M(v)$ and a proper morphism $M(v) \to S$ such that every fiber is a
good moduli space $M_{\sigma_s}(\Ku(\cX_s), v_s)$ of semistable objects.
\end{enumerate}
In all cases, we can choose the stability conditions $(\sigma_s)_{s \in S}$ on the fiber categories $\Ku(\cX_s)$ 
so that $M_{\sigma_{s_0}}(\Ku(\cX_{s_0}), v_{s_0}) = M_{\tau_{s_0}}(\Ku(\cX_{s_0}), v_{s_0})$.
\end{theorem} 

\begin{proof}
Using the results we have already proven, the same argument as for \cite[Theorem 29.4]{stability-families} works, 
except again instead of \cite[Theorem 31.1]{stability-families} we appeal to the general form of Mukai's theorem 
from \cite{IHC-CY2}. 
\end{proof} 

\begin{proof}[Proof of Theorem~\ref{theorem-unirational-families}]
The proof is analogous to that of \cite[Corollary 29.5]{stability-families}. 

Let $\cX \to S$ be a family of ordinary GM fourfolds with smooth canonical quadric 
over an open subset $S$ of the $24$-dimensional coarse moduli space of GM fourfolds constructed in~\cite{DebKuz:moduli}; 
such a family exists because generically GM fourfolds have trivial automorphism group. 
By \cite[Proposition 2.25]{KuzPerry:dercatGM}, for a very general point $s_0 \in S$ 
we have $\cN(\Ku(\cX_{s_0})) = A_1^{\oplus 2}$. 
We observe that this lattice is monodromy invariant. 
Indeed, this follows from the fact that for any GM fourfold $X$, the canonical sublattice $A_1^{\oplus 2} \subset \cN(\Ku(X))$ is identified with the projection into $\cN(\Ku(X))$ of the image of the pullback map $\rK(\Gr(2,5)) \to \rK(X)$ on Grothendieck groups. 

For any pair $(a,b)$ of coprime pair of integers, let $v = a \lambda_1 + b\lambda_2$ where 
$\lambda_1$ and $\lambda_2$ are the generators for $A_1^{\oplus 2}$. 
Let $\tau_{s_0} \in \Stab^{\dagger}(\Ku(\cX_{s_0}))$ be a $v$-generic stability condition. 
By Lemma~\ref{lemma-eta-sigma} we have $\eta(\tau_{s_0}) \in A_1^{\oplus 2} \otimes \bC$, 
which implies the central charge of $\tau_{s_0}$ is monodromy invariant. 

Thus Theorem~\ref{theorem-relative-moduli}\eqref{M0v} gives a smooth projective relative moduli space $g \colon M^\circ(v) \to S^\circ$ over an open subset $S^\circ \subset S$. 
The base $S^\circ$ of this family is unirational, because this is true for 
the moduli space of ordinary GM fourfolds. 
By Theorem~\ref{theorem-moduli-space}\eqref{theorem-moduli-space-dim}, 
the fibers of $g \colon M^\circ(v) \to S^\circ$ are smooth polarized hyperk\"{a}hler varieties of 
dimension $(v,v) + 2 = 2(a^2+ b^2 + 1)$. 
The local completeness of this family follows by combining Lemma~\ref{lemma-periods}\eqref{p-dominant} and 
Theorem~\ref{theorem-moduli-space}\eqref{theorem-moduli-space-H2}. 

By Theorem \ref{theorem-moduli-space}\eqref{theorem-moduli-space-H2}, the polarization class $h$ on a fiber of $g$ is orthogonal to $v$ and, over a very general point, is a combination of $\lambda_1$ and $\lambda_2$. Thus we have $h=b \lambda_1-a\lambda_2$, which has degree $(h,h)=2(a^2+b^2)$. It remains to compute the divisibility of $h$ in $\rH^2(M,\bZ)$, where $M$ is a fiber of $g$, which is the positive generator $\gamma \in \bZ$ of $h \cdot \rH^2(M,\bZ)$. Consider the sequence
$$0 \to \rH^2(M,\bZ)_{\text{prim}} \oplus \bZ h \to \rH^2(M,\bZ) \to \frac{\bZ}{k\bZ} \to 0.$$
Since $\rH^2(M,\bZ)_{\text{prim}} \cong \langle \lambda_1, \lambda_2 \rangle^\perp$ in the Mukai lattice by Theorem \ref{theorem-moduli-space}\eqref{theorem-moduli-space-H2}, the discriminant group of $\rH^2(M,\bZ)_{\text{prim}}$ has order $4$. As the discriminant groups of $\bZ h$ and $\rH^2(M,\bZ)$ have order $2(a^2+b^2)$, it follows that $k=2$. As a consequence, there exists an element of the form $\lambda= \frac{1}{2}h + \frac{1}{2} \tau$, with $\tau \in \rH^2(M,\bZ)_{\text{prim}}$, such that $\lambda \in \rH^2(M,\bZ)$. Since $h^2=2(a^2+b^2)$ and $h \cdot \lambda=a^2 + b^2$, we deduce that $\gamma=a^2 + b^2$.
\end{proof} 

We end this section by observing an interesting property of the families constructed in Theorem \ref{theorem-unirational-families}. 

\begin{proposition}
\label{prop_existinvolution}
Let $M$ be a very general polarized hyperk\"ahler variety in a family as in Theorem \ref{theorem-unirational-families}. Then $\Aut(M)=\emph{Bir}(M) =\bZ /2\bZ$, where $\Aut(M)$ is the group of automorphisms of $M$ and $\emph{Bir}(M)$ is the group of birational automorphisms of $M$. The corresponding involution of $M$ is antisymplectic. 
\end{proposition}

\begin{proof}
By \cite[Proposition 4.3]{Deb:survey}, we just need to check that $-1$ is a square modulo $a^2 + b^2$. For this, note that $a$ and $b$ are coprime, hence they are invertible in the group $\bZ/(a^2 + b^2)\bZ$. So we have that $-1\equiv a^{-2}b^2 \pmod{a^2 + b^2}$.
\end{proof}

\subsection{Examples of Theorem~\ref{theorem-unirational-families} in low dimensions}
\label{subsection_examples}
It is interesting to interpret the hyperk\"{a}hler varieties from Theorem~\ref{theorem-unirational-families} in terms of the geometry of GM fourfolds and classically known hyperk\"{a}hler varieties, in analogy to \cite{LPZ, LPZ2} for cubic fourfolds. 
Here we sketch this relation for some small values of $a$ and $b$, leaving a more detailed treatment to future studies.

\subsubsection{Double EPW sextics ($a^2 + b^2 = 1$)} 
\label{section-EPW}
An EPW sextic is a special type of degree $6$ hypersurface $Y_A \subset \bP(V_6)$ 
constructed from a Lagrangian subspace $A \subset \wedge^3 V_6$, where $V_6$ is a 
$6$-dimensional vector space and $\wedge^3 V_6$ is equipped with the natural $\det(V_6)$-valued symplectic form. 
O'Grady~\cite{OG-EPW} showed that there is a canonical double cover $\wtilde{Y}_{A} \to Y_A$, called a \emph{double EPW sextic}, which when smooth (as is true for generic $A$) is a polarized hyperk\"{a}hler fourfold of degree $2$ and divisibility $1$, deformation equivalent to the Hilbert square of a K3 surface. 
The resulting family of polarized hyperk\"{a}hler varieties is locally complete. 
We note that there is also a natural duality operation on these hyperk\"{a}hlers: the orthogonal $A^{\perp} \subset \wedge^3 V_6^{\vee}$ similarly gives rise to the \emph{dual double EPW sextic} $\wtilde{Y}_{A^\perp} \to Y_{A^{\perp}}$. 

Iliev and Manivel \cite{IM-EPW}, and in a more general setting Debarre and Kuznetsov \cite{DebKuz:birGM}, showed  
that to any GM fourfold $X$ there is a naturally associated Lagrangian $A \subset \wedge^3 V_6$ as above, and that a generic $A \subset \wedge^3 V_6$ arises in this way.  
There is a close relationship between $X$ and the double EPW sextics associated to $A$. 
Indeed, the Hilbert scheme of conics on a general $X$ is birational to a 
$\bP^1$-bundle over $\wtilde{Y}_{A^\perp}$ \cite{IM-EPW}. 
From a derived categorical viewpoint, by \cite[Remark 2.5]{Pert} the Mukai vector of the projection into $\Ku(X)$ of a twist of the structure sheaf of a general conic in $X$ is $\lambda_1$. For this reason, we expect that there are isomorphisms $\wtilde{Y}_{A^\perp} \cong M_\sigma(\Ku(X),\lambda_1)$ and $\wtilde{Y}_{A} \cong M_\sigma(\Ku(X),\lambda_2)$ for generic $X$ and suitable $\sigma \in \Stab^{\dagger}(\Ku(X))$. 
Here we only prove the following weaker statement, based on an identification of period points from \cite{DebKuz:periodGM}. 

\begin{proposition}
\label{proposition-EPW}
If $X$ is a very general GM fourfold with associated Lagrangian $A \subset \wedge^3 V_6$ 
and $\sigma \in \Stab^{\dagger}(\Ku(X))$ is generic with respect to $\lambda_1$ and $\lambda_2$, 
then either $M_\sigma(\Ku(X),\lambda_1) \cong \wtilde{Y}_{A}$ or $M_\sigma(\Ku(X),\lambda_1) \cong \wtilde{Y}_{A^\perp}$. 
\end{proposition}

\begin{proof}
We assume the reader is familiar with the notation used in \cite[Section 5.4]{DebKuz:periodGM}. By Theorem \ref{theorem-moduli-space}\eqref{theorem-moduli-space-H2}, there is a Hodge isometry
\begin{equation*}
\rH^2(M_\sigma(\Ku(X),\lambda_1), \bZ) \cong \lambda_1^\perp;
\end{equation*} 
moreover, $\lambda_2$ is identified with the polarization class on $M_\sigma(\Ku(X),\lambda_1)$. In particular, we have a Hodge isometry
\begin{equation*}
\rH^2(M_\sigma(\Ku(X),\lambda_1), \bZ)_0 \cong \langle \lambda_1, \lambda_2 \rangle^\perp.
\end{equation*} 
Therefore, by Proposition~\ref{proposition-HKu-HX} we obtain a Hodge isometry 
\begin{equation*}
\rH^2(M_\sigma(\Ku(X),\lambda_1), \bZ)_0 \cong \rH^4(X,\bZ)_{0}(1).
\end{equation*} 
By \cite[Theorem 5.1, Remark 5.25]{DebKuz:periodGM}, we can assume that $X$ and $\wtilde{Y}_A$ have the same period point. Thus composing with the Hodge isometry $\rH^4(X, \bZ)_{0}(1) \cong \rH^2(\wtilde{Y}_A, \bZ)_0$, we get
$$f: \rH^2(M_\sigma(\Ku(X),\lambda_1), \bZ)_0 \cong \rH^2(\wtilde{Y}_A, \bZ)_0.$$

Set $\Lambda:=E_8(-1)^{\oplus 2} \oplus U^{\oplus 2} \oplus A_1(-1)^{\oplus 2}$ and fix two markings $\phi_1: \rH^2(M_\sigma(\Ku(X),\lambda_1), \bZ)_0 \cong \Lambda$ and $\phi_2: \rH^2(\wtilde{Y}_A, \bZ)_0 \cong \Lambda$. Consider the composition $g:= \phi_2 \circ f \circ \phi_1^{-1}$. If $g$ acts trivially on the discriminant group $d(\Lambda) \cong (\bZ / 2\bZ)^2$ of $\Lambda$, then $M_\sigma(\Ku(X),\lambda_1)$ and $\wtilde{Y}_A$ have the same period point. By Verbitsky's Torelli theorem \cite{verbitsky}, we deduce that $M_\sigma(\Ku(X),\lambda_1)$ and $\wtilde{Y}_A$ are birational. Moreover, since $X$ is very general, we deduce that they are isomorphic.

In the other case, $g$ acts on the discriminant group by exchanging the generators of the two copies of $\bZ /2\bZ$. Then $M_\sigma(\Ku(X),\lambda_1)$ has the same period point of $\wtilde{Y}_{A^\perp}$ by \cite{OG:dualEPW}. Arguing as before, we conclude that $M_\sigma(\Ku(X),\lambda_1) \cong \wtilde{Y}_{A^\perp}$. 
\end{proof}

\subsubsection{EPW cubes ($a^2 + b^2 = 2$)} 
Given a (suitably generic) 
Lagrangian subspace $A \subset \wedge^3V_6$ as in Section~\ref{section-EPW} above, 
Iliev, Kapustka, Kapustka, and Ranestad \cite{IKKR} constructed a polarized hyperk\"{a}hler sixfold $\wtilde{Z}_{A}$ as a double cover of an associated subvariety of the Grassmannian $\Gr(3,V_6)$. 
These hyperk\"{a}hlers, called \emph{EPW cubes}, are deformation equivalent to the Hilbert cube of a K3 surface, have degree $4$ and divisibility $2$, and form a locally complete family. 

If $X$ is a GM fourfold, then the moduli space $M_{\sigma}(\Ku(X), \pm (\lambda_1 \pm \lambda_2))$ is a hyperk\"{a}hler variety with the same numerical invariants as an EPW cube. 
If $X$ is generic with associated Lagrangian $A \subset \wedge^3V_6$, 
then we expect that $\wtilde{Z}_{A}$ can be realized as a moduli space $M_{\sigma}(\Ku(X), \pm (\lambda_1 \pm \lambda_2))$. 
We note that, if an identification of the period points of $X$ and $\wtilde{Z}_A$ were known, then 
this would follow in the very general case as in Proposition~\ref{proposition-EPW}. 

\subsubsection{Projections of points ($a^2 + b^2 = 5$)} 
The Mukai vector of the projection into $\Ku(X)$ of skyscraper sheaves of points in $X$ is $v = \lambda_1+2\lambda_2$. 
We expect these objects (at least for a generic point of $X$) are stable for the stability conditions $\sigma$ we have constructed, 
and give rise to a (possibly only rationally defined) embedding of $X$ into the hyperk\"ahler $12$-fold $M_\sigma(\Ku(X),v)$.


\providecommand{\bysame}{\leavevmode\hbox to3em{\hrulefill}\thinspace}
\providecommand{\MR}{\relax\ifhmode\unskip\space\fi MR }
\providecommand{\MRhref}[2]{%
  \href{http://www.ams.org/mathscinet-getitem?mr=#1}{#2}
}
\providecommand{\href}[2]{#2}


\end{document}